\documentclass[11pt,a4paper]{amsart}
\usepackage[T1]{fontenc}
\usepackage[utf8]{inputenc}
\usepackage{newtxtext}
\usepackage{amsmath,amssymb,mathtools}
\usepackage{newtxmath}
\usepackage[margin=26mm,headheight=14pt]{geometry}
\usepackage{microtype}
\usepackage{enumitem}
\usepackage{needspace}
\usepackage[hidelinks,unicode]{hyperref}
\usepackage{bookmark}
\setlist[enumerate]{label=\textup{(\roman*)},leftmargin=2.3em,itemsep=3pt,topsep=5pt}
\numberwithin{equation}{section}
\newtheorem{theorem}{Theorem}[section]
\newtheorem{lemma}[theorem]{Lemma}
\newtheorem{proposition}[theorem]{Proposition}
\newtheorem{corollary}[theorem]{Corollary}
\newtheorem{result}[theorem]{Result}
\theoremstyle{definition}
\newtheorem{definition}[theorem]{Definition}
\theoremstyle{remark}
\newtheorem{remark}[theorem]{Remark}
\newcommand{\C}{\mathbb C}
\newcommand{\R}{\mathbb R}
\newcommand{\D}{\mathbb D}
\newcommand{\B}{\mathbb B}
\newcommand{\ddbar}{\bar\partial}
\newcommand{\eps}{\varepsilon}
\newcommand{\norm}[1]{\lVert#1\rVert}
\newcommand{\abs}[1]{\lvert#1\rvert}
\newcommand{\ip}[2]{\langle#1,#2\rangle}
\newcommand{\qone}{q_\Omega^{(1)}}
\newcommand{\qk}{q_\Omega^{(k)}}
\newcommand{\dk}{d_\Omega^{(k)}}
\newcommand{\Ht}{H_t^{(k)}}
\newcommand{\Hs}{H_s^{(k)}}
\DeclareMathOperator{\dist}{dist}
\DeclareMathOperator{\diam}{diam}

\DeclareMathOperator{\Vol}{Vol}
\DeclareMathOperator{\ord}{ord}
\DeclareMathOperator{\Gr}{Gr}
\DeclareMathOperator{\Ree}{Re}
\DeclareMathOperator{\Imm}{Im}
\DeclareMathOperator{\spanC}{span_{\C}}

\newcommand{\HH}{\mathbb H}
\newcommand{\dom}{\operatorname{Dom}}
\newcommand{\OO}{\mathcal O}
\newcommand{\RR}{\mathcal R}
\hypersetup{pdftitle={Gromov hyperbolicity, finite type, and subellipticity on C-convex domains},pdfauthor={Amar Deep Sarkar}}
\title[Gromov hyperbolicity, finite type and subellipticity]{Gromov hyperbolicity, finite type, and subellipticity\\ on $\C$-convex domains}
\author{Amar Deep Sarkar}
\address{School of Basic Sciences, Indian Institute of Technology Bhubaneswar, Odisha, India}
\email{amar@iitbbs.ac.in}
\date{}
\begin{document}
\begin{abstract}
In this article, we study three properties of a bounded $\C$-convex domain $\Omega\subset\C^n$, $n\ge2$, with smooth boundary, namely Gromov hyperbolicity of its Kobayashi distance, finiteness of the D'Angelo type of $\partial\Omega$, and the existence of a subelliptic estimate for the $\ddbar$-Neumann problem on $(0,1)$-forms. We prove that these three properties are equivalent. The subelliptic estimates are obtained without Catlin's construction of plurisubharmonic weights. Our weights are explicit bounded transforms of the logarithm of the Bergman kernel. A single quantity controls all three properties, namely the radius of the largest disc through a point of $\Omega$ in a given complex direction. Gromov hyperbolicity forces a power bound for these radii. Finite type gives the power $1/M$, where $M$ is the maximal type, and this power cannot be improved. A subelliptic estimate gives a power bound through the canonical solution operator applied to normalized Bergman kernels.

The analytic estimates need no boundary regularity. On a bounded pseudoconvex domain, a bounded weight whose induced Hessian on $(0,q)$-forms is bounded below by a negative power of the boundary distance gives a Sobolev estimate for the extension by zero of $(0,q)$-forms, and also eigenvalue bounds for the $\ddbar$-Neumann operator. On smooth $\C$-convex domains of finite type, we show that the directional expansion exponent, the normal expansion exponent and the supremum of the subelliptic gains are all equal to $1/M$. Next, we prove that on a Lipschitz $\C$-convex domain whose Kobayashi distance is Gromov hyperbolic, the Euclidean boundary and the Gromov boundary are bi-H\"older equivalent. Finally, we give a Hardy-type criterion for Gromov hyperbolicity of collar metrics.
\end{abstract}
\maketitle
\tableofcontents

\section{Introduction}\label{sec:introduction}

Let $\Omega\subset\C^n$ be a bounded domain with smooth boundary. Three properties of $\Omega$ have been studied extensively. The first is the large-scale geometry of its Kobayashi distance. The second is the order of contact of complex curves with $\partial\Omega$. The third is the regularity of the $\ddbar$-Neumann problem on $\Omega$. For pseudoconvex domains, the last two properties are linked by the theorems of Catlin \cite{Cat83,Cat87}, which say that a subelliptic estimate holds if and only if $\partial\Omega$ is of finite type in the sense of D'Angelo. For convex domains, the first two properties are linked by the theorem of Zimmer \cite{Zim16}, which says that the Kobayashi distance is Gromov hyperbolic if and only if $\partial\Omega$ is of finite type. In this article, we show that the three properties are equivalent on $\C$-convex domains. We also show that one geometric quantity, the radius of the largest disc through a point in a given complex direction, controls all three of them.

Our starting point is a remark of Zimmer. The sufficiency half of Catlin's characterization rests on bounded plurisubharmonic functions with large complex Hessian. By Catlin's criterion \cite{Cat87}, which Straube \cite{Str97} extended to Lipschitz domains, a bounded plurisubharmonic function whose complex Hessian is at least $c\,\delta^{-2\varepsilon}$ near the boundary gives a subelliptic estimate of gain $\varepsilon$. Zimmer constructed such functions on bounded convex domains whose K\"ahler--Einstein distance is Gromov hyperbolic, and he remarked in \cite[Section~1]{Zim22} that Catlin's construction of these functions in the finite type case is very involved, so that an alternative way of constructing them would be desirable. On $\C$-convex domains we give such an alternative. Our weights are explicit. They have the form $-(B+\log K_\Omega)^{-a}$, where $K_\Omega$ is the Bergman kernel on the diagonal, $B$ is a constant and $a>0$. Their complex Hessians are controlled by the Bergman metric, and hence by the directional radii of $\Omega$. Finite type, and also Gromov hyperbolicity, enter the argument only through a power bound for these radii.

\subsection{The main theorem}
We first recall the notions that appear in the statement of our main theorem. A domain $\Omega\subset\C^n$ is \emph{$\C$-convex} if its intersection with every complex affine line is empty or connected and simply connected. Every convex domain is $\C$-convex, and the class of $\C$-convex domains is strictly larger. We refer the reader to \cite{APS04,Hor94}. We write $k_\Omega$ for the Kobayashi distance and $\kappa_\Omega$ for the Kobayashi--Royden metric, normalized by $\kappa_\D(0;v)=|v|$. Thus
\[
 \kappa_\Omega(z;v)=\inf\{|a|:f\in\OO(\D,\Omega),\ f(0)=z,\ af'(0)=v\},
\]
where $\D\subset\C$ is the unit disc and $\OO$ denotes holomorphic maps or functions. The distance $k_\Omega$ is the integrated distance of $\kappa_\Omega$ (see \cite{Roy71,Ven89}). A geodesic metric space is \emph{Gromov hyperbolic} if there is $\delta_0\ge0$ such that every side of every geodesic triangle lies in the closed $\delta_0$-neighborhood of the other two sides (see \cite{BH99,Gro87}). If $\Omega$ is a bounded $\C$-convex domain, then $(\Omega,k_\Omega)$ is a proper geodesic metric space (Lemma~\ref{lem:proper}).

A smooth defining function near $p\in\partial\Omega$ is a real $C^\infty$ function $\rho$ with $d\rho\ne0$ on the boundary and $\Omega=\{\rho<0\}$ locally. The \emph{D'Angelo one-type} of $\partial\Omega$ at $p$ is
\[
 \Delta_1(\partial\Omega,p)
 =\sup_\gamma\frac{\ord_0(\rho\circ\gamma)}{\ord_0(\gamma-p)},
\]
where $\gamma$ ranges over all nonconstant holomorphic curve germs through $p$, singular germs included, and where the order of $\rho\circ\gamma$ is taken in the real variables. This quotient does not depend on the choice of the defining function \cite{DAn82}. We say that $\partial\Omega$ is of \emph{finite type} if $\Delta_1(\partial\Omega,p)<\infty$ for every $p\in\partial\Omega$.

Let $D$ be a bounded domain. We let $\ddbar$ be the maximal distributional operator on square-integrable forms, we let $\ddbar^*$ be its Hilbert-space adjoint, and we put
\begin{equation}\label{eq:Vq}
 \mathcal V_q(D)=\dom\ddbar_q\cap\dom\ddbar^*_{q-1},\qquad
 Q_D(u)=\|\ddbar u\|_2^2+\|\ddbar^*u\|_2^2.
\end{equation}
The space $H^s(D)$ is the space of restrictions to $D$ of functions in $H^s(\R^{2n})$, with the quotient norm. For forms the norm is taken coefficientwise. A \emph{subelliptic estimate of gain $s>0$ in degree $q$} is an inequality
\begin{equation}\label{eq:subelliptic-estimate}
 \|u\|_{H^s(D)}^2\le C\bigl(Q_D(u)+\|u\|_2^2\bigr),
 \qquad u\in\mathcal V_q(D).
\end{equation}
None of these definitions uses any regularity of $\partial D$.

We now present our main result. It shows that the three properties described above are equivalent on smoothly bounded $\C$-convex domains, and it gives the range of subelliptic gains in every degree.

\Needspace{12\baselineskip}
\begin{theorem}[Gromov hyperbolicity, finite type, and subellipticity]\label{thm:main}
Let $\Omega\Subset\C^n$, $n\ge2$, be $\C$-convex with $C^\infty$ boundary. The following conditions are equivalent.
\begin{enumerate}
\item $(\Omega,k_\Omega)$ is Gromov hyperbolic.
\item $\partial\Omega$ is of finite type.
\item A subelliptic estimate \eqref{eq:subelliptic-estimate} of some gain $s>0$ holds in degree $q=1$.
\end{enumerate}
If these conditions hold, then for every $1\le q\le n$ a subelliptic estimate of every gain $0<s<1/M$ holds in degree $q$, where $M=\max_{p\in\partial\Omega}\Delta_1(\partial\Omega,p)$.
\end{theorem}

\begin{remark}\label{rem:intro-new}
For convex domains, the equivalence of (i) and (ii) is due to Zimmer \cite[Theorem~1.1]{Zim16}. A second proof of it was given by Wang and Zimmer \cite[Corollary~1.7]{WZ24}. For $n=2$, the implication (ii)$\Rightarrow$(i) holds on every smoothly bounded pseudoconvex domain by the work of Fiacchi \cite{Fia22}. We also refer the reader to Li--Pu--Wang \cite{LPW24}. For bounded $\C$-convex domains, Zimmer \cite[Theorem~1.4]{Zim17} showed that Gromov hyperbolicity excludes nonconstant holomorphic discs in a $C^1$ boundary, and Bracci--Gaussier--Nikolov--Thomas \cite[Proposition~4.2]{BGNT24} showed that it excludes nontrivial affine discs in a Lipschitz boundary. Neither of these two statements gives finite type. Before the present article, for $\C$-convex domains that are not convex, the implication (ii)$\Rightarrow$(i) was known for $n=2$ \cite{Fia22}, and both implications were known for domains biholomorphic, smoothly up to the boundary, to convex domains, since Zimmer's theorem transfers to such domains. The equivalence of (ii) and (iii) holds on every smoothly bounded pseudoconvex domain by Catlin's theorems \cite{Cat83,Cat87}. Our proofs of (ii)$\Rightarrow$(iii) and of (iii)$\Rightarrow$(ii) use neither Catlin's construction of weights nor his necessary conditions. We say more about this in Subsection~\ref{sec:intro-prior}.
\end{remark}

\begin{remark}
Condition (iii) cannot be replaced by an estimate in top degree. By Proposition~\ref{prop:top-degree}, a first-order estimate holds in degree $q=n$ on every bounded pseudoconvex domain. In particular, such an estimate holds on every smoothly bounded convex domain of infinite type.
\end{remark}

\subsection{The directional radius}\label{sec:intro-radius}
We now introduce the quantity through which the three conditions of Theorem~\ref{thm:main} are compared. For a domain $D\subset\C^n$, a point $z\in D$ and a vector $v\ne0$, put
\begin{equation}\label{eq:radius}
 r_D(z;v)=\sup\{r>0:z+\zeta v/|v|\in D\text{ for every }|\zeta|<r\}.
\end{equation}
Thus $r_D(z;v)$ is the Euclidean distance from $z$ to the boundary of the slice of $D$ by the complex line through $z$ in the direction $v$. We write $\delta_D(z)=\dist(z,\partial D)$ and $\RR_D(z)=\sup_{|v|=1}r_D(z;v)$. A bounded $\C$-convex domain $D$ is \emph{$m$-$\C$-convex}, for a real number $m>0$, if there are $C_0,t_0>0$ such that
\begin{equation}\label{eq:mCconvex}
 r_D(z;v)\le C_0\delta_D(z)^{1/m},\qquad \delta_D(z)<t_0,\quad v\ne0.
\end{equation}
On a bounded domain this estimate in a collar is equivalent to the same estimate on the whole domain, with a larger constant. For convex domains, \eqref{eq:mCconvex} is the $m$-convexity introduced by Mercer \cite{Mer93}, which plays a central role in the work of Zimmer \cite{Zim16,Zim22}. Every $m$-$\C$-convex domain in $\C^n$, $n\ge2$, has $m\ge2$ (Proposition~\ref{prop:m-at-least-2}).

We now describe how the proof of Theorem~\ref{thm:main} goes through the directional radius. Each of the four implications below is a statement about the directional radius in every direction.

The implication (i)$\Rightarrow$(ii) is proved in two steps. On Lipschitz $\C$-convex domains, Gromov hyperbolicity gives a power bound $r_\Omega(z;v)\le C\delta_\Omega(z)^\alpha$ (Theorem~\ref{thm:lipschitz}). On smooth $\C$-convex domains, such a power bound gives $\Delta_1\le1/\alpha$ (Lemma~\ref{lem:full-curve-type}).

For the implication (ii)$\Rightarrow$(i), we show that finite type gives a contraction of the directional radii between two normal depths with exponent $1/M_1(\Omega)$, where $M_1(\Omega)$ is the maximal order of contact of complex lines with $\partial\Omega$ (Proposition~\ref{prop:two-depth}). This exponent is sharp (Theorem~\ref{thm:planes}), and $M_1(\Omega)=M$ (Corollary~\ref{cor:finite-gains}). By the expanding-collar theorem of Wang and Zimmer \cite{WZ24}, the integrated distance of the gauge $|v|/r_\Omega(z;v)$ is then Gromov hyperbolic (Theorem~\ref{thm:planes}). On $\C$-convex domains this distance is bi-Lipschitz to $k_\Omega$ by the estimates of Nikolov, Pflug and Zwonek \cite{NPZ11}.

For the implication (ii)$\Rightarrow$(iii), the power bound $r_\Omega\le C\delta_\Omega^{1/M}$ (Corollary~\ref{cor:finite-gains}) and the minimal-basis estimates of \cite{NPZ11} bound the complex Hessian of $\log K_\Omega$, which is the Bergman metric, from below by $c\,\delta_\Omega^{-2/M}\,\mathrm{Id}$. Since $\log K_\Omega$ has self-bounded gradient \cite{Zim21}, a bounded transform of it has complex Hessian at least $c\,\delta_\Omega^{-2/M}\ell(\delta_\Omega)^{-1-a}$. Theorem~\ref{thm:zero-extension} below turns this weight into estimates of every gain below $1/M$.

For the implication (iii)$\Rightarrow$(ii), we apply the canonical solution operator to the form $f_z\,d\bar w_n$, where $f_z$ is the normalized Bergman kernel at $z$ and $w_n$ is the coordinate of the longest axis of a minimal basis at $z$. The regularity theorem of Kohn and Nirenberg and a fractional Hardy inequality bound the $L^2$ mass of the solution on a polydisc centered at $z$ from above by $C\delta_\Omega(z)^{2s}$, and the mean-value property bounds it from below. A comparison of the two bounds gives $r_\Omega\le C\delta_\Omega^s$ (Theorem~\ref{thm:analytic-necessity}), and Lemma~\ref{lem:full-curve-type} then gives $\Delta_1\le1/s$.

\subsection{Estimates without boundary regularity}\label{sec:intro-estimates}
The analytic implication (ii)$\Rightarrow$(iii) rests on an estimate that needs no regularity of the boundary at all. We first fix the notation for it. A domain is \emph{pseudoconvex} if it admits a continuous plurisubharmonic exhaustion. Write $\ell(t)=\log(e/t)$ for $0<t\le1$, and let $E_0u$ denote the coefficientwise extension of $u$ by zero. For $s>0$ and $b\ge0$ set
\begin{equation}\label{eq:Hsb}
 \|f\|_{H^{s,-b}(\R^d)}^2
 =\int_{\R^d}\frac{(1+|\xi|^2)^s}{[\log(e+|\xi|)]^{2b}}|\widehat f(\xi)|^2\,d\xi,
\end{equation}
with the unitary Fourier transform. Thus $H^{s,0}=H^s$. Since $(1+|\xi|^2)^s\le C(1+|\xi|^2)^{\alpha}[\log(e+|\xi|)]^{-2b}$ for $0<s<\alpha$, an estimate for $\|E_0u\|_{H^{\alpha,-b}}$ implies the corresponding estimate for $\|E_0u\|_{H^s}$ for every $0<s<\alpha$. Since also $\|u\|_{H^s(D)}\le\|E_0u\|_{H^s(\R^{2n})}$, it implies a subelliptic estimate of every gain $0<s<\alpha$. We shall use this observation throughout without further comment.

For a real $C^2$ function $\lambda$, the induced Hessian on $(0,q)$-forms is
\begin{equation}\label{eq:induced-hessian}
 \mathcal L_q^\lambda(u)
 =\sideset{}{'}\sum_{|K|=q-1}\sum_{j,k=1}^n
     \lambda_{j\bar k}u_{jK}\overline{u_{kK}},
\end{equation}
where the prime indicates summation over increasing multi-indices and where the coefficients are extended antisymmetrically. We write $H^{[q]}_\lambda$ for the associated Hermitian endomorphism. If the complex Hessian $(\lambda_{j\bar k})$ has eigenvalues $\mu_1\le\cdots\le\mu_n$, then the least eigenvalue of $H^{[q]}_\lambda$ is $\mu_1+\cdots+\mu_q$. For a bounded real $\lambda$, the inequality $H^{[q]}_\lambda\ge A$ in distributions means that, for every constant form $\eta$ and every nonnegative $\chi\in C_c^\infty(D)$, the distribution $\langle(H^{[q]}_\lambda-A)\eta,\eta\rangle$ has nonnegative pairing with $\chi$.

We can now state the estimate.

\begin{theorem}[Zero-extension estimates from bounded weights]\label{thm:zero-extension}
Let $D\Subset\C^n$ be pseudoconvex, $1\le q\le n$, $0<s<1$, $b\ge 0$, $c>0$ and $0<t_0<1$. Suppose that $\lambda\in L^\infty(D;\R)$ satisfies $H^{[q]}_\lambda\ge0$ in distributions on $D$ and
\begin{equation}\label{eq:weight-lower}
H^{[q]}_\lambda\ge c\,\delta_D^{-2s}\ell(\delta_D)^{-2b}\,\mathrm{Id}\quad\text{ in distributions on }\{\delta_D<t_0\}.
\end{equation}
Then for every $q\le q'\le n$ there is $C>0$, depending only on $n$, $s$, $b$, $c$, $t_0$, $\operatorname{osc}_D\lambda$ and $\diam D$, such that
\begin{equation}\label{eq:zero-extension}
\norm{E_0u}^2_{H^{s,-b}(\R^{2n})}\le CQ_D(u),\qquad u\in\mathcal V_{q'}(D).
\end{equation}
\end{theorem}

For $b=0$ and smoothly bounded domains, the existence of such a weight gives a subelliptic estimate of gain $s$ by Catlin's criterion \cite{Cat87}. Straube \cite[Theorem~1]{Str97} proved the local version of this criterion on Lipschitz domains for $s<1/2$. Theorem~\ref{thm:zero-extension} extends Straube's theorem in three directions. It holds on arbitrary bounded pseudoconvex domains, it allows logarithmic weights, and it has the zero-extension norm on the left with a hypothesis on the induced $q$-Hessian only. The proof follows the proof of Straube. It consists of an exhaustion argument that corrects the forms on smooth subdomains, a weighted $L^2$ estimate on the full form domain (Theorem~\ref{thm:matrix}), and a reduction to harmonic coefficients. The new ingredient is Lemma~\ref{lem:harmonic}, an estimate for harmonic functions on arbitrary bounded open sets. This lemma carries the logarithmic factor, and it replaces the estimate of Jerison and Kenig \cite[Theorem~4.2]{JK95} which Straube used.

The norm of the zero extension always dominates the restriction norm. On a bounded Lipschitz domain the two norms are equivalent for $0<s<1/2$, see \cite[Theorems~3.33 and~3.40]{McL00}, \cite[Lemma~2.1]{HM17} and \cite{CHM17}, but they are not equivalent for $s=1/2$, even on bounded $C^\infty$ domains \cite[Remark~2.5]{HM17}. Accordingly, Proposition~\ref{prop:half-order} shows that the condition $b>1/2$ in Theorem~\ref{thm:directional} below cannot be removed from a universal half-order zero-extension estimate in degrees $q<n$. In top degree, on the other hand, Proposition~\ref{prop:top-degree} gives a first-order estimate on every bounded pseudoconvex domain.

The nonnegative self-adjoint operator associated with the closed form $Q_D$ is the $\ddbar$-Neumann Laplacian $\Box_q$. On a bounded pseudoconvex domain it has a bounded inverse $N_q$, the $\ddbar$-Neumann operator (Lemma~\ref{lem:uniform-L2}). For a compact positive operator $T$, we write $\mu_j(T)$ for its eigenvalues in decreasing order. The Schatten class $\mathfrak S_p$, $p>0$, consists of the compact operators whose singular values $s_j(T)$ satisfy $\sum_j s_j(T)^p<\infty$.

Our next result applies Theorem~\ref{thm:zero-extension} to the weights described above. On every $m$-$\C$-convex domain, it gives the logarithmic endpoint estimate at the order $1/m$, together with the eigenvalue bounds that follow from it.

\begin{theorem}[Directional power geometry]\label{thm:directional}
Let $D\Subset\C^n$, $n\ge 2$, be $m$-$\C$-convex. For every $b>1/2$ there is $C_b>0$, depending only on $D$, $m$, $b$ and $n$, such that for every $1\le q\le n$
\begin{equation}\label{eq:directional-estimate}
\norm{E_0u}^2_{H^{1/m,-b}(\R^{2n})}\le C_bQ_D(u),\qquad u\in\mathcal V_q(D).
\end{equation}
Moreover, $N_q$ is compact and its eigenvalues satisfy
\begin{equation}\label{eq:eigenvalues}
\mu_j(N_q)\le C_bj^{-1/(nm)}[\log(e+j)]^{2b},\qquad j\ge1.
\end{equation}
In particular, a subelliptic estimate of every gain $0<s<1/m$ holds in every degree, and $N_q\in\mathfrak S_p$ for every $p>nm$.
\end{theorem}

Theorem~\ref{thm:directional} assumes neither boundary regularity nor Gromov hyperbolicity. The lens of Remark~\ref{rem:lens} is $2$-$\C$-convex with Lipschitz boundary, but its Kobayashi distance is not Gromov hyperbolic. Within the class of bounded convex domains with Gromov hyperbolic Kobayashi distance, Theorem~\ref{thm:directional} gives every gain below $1/m$ for every $m$ for which the domain is $m$-convex. Zimmer's theorem \cite[Theorem~16.1]{Zim22} gives the gains below $1/(\alpha m)$ with an additional parameter $\alpha\ge1$. We compare the two results in Section~\ref{sec:zimmer-comparison}.

\subsection{Lipschitz domains and the Gromov boundary}
We now turn to the geometric implication (i)$\Rightarrow$(ii). On Lipschitz $\C$-convex domains, Gromov hyperbolicity forces a power bound of the form \eqref{eq:mCconvex}. To state the result, we need some terminology. A complex affine hyperplane through a boundary point is a \emph{complex support} if it does not meet the domain. In a Lipschitz boundary chart, the domain is the epigraph of a Lipschitz function in a unit real direction $e$. An \emph{admissible frame} $(p,e,p+V_p,v)$ consists of a boundary point $p$ in such a chart, its inward chart direction $e$, a complex support $p+V_p$ at $p$, and a unit vector $v\in V_p$. The precise conventions are given in Section~\ref{sec:lipschitz}.

\begin{theorem}[Directional power law on Lipschitz domains]\label{thm:lipschitz}
Let $\Omega\Subset\C^n$ be $\C$-convex with Lipschitz boundary, suppose that $(\Omega,k_\Omega)$ is Gromov hyperbolic, and fix a finite Lipschitz boundary atlas. There are $C>0$, $A\ge1$, $t_1>0$ and $0<\alpha<1/2$, depending only on $\Omega$ and the atlas, such that
\[
 r_\Omega(z;v)\le C\delta_\Omega(z)^\alpha,\qquad z\in\Omega,\quad v\ne0,
\]
and, for every admissible frame $(p,e,p+V_p,v)$ of this atlas,
\begin{equation}\label{eq:lipschitz-relative-intro}
 r_\Omega(p+se;v)\le A(s/t)^\alpha r_\Omega(p+te;v),
 \qquad 0<s\le t\le t_1.
\end{equation}
\end{theorem}

Theorem~\ref{thm:lipschitz} extends the directional power bound of Zimmer \cite[Corollary~7.2]{Zim22} and the two-depth estimate of Wang and Zimmer \cite[Theorem~1.5]{WZ24} (for $k=1$) from convex domains to Lipschitz $\C$-convex domains. Combined with Theorem~\ref{thm:directional}, it gives subelliptic estimates of every gain below $\alpha$ and $N_q\in\mathfrak S_p$ for $p>n/\alpha$ (Corollary~\ref{cor:lipschitz-subelliptic}). The Lipschitz hypothesis cannot be dropped. A bounded projective image of Zimmer's cone \cite{Zim17} is a $\C$-convex domain with Gromov hyperbolic Kobayashi distance on which no power bound holds (Remark~\ref{rem:projective-cone}).

The power bound also controls the correspondence between the Euclidean boundary and the Gromov boundary. Let $\Omega$ be a bounded Lipschitz $\C$-convex domain such that $(\Omega,k_\Omega)$ is Gromov hyperbolic, and let $m$ be as in \eqref{eq:mCconvex}. Fix $o\in\Omega$ and $R>\max\{1,3\diam\Omega\}$. By \cite[Lemma~2.3]{BZ17}, there are $A\ge0$ and $B>0$ with $k_\Omega(o,z)\le A+B\log(R/\delta_\Omega(z))$. Theorem~\ref{thm:boundary-comparison} gives
\[
 \frac14\log\frac{R}{\norm{x-y}+\delta_\Omega(x)+\delta_\Omega(y)}-C
 \le (x\mid y)_o
 \le Bm\log\frac{R}{\norm{x-y}+\delta_\Omega(x)+\delta_\Omega(y)}+C,
\]
where $(x\mid y)_o$ is the Gromov product. Consequently, the identity map extends to a homeomorphism between the Euclidean and the Gromov compactifications, and its restriction to the boundaries is bi-H\"older for every visual metric. The qualitative identification of the two compactifications is due to Bracci, Gaussier, Nikolov and Thomas \cite[Proposition~4.4]{BGNT24}. Bi-H\"older comparisons of this kind were obtained by Liu, Wang and Zhou for convex Dini-smooth domains \cite[Proposition~1.8]{LWZ22} and by Liu, Pu and Wang for smooth pseudoconvex domains of finite type in $\C^2$ \cite[Theorems~1.1 and~1.2]{LPW23}. On smooth $\C$-convex domains of finite type we give the exponents in terms of the maximal type (Corollary~\ref{cor:smooth}).

\subsection{Critical exponents}\label{sec:critical-intro}
Let $\Omega$ be as in Theorem~\ref{thm:main} and of finite type, and put $M=\max_{p\in\partial\Omega}\Delta_1(\partial\Omega,p)$. Three exponents can be attached to $\Omega$. The first is the supremum of the $\alpha$ with $r_\Omega\le C\delta_\Omega^\alpha$. The second is the supremum of the rates $\lambda$ at which the directional gauge expands along inward normals. The third is the supremum of the subelliptic gains in degree one. Corollary~\ref{cor:critical} shows that all three are equal to $1/M$, and that the first two suprema are attained. That no gain exceeds $1/M$ also follows from Catlin's necessary condition \cite{Cat83}. A subelliptic estimate of gain exactly $1/M$ is known in two cases. For smoothly bounded convex domains of finite type it follows from McNeal's weights \cite{McN92} and Catlin's criterion \cite{Cat87}. For smoothly bounded pseudoconvex domains of finite type in $\C^2$ it follows from the work of Kohn \cite{Koh72} together with that of Rothschild and Stein \cite{RS76}, and the fact that $1/M$ cannot be exceeded in $\C^2$ is due to Greiner \cite{Gre74}.

\subsection{Further results}
The geometric part of the article treats the generalized complex quasihyperbolic gauges of Wang and Zimmer \cite{WZ24} on domains with global complex supports. Theorem~\ref{thm:planes} characterizes Gromov hyperbolicity of the generalized $k$-dimensional distances by finite contact with affine $k$-planes, and it identifies the sharp normal expansion exponent. For convex domains this characterization is \cite[Theorem~1.6]{WZ24}. The Hardy collar criterion (Theorem~\ref{thm:hardy}) gives Gromov hyperbolicity of a collar metric from an integral condition on the transverse gauge, and it needs no convexity. Theorem~\ref{thm:curve-type} shows that this criterion, together with a comparison on Kobayashi balls of fixed radius, bounds the D'Angelo type.

On the analytic side we prove a criterion that depends on the form degree (Theorem~\ref{thm:partial-collapse}) and spectral bounds (Theorem~\ref{thm:spectral}). Condition~$R$ and the smooth extension of biholomorphisms on smooth $\C$-convex domains of finite type follow from our estimates in the classical way (Remark~\ref{rem:condition-R}).

\subsection{Relation to earlier work}\label{sec:intro-prior}
Balogh and Bonk \cite{BB00} proved that the Kobayashi distance of a bounded strongly pseudoconvex domain with $C^2$ boundary is Gromov hyperbolic. Besides the works cited in Remark~\ref{rem:intro-new}, we refer the reader to Gaussier--Seshadri \cite{GS18}, Nikolov--Thomas--Trybu\l a \cite{NTT16}, Andreev--Nikolov--Trybu\l a \cite{ANT19} and to the survey of Zimmer \cite{Zim17s}. Zhang \cite{Zha22} proved Gromov hyperbolicity for bounded pseudoconvex domains of finite type whose Levi form has corank one at every boundary point. Nikolov and Trybu\l a \cite{NT18} showed that removing a complex affine hyperplane from a bounded $\C$-convex domain destroys Gromov hyperbolicity. Our geometric arguments use the expanding-collar method of Wang and Zimmer \cite{WZ24}, with global complex supports in place of real supporting hyperplanes, and they use the results of Bharali--Zimmer \cite{BZ17} and Bracci--Gaussier--Nikolov--Thomas \cite{BGNT24} on visibility and the Gromov boundary. For related results on the boundary extension of isometries we refer the reader to \cite{BGZ21,BNT22,CMS24}. The local dependence of the hyperbolicity constants on the type of a convex domain is studied in \cite{LQWZ26}.

The basic estimates for invariant metrics and minimal bases on $\C$-convex domains are due to Nikolov, Pflug and Zwonek \cite{NPZ11}. Their convex predecessors include McNeal \cite{McN94} and Nikolov--Pflug \cite{NP03}, and the relation between different extremal bases is discussed in \cite{NPT13}. For bounded domains with $C^1$ boundary, $\C$-convexity is the same as lineal convexity \cite{APS04}, and there is a substantial theory of lineally convex domains of finite type. Extremal bases and the multitype were introduced by Conrad \cite{Con02}, holomorphic support functions were constructed by Diederich and Forn\ae ss \cite{DF03}, H\"older estimates for $\ddbar$ were proved by Diederich and Fischer \cite{DFi06}, and adapted plurisubharmonic functions on geometrically separated domains were constructed by Charpentier and Dupain \cite{CD14}. These works construct weights and kernels adapted to the anisotropic geometry of the boundary. Our weight is different in nature. It is a single explicit transform of $\log K_\Omega$, and its Hessian bound uses only the minimal-basis estimates of \cite{NPZ11} and the self-bounded gradient of $\log K_\Omega$ proved by Zimmer \cite[Proposition~4.12]{Zim21}.

For the theory of subelliptic estimates we also refer the reader to Catlin--D'Angelo \cite{CD10}, to Kohn's multiplier method \cite{Koh79} and to McNeal's work on convex domains \cite{McN92,McN02b}. In our proof of (ii)$\Rightarrow$(iii), Catlin's construction of weights is replaced by the transform of $\log K_\Omega$ described above. The transform of a potential with self-bounded gradient goes back to McNeal \cite{McN02a}, and estimates with general and logarithmic gains from such weights were developed by Khanh and Zampieri \cite{KZ10,KZ12}. The passage from the weight to the estimate is Theorem~\ref{thm:zero-extension}, which extends Straube's theorem \cite{Str97} and is proved by his method. In our proof of (iii)$\Rightarrow$(ii), Catlin's necessary conditions are replaced by the regularity theorem of Kohn and Nirenberg \cite{KN65} and by canonical solutions with Bergman-kernel data. The same conclusion can be reached through McNeal's lower bound for the Bergman metric \cite{McN92a}, see also Khanh--Zampieri \cite{KZ12}. Sobolev estimates on pseudoconvex domains without boundary regularity were obtained earlier by Berndtsson and Charpentier \cite{BC00} for the Bergman projection.

The weighted identity is due to Morrey \cite{Mor58} and Kohn \cite{Koh63}, in the weighted form of H\"ormander \cite{Hor65}. We also refer the reader to Folland--Kohn \cite{FK72} and McNeal--Varolin \cite{MV15}. Estimates on nonsmooth strictly pseudoconvex domains were obtained by Henkin--Iordan--Kohn \cite{HIK96} and Michel--Shaw \cite{MS98,MS01}, and bounded plurisubharmonic functions on such domains were studied by Harrington \cite{Har07} and Harrington--Shaw \cite{HS07}. Degree-dependent positivity and self-bounded gradients appear in \cite{Her07,KZ11,McN02a}. For compactness of $N_q$ without any rate we refer the reader to Zimmer's metric characterization \cite{Zim21}, which extends the theorem of Fu and Straube on convex domains \cite{FS98}, and to Boas \cite{Boa88} for small exceptional sets of infinite type. For spectral properties of the $\ddbar$-Neumann Laplacian we refer the reader to Fu \cite{Fu08} and to G\"o\u{g}\"u\c{s}--\c{S}ahuto\u{g}lu \cite{GoS18}.

\subsection{Plan of the paper}
We now present the plan of this article. Part~\ref{part:geometry}, which consists of Sections~\ref{sec:gauges} to~\ref{sec:comparison}, develops the directional geometry. There we treat complex supports and the generalized gauges, the Hardy collar criterion, finite contact with affine planes, the power law on Lipschitz domains, and the comparison of the Euclidean and Gromov boundaries. Part~\ref{part:analysis}, which consists of Sections~\ref{sec:prelim} to~\ref{sec:zimmer-comparison}, proves Theorem~\ref{thm:zero-extension}, the spectral bounds, the Bergman-kernel weights and Theorem~\ref{thm:directional}. Part~\ref{part:three}, which consists of Section~\ref{sec:necessity}, proves the converse analytic implication, completes the proof of Theorem~\ref{thm:main} and identifies the critical exponents.

\part{Directional geometry and Gromov hyperbolicity}\label{part:geometry}

\section{Directional gauges and complex supports}\label{sec:gauges}
In this section, we introduce the directional gauges of Wang and Zimmer and the complex supports that take the place of real supporting hyperplanes throughout the geometric part of the article. We show that, on a bounded domain each of whose boundary points admits a complex support, the gauge $\qone$ is a complex norm in its vector variable, the gauges $\qk$ are continuous and satisfy a quasi-triangle inequality, and the integrated distances $d_\Omega^{(k)}$ are proper and geodesic. We end the section with the splitting of the gauges into a normal part and a tangential part in a normal collar. We begin with notation.

Write $\Gr_\C(k,n)$ for the Grassmannian of complex $k$-dimensional linear subspaces. The complex tangent space is $T_p^\C\partial\Omega=\ker\partial\rho(p)$. For a complex $k$-plane $V$, we let
\[
 R_\Omega(z;V)=\dist\bigl(z,(z+V)\cap\partial\Omega\bigr),\qquad
 \delta_\Omega^{(k)}(z;v)=\sup_{\substack{V\in\Gr_\C(k,n)\\v\in V}}R_\Omega(z;V).
\]
The first quantity is the distance from the point to the boundary inside the affine plane through the point parallel to the given plane. The second quantity is the supremum of the first over all planes of the given dimension that contain the given direction. The generalized complex $k$-quasihyperbolic gauge of Wang and Zimmer \cite{WZ24} is
\[
 q_\Omega^{(k)}(z;v)=\frac{\norm v}{\delta_\Omega^{(k)}(z;v)},\qquad v\ne0,
\]
with value zero at $v=0$. We denote its integrated distance by $d_\Omega^{(k)}$.

When the boundary is $C^2$, we let $\nu$ be the outward unit normal, and we choose $r_0>0$ below a tubular-neighborhood radius. Throughout the geometric part of the article we use the normal collar
\begin{equation}\label{eq:collar-intro}
 p_r=p-r\nu(p),\qquad \Psi(p,t)=p_{r_0e^{-t}},\qquad
 H_t^{(k)}(p;v)=q_\Omega^{(k)}\bigl(\Psi(p,t);D_p\Psi(p,t)v\bigr),
\end{equation}
where $v\in T_p\partial\Omega$ is a real tangent vector. The parameter $t$ is thus the logarithmic Euclidean depth of the point of the collar. For a smooth defining function $\rho$, we write
\[
 M_k(\Omega)=\sup_{\substack{p\in\partial\Omega\\E^*E=I_k}}
 \ord_0\bigl(\rho(p+E\zeta)\bigr),
\]
where $E:\C^k\to\C^n$ is complex linear. The order of vanishing in this definition does not depend on the defining function, and it does not depend on the orthonormal coordinates chosen on the plane.

We also fix the following terminology. A symmetric homogeneous gauge $F$ is a nonnegative function on tangent vectors with $F(z;av)=|a|F(z;v)$ for real $a$. Its integrated distance is the infimum of $\int F(\sigma;\sigma')$ over all absolutely continuous curves joining the two given points. A metric is proper when its closed bounded balls are compact. It is geodesic when every pair of points is joined by a distance-realizing segment. We write $A\asymp B$ for a uniform two-sided comparison by positive constants. We write $B(a,r)$ for a Euclidean ball, and $\B^n=B(0,1)\subset\C^n$.

Euclidean norms are taken with respect to the standard Hermitian structure on $\C^n$, and the inner product is linear in its first variable. For a bounded domain $\Omega$, the directional radius in \eqref{eq:radius} equals
\[
 \dist\bigl(z,(z+\C v)\cap\partial\Omega\bigr).
\]
To see this, note that a centered disc of any smaller radius cannot leave $\Omega$ without meeting its boundary. On the other hand, a centered disc of larger radius contains a nearest boundary point in the line. The same argument shows that $R_\Omega(z;V)$ is the radius of the largest centered Euclidean ball in $z+V$ contained in $\Omega$. In particular, we have $q_\Omega^{(1)}(z;v)=\norm v/r_\Omega(z;v)$ and $q_\Omega^{(n)}(z;v)=\norm v/\delta_\Omega(z)$.

In what follows, we shall use the gauges $q_\Omega^{(k)}$ and the distances $d_\Omega^{(k)}$ defined above. The elementary bounds for these gauges are
\begin{equation}\label{eq:elementary}
 \frac{\norm v}{\diam\Omega}\le\qone(z;v)\le\qk(z;v)
 \le\frac{\norm v}{\delta_\Omega(z)}.
\end{equation}
We now introduce complex supports. A complex affine hyperplane $H$ is a \emph{support} for $\Omega$ if $H\cap\Omega=\varnothing$ and $H\cap\partial\Omega\ne\varnothing$. The geometric notion of $\C$-convexity goes back to Behnke and Peschl \cite{BP35}, to Martineau \cite{Mar68} and to Znamenskij \cite{Zna79}, and we refer the reader to \cite{APS04,Hor94} for the facts that we use. By a theorem of Andersson, Passare and Sigurdsson, and of H\"ormander \cite{APS04,Hor94}, every $\C$-convex domain is linearly convex, that is, every point of its complement lies on a complex affine hyperplane disjoint from the domain. In particular, every boundary point of a $\C$-convex domain admits a complex support. At a boundary point $p$ where $\partial\Omega$ is the graph of a function differentiable at $p$, such a support is necessarily the complex tangent hyperplane. Indeed, a complex hyperplane through $p$ which avoids $\Omega$ contains no direction transversal to the real tangent hyperplane at $p$, since every such direction enters $\Omega$, and the only complex hyperplane inside the real tangent hyperplane is the complex tangent hyperplane. Consequently, a $\C$-convex domain with $C^1$ boundary satisfies the \emph{support condition}
\begin{equation}\label{eq:global-support}
 (p+T_p^\C\partial\Omega)\cap\Omega=\varnothing,\qquad p\in\partial\Omega.
\end{equation}

Our first result expresses $\qone$ through the family of all complex affine hyperplanes that miss the domain. This representation gives the triangle inequality and the continuity of the gauge at the same time.

\begin{lemma}[The supporting norm]\label{lem:support-norm}
Suppose that $\Omega$ is bounded and every boundary point admits a complex support. Let $\mathcal H$ be the family of all complex affine hyperplanes disjoint from $\Omega$, and write $\phi_H(z)=\ell_H(z)-c_H$ for an affine equation of $H$. Then
\begin{equation}\label{eq:support-representation}
 \qone(z;v)=\sup_{H\in\mathcal H}\frac{\abs{\ell_H(v)}}{\abs{\phi_H(z)}}.
\end{equation}
Consequently, $\qone(z;\cdot)$ is a complex norm, and $\qone$ is continuous on $\Omega\times\C^n$.
\end{lemma}
\begin{proof}
We first prove the representation formula. Fix a hyperplane in the family. If $\ell_H(v)\ne0$, the line $z+\C v$ meets $H$ at parameter $-\phi_H(z)/\ell_H(v)$. The largest centered disc of this line in $\Omega$ cannot contain that point, because the hyperplane is disjoint from the domain. This proves that the right side of \eqref{eq:support-representation} does not exceed the left side. For the reverse inequality, take a nearest boundary point $a=z+\zeta_0v$ in the line, and take a complex support $H$ through $a$. The equality $\ell_H(v)=0$ would force $z\in H$, which is impossible because the support is disjoint from the domain. Hence
\[
 \abs{\phi_H(z)}=\abs{\zeta_0}\abs{\ell_H(v)},\qquad
 \abs{\zeta_0}=r_\Omega(z;v)/\norm v.
\]
This gives equality in the representation formula.

We now prove the remaining assertions. A supremum of absolute values of complex linear functionals is subadditive and complex homogeneous. Positivity of $\qone(z;\cdot)$ follows from \eqref{eq:elementary}. Hence $\qone(z;\cdot)$ is a complex norm. It remains to prove continuity. The supremum formula gives lower semicontinuity. For the upper semicontinuity, observe that a compact centered disc strictly inside $\Omega$ remains inside after a sufficiently small change of center and direction. Thus the directional radius is lower semicontinuous, and the gauge is upper semicontinuous away from $v=0$. The upper bound in \eqref{eq:elementary} gives continuity at $v=0$.
\end{proof}

The infinitesimal Carath\'eodory metric is $\gamma_\Omega(z;v)=\sup\{|df_z(v)|:f\in\OO(\Omega,\D),\ f(z)=0\}$. The next lemma records the distance estimates that complex supports provide. The first estimate holds for the integrated distance of every gauge. The last two estimates rest on a comparison between the invariant metrics and $\qone$, which we recall in the proof.
\begin{lemma}[Global support estimates]\label{lem:global-support}
Under the hypotheses of Lemma~\ref{lem:support-norm}, for every $H\in\mathcal H$, $1\le k\le n$, and $z,w\in\Omega$,
\begin{equation}\label{eq:dk-support}
 \dk(z,w)\ge\left|\log\frac{\abs{\phi_H(z)}}{\abs{\phi_H(w)}}\right|.
\end{equation}
If $\Omega$ is $\C$-convex, then
\begin{equation}\label{eq:invariant-comparison}
 \tfrac14\qone\le\gamma_\Omega\le\kappa_\Omega\le\qone
\end{equation}
and
\begin{equation}\label{eq:k-support}
 k_\Omega(z,w)\ge\tfrac14\left|\log\frac{\abs{\phi_H(z)}}{\abs{\phi_H(w)}}\right|.
\end{equation}
\end{lemma}
\begin{proof}
For an absolutely continuous curve $\sigma$ in $\Omega$, we have almost everywhere
\[
 \left|\frac{d}{du}\log\abs{\phi_H(\sigma(u))}\right|
 \le\frac{\abs{\ell_H(\sigma'(u))}}{\abs{\phi_H(\sigma(u))}}
 \le\qone(\sigma(u);\sigma'(u))\le\qk(\sigma(u);\sigma'(u)).
\]
The second inequality is the representation formula of the previous lemma, and the third is the elementary comparison between the gauges. We integrate along the curve and take the infimum over all curves joining the two points to obtain \eqref{eq:dk-support}. The invariant comparison is \cite[Proposition~1]{NPZ11}, stated with our normalization of the directional radius. Its lower bound gives the same derivative estimate with $4\kappa_\Omega$ on the right side, and integration gives \eqref{eq:k-support}. The same argument for the Carath\'eodory distance is in \cite[Proposition~2]{Nik14}.
\end{proof}

The gauge $\qone$ is a norm in its vector variable, but the gauges $\qk$ with $k\ge2$ need not satisfy a triangle inequality. The next lemma shows that they satisfy a quasi-triangle inequality with a constant depending only on the dimension, and that they are continuous.
\begin{lemma}\label{lem:quasitriangle}
Under the hypotheses of Lemma~\ref{lem:support-norm}, the gauges $\qk$ are continuous. There is $J_n\ge1$, depending only on $n$, such that
\begin{equation}\label{eq:quasitriangle}
 \qk(z;u+v)\le J_n\bigl(\qk(z;u)+\qk(z;v)\bigr)
\end{equation}
for all $z,u,v$ and $1\le k\le n$.
\end{lemma}
\begin{proof}
Fix $z$ and set $q=\qone(z;\cdot)$. For a complex $k$-plane $V$, let
\[
 Q(V)=\max_{\substack{u\in V\\\norm u=1}}q(u).
\]
The interpretation of the directional radius as the radius of the largest centered ball gives
\begin{equation}\label{eq:minimax-gauge}
 \qk(z;v)=\norm v\inf_{V\ni v}Q(V).
\end{equation}
We first compare the norm just defined with a Hermitian norm. Choose $b_1,\ldots,b_n$ in the compact unit ball of $q$ maximizing the absolute value of their complex determinant. They form a basis. If $u=\sum_j c_jb_j$, then replacement of the $j$th column by $u/q(u)$ shows that $\abs{c_j}\le q(u)$. This is Auerbach's lemma. Also $q(u)\le\sum_j\abs{c_j}$. Therefore the Hermitian norm
\[
 h(u)=n^{-1/2}\Bigl(\sum_j\abs{c_j}^2\Bigr)^{1/2}
\]
satisfies $h\le q\le nh$.

Let $0<\lambda_1\le\cdots\le\lambda_n$ be the principal rates of $h$ relative to the Euclidean norm, and define
\[
 N_k(v)=\max\{h(v),\lambda_k\norm v\}.
\]
This is a norm. Write $q_h^{(k)}$ for the quantity in \eqref{eq:minimax-gauge} with $h$ in place of $q$. We compare these two quantities. The minimax principle and the constraint $v\in V$ give $N_k(v)\le q_h^{(k)}(v)$.

For the reverse estimate, we choose a complex $(k-1)$-plane
\[
 W\subset\spanC\{e_1,\ldots,e_k\}\cap v^\perp,
\]
where $e_j$ is an orthonormal eigenbasis for $h$. Such a plane exists by the dimension inequality. On $V=\C v\oplus W$, we write $u=a\widehat v+w$ orthogonally, where $\widehat v=v/\norm v$. Then
\[
 h(u)\le\abs a\,h(\widehat v)+\lambda_k\norm w
 \le\sqrt{h(\widehat v)^2+\lambda_k^2}\,\norm u.
\]
It follows that $q_h^{(k)}(v)\le\sqrt2N_k(v)$. Combining the two comparisons, we get
\[
 N_k(v)\le\qk(z;v)\le\sqrt2\,nN_k(v).
\]
Since the quantity on the right side of this comparison is a norm, this proves \eqref{eq:quasitriangle}, for example with $J_n=\sqrt2\,n$.

It remains to prove continuity, which follows from \eqref{eq:minimax-gauge} and the compactness of the Grassmannian. We make this explicit. A small change in a nonzero direction can be realized by a unitary map tending to the identity, and this map transports the constrained planes containing the old direction onto those containing the new one. Continuity of $q$ is uniform on the relevant compact unit spheres. The bound \eqref{eq:elementary} handles the zero vector. This completes the proof.
\end{proof}

We now show that the integrated distances are proper and geodesic. The logarithmic support estimates show that the distance from a fixed point tends to infinity at the boundary, and this is what properness needs.
\begin{lemma}\label{lem:proper}
If $\Omega$ is bounded and every boundary point admits a complex support, each $(\Omega,\dk)$ is a proper geodesic metric space. If $\Omega$ is $\C$-convex, the same holds for $(\Omega,k_\Omega)$.
\end{lemma}
\begin{proof}
The local Euclidean bounds imply that the integrated distances induce the Euclidean topology. The gauges can be integrated along every absolutely continuous curve. See also \cite[Proposition~2.5]{WZ24}, which applies to every bounded domain. We now show that closed bounded balls are compact. Fix $z_0\in\Omega$. If $z_j\to p\in\partial\Omega$, then a support $H$ through $p$ satisfies $\phi_H(z_j)\to0$ and $\phi_H(z_0)\ne0$. Lemma~\ref{lem:global-support} shows that the relevant distance from $z_0$ tends to infinity along the sequence. Thus every closed bounded metric ball is relatively compact in $\Omega$, and hence it is compact. Finally, integrated distances are length distances, and proper length spaces are geodesic by the Hopf--Rinow--Cohn-Vossen theorem \cite[Proposition~I.3.7]{BH99}. This proves the lemma.
\end{proof}

In what follows, we shall use the following consequence of Gromov hyperbolicity. Fix $\Lambda\ge1$ and $b\ge0$. In a proper geodesic Gromov hyperbolic space, quadrilaterals with $(\Lambda,b)$-quasigeodesic sides have a common thinness constant. Our convention for parametrized quasigeodesics is
\[
 \Lambda^{-1}\abs{u-u'}-b\le d(\sigma(u),\sigma(u'))
 \le\Lambda\abs{u-u'}+b.
\]
To see the assertion about quadrilaterals, note that Morse stability places each side at a uniformly bounded Hausdorff distance from a geodesic with the same endpoints. Splitting the resulting geodesic quadrilateral into two triangles then gives the common thinness constant. For the stability of quasigeodesics we refer the reader to \cite[Theorem~III.H.1.7]{BH99}.

Complex supports also show that the directional radii of a bounded $\C$-convex domain collapse at the boundary, unless the boundary contains an affine disc. The argument is standard. We refer the reader to \cite{FS98} for convex domains and to \cite{Zim17}. We include the short proof for completeness.

\begin{lemma}\label{lem:shrinking}
Let $D\subset\C^n$, $n\ge2$, be a bounded $\C$-convex domain with no nontrivial affine complex disc in its boundary. Then
\begin{equation}\label{eq:Rtozero}
 \RR_D(z)\longrightarrow0\qquad(z\to\partial D).
\end{equation}
\end{lemma}

\begin{proof}
Suppose, to get a contradiction, that \eqref{eq:Rtozero} fails. Then there are $z_j\to p\in\partial D$, unit vectors $v_j\to v$ and $a>0$ with $z_j+a\zeta v_j\in D$ for $|\zeta|<1$. The maps $\phi_j(\zeta)=z_j+a\zeta v_j$ converge to $\phi(\zeta)=p+a\zeta v$, and $\phi(\D)\subset\overline D$. Let $H=\{\Lambda=0\}$ be a complex support through $p$. The functions $\Lambda\circ\phi_j$ have no zeros on $\D$, whereas their limit vanishes at $0$. By Hurwitz's theorem, $\Lambda\circ\phi\equiv0$. Hence $\phi(\D)\subset H\cap\overline D\subset\partial D$ is a nontrivial affine disc in the boundary. This contradiction proves the lemma.
\end{proof}

\subsection{The normal collar}
In this subsection, we split the gauges in the normal collar into a normal part and a tangential part. Throughout, we suppose that $\partial\Omega$ is $C^2$ and that \eqref{eq:global-support} holds. We choose $r_0$ below a tubular-neighborhood radius, and we use the notation in \eqref{eq:collar-intro}.

The next lemma is the normal splitting. It shows that the gauge of a vector in the collar is comparable to the sum of the absolute value of its normal component and the transverse gauge of its tangential component. It also shows that the transverse gauge at the initial level is controlled by the transverse gauge at every depth.

\begin{lemma}[Normal splitting]\label{lem:normal-splitting}
Suppose that $\Omega$ is bounded, $\partial\Omega$ is $C^2$, \eqref{eq:global-support} holds, and $r_0$ is as in \eqref{eq:collar-intro}, decreased if necessary. There are $C,L,B\ge1$ such that, for all $p\in\partial\Omega$, $0<r\le r_0$, $t\ge0$, $v\in T_p\partial\Omega$, $a\in\R$ and $1\le k\le n$,
\begin{align}
 C^{-1}\qk(p_r;v)&\le\qk(p_r;(I-rD\nu_p)v)\le C\qk(p_r;v),\label{eq:collar-vector}\\
 L^{-1}\bigl(\abs a+\Ht(p;v)\bigr)&\le\qk\bigl(\Psi(p,t);D\Psi(p,t)(v,a)\bigr)
 \le L\bigl(\abs a+\Ht(p;v)\bigr),\label{eq:normal-splitting}\\
 H_0^{(k)}(p;v)&\le B\Ht(p;v).\label{eq:initial-level}
\end{align}
\end{lemma}
\begin{proof}
We decrease $r_0$ so that $\frac12\norm v\le\norm{(I-rD\nu_p)v}\le2\norm v$. The support at $p$ and \eqref{eq:elementary} give
\begin{equation}\label{eq:normal-support}
 \qk(p_r;u)\ge\frac{\abs{\ip u{\nu(p)}}}{r},
 \qquad \qk(p_r;r\nu(p))=1.
\end{equation}
Let $K=\sup_p\norm{D\nu_p}$ and $w=(I-rD\nu_p)v$. Since $\delta_\Omega(p_r)=r$, we have
\[
 \qk(p_r;w-v)\le K\norm v,\qquad
 \qk(p_r;v)\ge\frac{\norm v}{\diam\Omega},\qquad
 \qk(p_r;w)\ge\frac{\norm v}{2\diam\Omega}.
\]
We apply \eqref{eq:quasitriangle} in both directions to obtain \eqref{eq:collar-vector}.

We now prove the splitting. Since $\nu$ has unit length, $D\nu_pv$ is real tangent at $p$. Thus $\Ree\ip w{\nu(p)}=0$. We have $D\Psi(p,t)(v,a)=w+ar\nu(p)$, and \eqref{eq:normal-support} gives
\[
 \qk(p_r;w+ar\nu(p))\ge\abs a.
\]
The quasi-triangle inequality implies
\[
 \qk(p_r;w)\le J_n\bigl(\qk(p_r;w+ar\nu(p))+\abs a\bigr)
 \le2J_n\qk(p_r;w+ar\nu(p)).
\]
This proves the lower splitting. The upper splitting follows from the same quasi-triangle inequality and \eqref{eq:normal-support}. Finally, we have
\[
 \Ht(p;v)\ge\frac{\norm v}{2\diam\Omega},\qquad
 H_0^{(k)}(p;v)\le\frac{2\norm v}{r_0},
\]
which proves \eqref{eq:initial-level}.
\end{proof}

\section{Hardy collars and finite-type obstructions}\label{sec:hardy}
In this section, we give a criterion for the Gromov hyperbolicity of a collar metric in terms of an integral inequality of Hardy type for its transverse gauge. We then show that the same inequality, together with a comparison on one Kobayashi ball of fixed radius, bounds the D'Angelo type of the boundary. We first consider transverse fields without any convexity assumption. A continuous positive symmetric homogeneous gauge on a tangent bundle need not satisfy a triangle inequality. On compact subsets it is nevertheless uniformly comparable to a background norm. Its integral over absolutely continuous curves therefore defines a length distance with the usual local topology. Such a gauge $F$ is strongly integrable in the sense of \cite[Section~2.4]{WZ24}. To see this, let $\sigma$ be absolutely continuous. Then $t\mapsto F(\sigma(t);\sigma'(t))$ is measurable, since it is the composition of a continuous function with the measurable map $t\mapsto(\sigma(t),\sigma'(t))$. It is bounded by a constant times $\norm{\sigma'(t)}$, and hence it is integrable. For the Kobayashi--Royden metric this is due to Venturini \cite{Ven89}.

Our first lemma shows that the Hardy inequality forces exponential growth of the transverse gauge in the depth variable, at a rate given by the reciprocal of the Hardy constant.
\begin{lemma}\label{lem:hardy-growth}
Let $M$ be a compact smooth manifold with a background norm. Suppose that, for each $t\ge0$, $H_t$ is a positive symmetric homogeneous gauge on $TM$, that $(t,p,v)\mapsto H_t(p;v)$ is continuous on $[0,\infty)\times TM$, and that
\begin{equation}\label{eq:hardy}
 \int_0^t H_s(p;v)\,ds\le AH_t(p;v),
 \qquad p\in M,\quad v\in T_pM,\quad t\ge0,
\end{equation}
for some $A>0$. Then there are $b>0$ and $B\ge1$ such that
\[
 H_t(p;v)\ge be^{t/A}\norm v,\qquad H_0(p;v)\le BH_t(p;v).
\]
\end{lemma}
\begin{proof}
Fix $(p,v)$ with $v\ne0$ and put $G(t)=\int_0^tH_s(p;v)\,ds$. Then $G>0$ for $t>0$, and the Hardy inequality reads $G\le AG'$. Integrating this differential inequality, we get, for $t\ge1$,
\[
 G(t)\ge G(1)e^{(t-1)/A},\qquad
 H_t(p;v)\ge A^{-1}G(1)e^{(t-1)/A}.
\]
We now make the lower bound uniform. Compactness of $[0,1]$ times the unit tangent bundle gives $H_s(p;v)\ge b_0\norm v$ there. In particular, $G(1)\ge b_0\norm v$ uniformly. Decreasing the resulting constant extends the lower bound to $0\le t\le1$. The uniform upper bound for $H_0$ on the unit tangent bundle proves the second assertion.
\end{proof}

We are now ready to present the principal result of this section. It gives Gromov hyperbolicity of the collar metric from the Hardy inequality alone. The proof compares the given gauge with a primitive gauge that expands at a definite exponential rate.

\begin{theorem}[Hardy collar criterion]\label{thm:hardy}
Let $M$ be compact and connected, and let $H$ satisfy the hypotheses of Lemma~\ref{lem:hardy-growth}. On $X=M\times[0,\infty)$ define
\[
 F(p,t;v,a)=\abs a+H_t(p;v),\qquad
 \widehat H_t(p;v)=H_0(p;v)+\int_0^tH_s(p;v)\,ds.
\]
Let $d$ and $\widehat d$ be the integrated distances of $F$ and $\widehat F=\abs a+\widehat H_t$. Let $A$ be the constant in \eqref{eq:hardy}, let $B$ be given by Lemma~\ref{lem:hardy-growth}, and put $C=A+B$ and $C_* =\max\{1,C\}$. Then
\begin{align}
 \widehat H_s(p;v)&\le e^{-(t-s)/C}\widehat H_t(p;v),
 &&0\le s\le t,\label{eq:primitive-expansion}\\
 C_*^{-1}\widehat d(x,y)&\le d(x,y)\le\widehat d(x,y)+1.
 \label{eq:primitive-distance}
\end{align}
Both distances are proper and geodesic, and both spaces are Gromov hyperbolic.
\end{theorem}
\begin{proof}
We first prove the expansion of the primitive gauge. For fixed $(p,v)$ with $v\ne0$, the function $g(t)=\widehat H_t(p;v)$ satisfies
\[
 g(t)\le(A+B)H_t(p;v)=Cg'(t).
\]
The inequality combines the Hardy inequality with the bound for the initial gauge. Integrating $g'/g\ge1/C$ proves \eqref{eq:primitive-expansion}. Since $\widehat F\le C_*F$, the first inequality in \eqref{eq:primitive-distance} follows.

We now prove the second distance inequality. For $0\le u\le1$, consider the inward shift
\[
 R_u(p,t)=\bigl(p,(t-u)_+\bigr),\qquad (s)_+=\max\{s,0\}.
\]
The vertical displacement from a point to its image has $F$-length at most $u$. If $\sigma$ joins $x$ to $y$, then
\[
 d(x,y)\le2u+L_F(R_u\sigma).
\]
For every $t,p,v$,
\begin{align*}
 \int_0^1 H_{(t-u)_+}(p;v)\,du
 &=(1-t)_+H_0(p;v)+\int_{(t-1)_+}^tH_s(p;v)\,ds\\
 &\le\widehat H_t(p;v).
\end{align*}
The absolute vertical derivative of $R_u\sigma$ is no larger than that of $\sigma$. Tonelli's theorem therefore gives
\[
 \int_0^1 L_F(R_u\sigma)\,du\le L_{\widehat F}(\sigma).
\]
Integrating the preceding distance inequality in $u$ yields
$d(x,y)\le1+L_{\widehat F}(\sigma)$. Taking the infimum over all curves joining the two points proves the other inequality in \eqref{eq:primitive-distance}.

Next we show that both distances are proper and geodesic. Both distances dominate the absolute difference of heights, and both are locally comparable to a product Riemannian distance on compact slabs. Every closed bounded ball therefore lies in a compact slab, and hence it is compact. Thus both spaces are proper length spaces, and hence they are geodesic.

It remains to prove Gromov hyperbolicity. We use the expanding-collar theorem of Wang and Zimmer \cite[Theorem~1.1]{WZ24}. It applies to a strongly integrable pseudo-metric on $M\times(0,1]$, with $M$ compact, under three conditions. The first condition is that the radial vector field $\partial_r$ has length comparable to $1/r$. The second condition is that the radial and transverse directions are uniformly transverse. The third condition is that transverse lengths expand at a power rate as $r\downarrow0$. In the coordinate $r=e^{-t}$, the gauge $\widehat F$ becomes
\[
 \widehat F(p,r;v,b)=\frac{\abs b}{r}+\widehat H_{\log(1/r)}(p;v).
\]
It is continuous, homogeneous and locally comparable to a product norm, and hence it is strongly integrable. The radial vector field has length exactly $1/r$, and the splitting is exact, so the first two conditions hold. By \eqref{eq:primitive-expansion}, for $0<r\le r'\le1$,
\[
 \widehat H_{\log(1/r)}(p;v)\ge(r'/r)^{1/C}\,\widehat H_{\log(1/r')}(p;v),
\]
which is the third condition with exponent $1/C$. Hence $(X,\widehat d)$ is Gromov hyperbolic, and the quasi-isometry \eqref{eq:primitive-distance} transfers Gromov hyperbolicity to $(X,d)$. This completes the proof.
\end{proof}

The Hardy collar criterion treats one collar over a connected manifold. The next lemma allows the boundary to have more than one component. It assembles the collars over the components into a metric wedge and compares the wedge with the whole domain.
\begin{lemma}[Compact interior]\label{lem:compact-interior}
Let $U$ be a bounded domain with $C^2$ boundary, let $\Psi$ be the normal collar \eqref{eq:collar-intro} for $U$, and let $d_F$ be obtained by integrating a Borel symmetric homogeneous gauge locally uniformly comparable to the Euclidean norm. Suppose that the pullback of this gauge by $\Psi$ is uniformly comparable to $\abs a+H_t(p;v)$, for some continuous positive symmetric homogeneous gauge $H$ on $T\partial U$ depending continuously on $t\ge0$. Let $N_1,\ldots,N_\nu$ be the closed collars of the boundary components, and let $d_i$ be the distance on $N_i$ obtained by integrating the gauge along paths in $N_i$. If each $(N_i,d_i)$ is proper and Gromov hyperbolic, then $(U,d_F)$ is proper, geodesic, and Gromov hyperbolic.
\end{lemma}
\begin{proof}
The boundary has finitely many components. Let $\Sigma_1,\ldots,\Sigma_{\nu}$ be the initial levels of the collars $N_1,\ldots,N_\nu$. Each $\Sigma_i$ has finite intrinsic length diameter $D_i$. Indeed, it is a compact connected manifold, and the gauge is uniformly bounded above on it.

We first compare the collar distances with the distance of the domain. Let $x,y\in N_i$. A path in $U$ which leaves $N_i$ can have the portion between its first exit and its last re-entry replaced by a path in $\Sigma_i$ of length at most $D_i+\eta$. Taking infima and then letting $\eta\downarrow0$ gives
\begin{equation}\label{eq:same-collar}
 d_F(x,y)\le d_i(x,y)\le d_F(x,y)+D_i.
\end{equation}
If $x\in N_i$, $y\in N_j$ and $i\ne j$, then every joining path leaves $N_i$ through $\Sigma_i$ and enters $N_j$ through $\Sigma_j$. Hence
\begin{equation}\label{eq:different-collars}
 d_i(x,\Sigma_i)+d_j(y,\Sigma_j)
 \le d_F(x,y)
 \le d_i(x,\Sigma_i)+d_j(y,\Sigma_j)+C.
\end{equation}
For the upper bound, we choose one point $b_i\in\Sigma_i$ on each level and join the finitely many $b_i$ by fixed paths of finite length in $U$. Paths in the initial levels and these fixed paths give a uniform $C$.

We now form the metric wedge of the spaces $(N_i,d_i)$ by identifying the points $b_i$. Within one collar the wedge distance is $d_i$. Between different collars it is $d_i(x,b_i)+d_j(y,b_j)$. Since
\[
 d_i(x,\Sigma_i)\le d_i(x,b_i)\le d_i(x,\Sigma_i)+D_i,
\]
the inequalities \eqref{eq:same-collar} and \eqref{eq:different-collars} compare this wedge with the union of the collars in $U$, up to a bounded additive error. The compact remaining core has bounded $d_F$-distance to the initial levels. Thus the wedge and $(U,d_F)$ are quasi-isometric. The wedge is a proper geodesic Gromov hyperbolic space. Indeed, its geodesic triangles consist of triangles in its components and of segments through the common point, so the maximum of the finitely many hyperbolicity constants is a hyperbolicity constant for the wedge.

Finally, we prove properness directly. The lower normal estimate controls the variation of logarithmic depth along paths. Hence a bounded metric ball cannot approach the boundary. It lies in a compact Euclidean subset of $U$, and the local comparison of topologies gives properness. Since $d_F$ is a length distance, it is geodesic. Quasi-isometry invariance of Gromov hyperbolicity completes the proof.
\end{proof}

The collar arguments above were formulated over a smooth transverse manifold. They also apply to domains with $C^2$ boundary, for which the normal collar is only $C^1$, by the following reduction to a smooth reference hypersurface.

\begin{lemma}[A smooth reference hypersurface]\label{lem:reference}
Let $U$ be a bounded domain with $C^2$ boundary and outward unit normal $\nu$. There are a compact $C^\infty$ hypersurface $\Sigma$ and a $C^1$ diffeomorphism $h:\Sigma\to\partial U$ such that each $\sigma\in\Sigma$ lies on the normal segment through $h(\sigma)$. Consequently, Lemma~\ref{lem:hardy-growth}, Theorem~\ref{thm:hardy} and Lemma~\ref{lem:compact-interior} apply to collar gauges over $\partial U$.
\end{lemma}

\begin{proof}
The signed distance $d$ to $\partial U$, negative in $U$, is $C^2$ on a tubular neighborhood $N=\{\abs d<r_1\}$ by \cite[Lemma~14.16]{GT01}, and the nearest-point projection $\pi:N\to\partial U$ is $C^1$. We smooth the signed distance. Choose $\widetilde d\in C^\infty(N)$ with $\norm{\widetilde d-d}_{C^1(N)}<\min\{1/2,r_1/4\}$. Along each normal segment $t\mapsto p-t\nu(p)$, $\abs t\le r_1/2$, the derivative of $d$ is $-1$. Hence the derivative of $\widetilde d$ is less than $-1/2$, and $\widetilde d$ has opposite signs at the two ends of the segment. Therefore $\Sigma=\{\widetilde d=0\}\cap\{\abs d<r_1/2\}$ is a compact smooth hypersurface meeting each segment exactly once. The map $h=\pi|_\Sigma$ is a bijection onto $\partial U$. Its inverse is $p\mapsto p-t(p)\nu(p)$, where $t(p)$ solves $\widetilde d(p-t\nu(p))=0$, and this inverse is $C^1$ by the implicit function theorem. Pulling back by the bi-Lipschitz map $h$ changes the constants in the splitting, Hardy and expansion inequalities, but not their form, and it preserves the absolute continuity of curves. This proves the lemma.
\end{proof}

\subsection{A Hardy collar detects full curve type}
In this subsection, we show that a Hardy inequality for the transverse gauge of a collar metric bounds the D'Angelo type of the boundary. Yu \cite{Yu95} related finite type to the boundary growth of Kobayashi-type metrics. The next theorem uses three ingredients, namely a lower estimate for mixed vectors, a Hardy inequality for the transverse gauge, and a comparison on one fixed Kobayashi ball. We emphasize that the auxiliary distance need not be Gromov hyperbolic.
\begin{theorem}\label{thm:curve-type}
Let $U\Subset\C^n$, $n\ge2$, be a domain with $C^\infty$ boundary, and let $\Psi$ be the Euclidean normal collar \eqref{eq:collar-intro} for $U$. Let $d_F$ be induced by a Borel symmetric homogeneous gauge $F$ locally uniformly comparable to the Euclidean norm. Suppose that a transverse gauge $H$ on $T\partial U$ satisfies the hypotheses of Lemma~\ref{lem:hardy-growth} with $M=\partial U$, and that there is $c>0$ such that
\begin{equation}\label{eq:type-lower-splitting}
 F\bigl(\Psi(p,t);D\Psi(p,t)(v,a)\bigr)
 \ge c\bigl(\abs a+H_t(p;v)\bigr)
\end{equation}
for all $p\in\partial U$, $t\ge0$, $v\in T_p\partial U$ and $a\in\R$. Suppose that there are $R_*>0$ and $K_*<\infty$ such that, for all $x$ sufficiently near $\partial U$ and all $y\in U$,
\begin{equation}\label{eq:fixed-ball}
 k_U(x,y)\le R_*\quad\Longrightarrow\quad d_F(x,y)\le K_*.
\end{equation}
Then $A\ge2$ and
\[
 \Delta_1(\partial U,p)\le A,\qquad p\in\partial U.
\]
\end{theorem}

\begin{proof}
Put $\eps=\min\{1,1/A\}$. We first show that balls of bounded gauge radius are small in the Euclidean sense near the boundary. Since $r=r_0e^{-t}$ and $D\nu$ is bounded,
\[
 \norm{D\Psi(p,t)(v,a)}\le C_0\norm v+r\abs a.
\]
Lemma~\ref{lem:hardy-growth} and \eqref{eq:type-lower-splitting} imply
\begin{equation}\label{eq:euclidean-speed}
 \norm{D\Psi(p,t)(v,a)}
 \le C_1e^{-\eps t}F\bigl(\Psi(p,t);D\Psi(p,t)(v,a)\bigr).
\end{equation}
Fix $K<\infty$, let $x=\Psi(p,T)$, and suppose that $d_F(x,y)\le K$. Choose a path from $x$ to $y$ with gauge length at most $K+1$. While the path lies in the collar, its total variation in $t$ is at most $(K+1)/c$. When $T>(K+1)/c$, the path cannot reach the initial level, and therefore it remains in the collar. Integration of \eqref{eq:euclidean-speed} along the path gives
\begin{equation}\label{eq:ball-euclidean}
 \norm{x-y}\le C_K\delta_U(x)^\eps
\end{equation}
for all sufficiently small $\delta_U(x)$, uniformly in $x,y$. Combining this inequality with \eqref{eq:fixed-ball}, we obtain $C_*,\eta>0$ such that
\begin{equation}\label{eq:k-ball-euclidean}
 \delta_U(x)<\eta,\quad k_U(x,y)\le R_*
 \quad\Longrightarrow\quad \norm{x-y}\le C_*\delta_U(x)^\eps.
\end{equation}

We now prove that this bound forces $\Delta_1\le1/\eps$. Let $\gamma$ be a holomorphic curve germ through $p$, of multiplicity $m=\ord_0(\gamma-p)$. Suppose that its contact ratio exceeds $1/\eps$. Then there is an integer $N$ with $N\eps>m$ and
\[
 \abs{\rho(\gamma(\zeta))}\le C_2\abs\zeta^N.
\]
For finite contact order, we take $N$ to be that order. For infinite contact, we choose any sufficiently large $N$. Let $b=d\rho_p[\nu(p)]>0$. On a fixed neighborhood of $p$, the derivative in the direction $\nu(p)$ is at least $b/2$. Choose $D>2C_2/b$ and define
\[
 f_R(\zeta)=\gamma(R\zeta)-DR^N\nu(p),\qquad \abs\zeta<1.
\]
For small $R$, all the normal segments used here remain in this neighborhood. The fundamental theorem of calculus gives
\[
 \rho(f_R(\zeta))\le C_2R^N-\tfrac12bDR^N<0.
\]
Thus $f_R:\D\to U$ is holomorphic. Its center is $x_R=p-DR^N\nu(p)$ and $\delta_U(x_R)=DR^N$. Choose a fixed $0<\theta<1$ with $k_\D(0,\theta)<R_*$, and let $y_R=f_R(\theta)$. Writing
$\gamma(\zeta)=p+c_m\zeta^m+O(\zeta^{m+1})$, $c_m\ne0$, we obtain
\[
 \norm{y_R-x_R}\asymp R^m,\qquad
 k_U(x_R,y_R)\le k_\D(0,\theta)<R_*.
\]
Equation~\eqref{eq:k-ball-euclidean} would imply $c_3R^m\le C_*D^\eps R^{N\eps}$, which is impossible as $R\downarrow0$. This contradiction shows that the contact ratio of every curve germ is bounded by the reciprocal of the exponent. Taking the supremum over all curve germs gives
$\Delta_1\le1/\eps=\max\{1,A\}$.

Finally, a smooth hypersurface in $\C^n$, $n\ge2$, has type at least two, since a nonzero complex tangent affine line has zero first-order term. Hence $A\ge2$ and $\Delta_1\le A$. This completes the proof.
\end{proof}

\begin{remark}
Condition \eqref{eq:fixed-ball} holds, with $K_*=\Lambda R_*+C$, whenever $d_F(x,y)\le\Lambda k_U(x,y)+C$ for all $x,y\in U$.
\end{remark}

As an application of the above theorem, we obtain the following corollary for the Kobayashi--Royden metric itself. It gives Gromov hyperbolicity and finite type at the same time, from a two-sided normal splitting and the Hardy inequality.
\begin{corollary}\label{cor:kobayashi-hardy}
Let $U\Subset\C^n$, $n\ge2$, be a domain with $C^\infty$ boundary, and let $\Psi$ be the normal collar \eqref{eq:collar-intro} for $U$. Suppose that there is a gauge $H$ on $T\partial U$ satisfying the hypotheses of Lemma~\ref{lem:hardy-growth} with $M=\partial U$, including \eqref{eq:hardy}, and constants $0<c\le C$ such that
\[
 c\bigl(\abs a+H_t(p;v)\bigr)\le\kappa_U\bigl(\Psi(p,t);D\Psi(p,t)(v,a)\bigr)\le C\bigl(\abs a+H_t(p;v)\bigr)
\]
for all $p\in\partial U$, $t\ge0$, $v\in T_p\partial U$ and $a\in\R$. Then $(U,k_U)$ is Gromov hyperbolic, and $\partial U$ is of finite type with
\[
 \Delta_1(\partial U,p)\le A,\qquad p\in\partial U.
\]
\end{corollary}

\begin{proof}
Theorem~\ref{thm:hardy} and Lemma~\ref{lem:compact-interior} give Gromov hyperbolicity. For the type bound, we apply Theorem~\ref{thm:curve-type} with $F=\kappa_U$ and $d_F=k_U$. The fixed-ball comparison \eqref{eq:fixed-ball} then holds with $K_*=R_*$. This proves the corollary.
\end{proof}

\begin{remark}[Integrated perturbations]\label{rem:perturbations}
We note that the Hardy criterion tolerates errors that are controlled only after integration in the depth variable. Suppose that $H^0_s\le C_0e^{-\lambda(t-s)}H^0_t$ and that $E_t\ge0$ satisfies
\[
 \int_0^tE_s\,ds\le DH^0_t.
\]
Whenever $H=H^0+E$ is a continuous positive homogeneous gauge, we have
\[
 \int_0^tH_s\,ds\le(C_0/\lambda+D)H^0_t
 \le(C_0/\lambda+D)H_t.
\]
Thus the Hardy criterion permits errors for which $E_t/H^0_t$ is not uniformly bounded.
\end{remark}

\section{Finite affine-plane contact and the generalized metrics}\label{sec:planes}
In this section, we characterize Gromov hyperbolicity of the generalized distances by finite contact of affine complex planes with the boundary, and we identify the sharp normal expansion exponent. Throughout this section, $\Omega$ is bounded, has $C^\infty$ boundary, and satisfies \eqref{eq:global-support}. We recall the maximal contact order $M_k(\Omega)$ from Section~\ref{sec:gauges}. The order in its definition is taken in the $2k$ real variables underlying $\C^k$.

We are now ready to present the principal result of this section.

\begin{theorem}\label{thm:planes}
Let $\Omega\Subset\C^n$ be a domain with $C^\infty$ boundary, and suppose that \eqref{eq:global-support} holds.
For each $1\le k\le n$, the following conditions are equivalent.
\begin{enumerate}
\item $M_k(\Omega)<\infty$.
\item There are $C,\lambda>0$ such that
$H_s^{(k)}(p;v)\le Ce^{-\lambda(t-s)}H_t^{(k)}(p;v)$ for $0\le s\le t$.
\item There is $A<\infty$ such that
$\int_0^t H_s^{(k)}(p;v)\,ds\le AH_t^{(k)}(p;v)$ for all $p,v,t$.
\item $(\Omega,d_\Omega^{(k)})$ is Gromov hyperbolic.
\end{enumerate}
If $M=M_k(\Omega)<\infty$, then (ii) holds with $\lambda=1/M$ and fails for every $\lambda>1/M$. Moreover, after decreasing $r_0$, there is $C'>0$ such that
\begin{equation}\label{eq:two-depth-intro}
 R_\Omega(p_r;V)\le C'(r/u)^{1/M}R_\Omega(p_u;V),
 \qquad 0<r\le u\le r_0,
\end{equation}
for all $p\in\partial\Omega$ and $V\in\Gr_\C(k,n)$.
\end{theorem}

\begin{remark}
For smoothly bounded convex domains, the equivalence of (i) and (iv) is due to Wang and Zimmer \cite[Theorem~1.6]{WZ24}. For bounded convex domains, Wang and Zimmer also prove that (iv) is equivalent to a two-depth contraction of the form \eqref{eq:two-depth-intro}. We refer the reader to \cite[Theorem~1.5]{WZ24}. Theorem~\ref{thm:planes} differs from these results in three respects. It replaces real convexity by the complex support condition \eqref{eq:global-support}. It adds the Hardy condition (iii). It identifies the sharp exponent $1/M_k(\Omega)$.
\end{remark}

We first note that finite order at every pair $(p,E)$ already implies $M_k(\Omega)<\infty$. This follows from the compactness of the boundary and of the complex Stiefel manifold. Indeed, a sequence of pairs with orders tending to infinity would have a limit pair, and at the limit pair every derivative would vanish.

The proof of the theorem follows the approach of Wang and Zimmer \cite{WZ24} in the convex case. The expanding-collar theorem \cite[Theorem~1.1]{WZ24} gives Gromov hyperbolicity, and thick quasigeodesic quadrilaterals rule it out. In the distance estimates, global complex supports take the place of real convexity. The finite-jet calculation of the next subsection gives a uniform comparison between two depths in each fixed affine plane.

\subsection{Uniform finite jets on fixed planes}
In this subsection, we fix an affine plane and compare its first contact with the boundary at two normal depths. The comparison rests on uniform control of the Taylor polynomial of the boundary graph restricted to the plane. To this end, choose unitary coordinates centered at $p$, with $\nu(p)=e_1$, so that in a uniform neighborhood
\begin{equation}\label{eq:smooth-graph}
 \Omega=\{(z_1,z'):x_1+f_p(y_1,z')<0\},\qquad
 f_p(0)=0,\quad Df_p(0)=0.
\end{equation}
Finitely many smooth coordinate families cover the compact boundary, and they give uniform derivative bounds. Since the complex tangent hyperplane at the base point does not meet $\Omega$, we have
\begin{equation}\label{eq:tangent-nonnegative}
 f_p(0,z')\ge0.
\end{equation}
For a complex linear isometry $E:\C^k\to\C^n$, we set
\[
 G_{p,E}(\zeta)=\Ree(E\zeta)_1+f_p\bigl(\Imm(E\zeta)_1,(E\zeta)'\bigr),\qquad
 F_{p,E}(R)=\max_{\norm\zeta\le R}G_{p,E}(\zeta).
\]
Within the coordinate neighborhood, we have
\begin{equation}\label{eq:sublevel-radius}
 p_r+EB(0,R)\subset\Omega
 \quad\Longleftrightarrow\quad G_{p,E}(\zeta)<r\quad(\norm\zeta<R).
\end{equation}
In this description the plane is fixed, and a change of the normal depth changes only the sublevel value.

The following lemma compares the maximum of the restricted graph function with its jet polynomial, uniformly in the boundary point and in the plane. It is the main technical step in the sufficient direction.

\begin{lemma}\label{lem:finite-jets}
Under the standing assumptions of this section, there are $R_*>0$ and $C_0\ge1$, independent of $p$ and $E$, such that
\begin{equation}\label{eq:positive-maximum}
 F_{p,E}(R)\le\norm{G_{p,E}}_{L^\infty(B_R)}\le C_0F_{p,E}(R),
 \qquad 0<R\le R_*.
\end{equation}
If $M=M_k(\Omega)<\infty$, then there are $c,C>0$, independent of $p$ and $E$, such that the polynomials
\begin{equation}\label{eq:jet-polynomial}
 J_{p,E}(R)=\sum_{1\le\abs\alpha\le M}\abs{c_\alpha(p,E)}R^{\abs\alpha},
\end{equation}
where $c_\alpha(p,E)$ are the real Taylor coefficients of $G_{p,E}$ at $0$ in real coordinates on $\C^k$, have nonnegative coefficients, not all zero, and satisfy
\begin{equation}\label{eq:jet-comparison}
 cJ_{p,E}(R)\le F_{p,E}(R)\le CJ_{p,E}(R),\qquad 0<R\le R_*.
\end{equation}
\end{lemma}
\begin{proof}
Let $a=\norm{\pi_1E}_{\mathrm{op}}$, where $\pi_1$ is the projection onto the complex normal coordinate. We first bound the negative part of the restricted graph function. The uniform second derivative bounds and $Df_p(0)=0$ imply
\[
 \abs{f_p(y_1,z')-f_p(0,z')}\le C_1\abs{y_1}(\abs{y_1}+\norm{z'}).
\]
On $EB_R$ the right side is at most $C_2aR^2$. Together with \eqref{eq:tangent-nonnegative}, this gives
$G_{p,E}\ge-aR-C_2aR^2$ on $B_R$. We next bound the positive part from below. Suppose first that $a>0$. By complex linearity, there is $\zeta_*$ with $\norm{\zeta_*}=R$ such that $(E\zeta_*)_1=aR$ is real and positive. At this vector, we have $G_{p,E}(\zeta_*)\ge aR$. Thus the absolute size of the negative part is at most $(1+C_2R)F_{p,E}(R)$. If $a=0$, then the whole restriction is nonnegative. This proves \eqref{eq:positive-maximum}. We note that the estimate is uniform also for planes approaching tangency.

We now compare the maximum with the jet polynomial. In real coordinates on $\C^k$, write
\[
 G_{p,E}(\zeta)=P_{p,E}(\zeta)+\mathcal R_{p,E}(\zeta),\qquad
 P_{p,E}(\zeta)=\sum_{1\le\abs\alpha\le M}c_\alpha(p,E)\zeta^\alpha.
\]
The remainder is bounded by $C_3\norm\zeta^{M+1}$, uniformly in the pair. Since the order of contact at every pair is at most the maximal order, no pair $(p,E)$ has all coefficients zero. Continuity of the coefficients and compactness therefore give
\[
 \sum_{1\le\abs\alpha\le M}\abs{c_\alpha(p,E)}\ge c_1>0.
\]
We next compare the polynomial with the sum of the absolute values of its coefficients. After the rescaling $\zeta=R\eta$, norm equivalence on the finite-dimensional space of polynomials gives a uniform $K$ with
\[
 K^{-1}J_{p,E}(R)\le\norm{P_{p,E}}_{L^\infty(B_R)}\le KJ_{p,E}(R).
\]
For $R\le1$, we have $J_{p,E}(R)\ge c_1R^M$. After decreasing $R_*$, the remainder is at most half of this lower bound for the polynomial. Combining these estimates with \eqref{eq:positive-maximum} proves \eqref{eq:jet-comparison}. We point out that no positivity of the Taylor polynomial has been used.
\end{proof}

The following elementary lemma converts the jet comparison into a two-depth estimate for the inverse function of the jet polynomial.

\begin{lemma}\label{lem:inverse-polynomial}
Let $J(R)=\sum_{j=1}^M a_jR^j$, where $a_j\ge0$ and $\sum_j a_j>0$, and let $\sigma=J^{-1}$. Then
\[
 \frac{\sigma(r)}{\sigma(u)}\le(r/u)^{1/M}\quad(0<r\le u),\qquad
 \sigma(br)\le b\sigma(r)\quad(b\ge1).
\]
\end{lemma}
\begin{proof}
For $0<c\le1$, the nonnegativity of the coefficients gives $J(cR)\ge c^MJ(R)$. Put $R=\sigma(u)$ and $c=(r/u)^{1/M}$, and apply the increasing inverse to the resulting inequality. This proves the first inequality. The second inequality follows in the same way from $J(bR)\ge bJ(R)$.
\end{proof}

We are now ready to prove the two-depth contraction, which is the main step in the sufficient direction.

\begin{proposition}[Two-depth contraction]\label{prop:two-depth}
Under the standing assumptions of this section, suppose that $M=M_k(\Omega)<\infty$. Then, after decreasing $r_0$, there is $C>0$, independent of $p$, $V$, $r$ and $u$, such that
\begin{equation}\label{eq:two-depth}
 R_\Omega(p_r;V)\le C(r/u)^{1/M}R_\Omega(p_u;V),
 \qquad 0<r\le u\le r_0,\quad V\in\Gr_\C(k,n).
\end{equation}
Consequently, there is $C'>0$, independent of $p$, $v$, $s$ and $t$, such that
\begin{equation}\label{eq:finite-type-expansion}
 \Hs(p;v)\le C'e^{-(t-s)/M}\Ht(p;v),\qquad 0\le s\le t.
\end{equation}
\end{proposition}
\begin{proof}
We first fix the scale. Lemma~\ref{lem:finite-jets} gives $\inf_{p,E}F_{p,E}(R_*)>0$. We choose $r_0$ below this infimum and below the normal collar radius. By \eqref{eq:sublevel-radius}, the first contact of every affine plane based at $p_r$ then lies inside the coordinate neighborhood. With the notation
\[
 \tau_{p,E}(r)=R_\Omega(p_r;\operatorname{image}E),
\]
continuity and first contact give $F_{p,E}(\tau_{p,E}(r))=r$. Thus $\tau_{p,E}(r)$ is the first radius at which $F_{p,E}$ reaches the value $r$.

We now compare this radius with the inverse of the jet polynomial. Let $\sigma_{p,E}=J_{p,E}^{-1}$. The jet comparison gives
\[
 \sigma_{p,E}(r/C)\le\tau_{p,E}(r)\le\sigma_{p,E}(r/c).
\]
The second inequality of Lemma~\ref{lem:inverse-polynomial}, applied directly and also after replacing $r$ by $r/b$, compares both sides uniformly with $\sigma_{p,E}(r)$. Its first inequality, applied at $r$ and $u$, then proves \eqref{eq:two-depth}.

To pass from planes to directions, we take the supremum over all planes containing a fixed vector $v\in\C^n$. This gives
\[
 \qk(p_u;v)\le C(r/u)^{1/M}\qk(p_r;v).
\]
Finally, we put $u=r_0e^{-s}$ and $r=r_0e^{-t}$, and we use \eqref{eq:collar-vector} at both depths. This gives \eqref{eq:finite-type-expansion} and completes the proof.
\end{proof}

Integrating \eqref{eq:finite-type-expansion} in the lower depth gives the Hardy inequality
\begin{equation}\label{eq:finite-type-hardy}
 \int_0^tH_s^{(k)}(p;v)\,ds\le C'MH_t^{(k)}(p;v).
\end{equation}
Together with Lemma~\ref{lem:normal-splitting}, Theorem~\ref{thm:hardy} and Lemma~\ref{lem:compact-interior}, this inequality proves Gromov hyperbolicity. Alternatively, \eqref{eq:finite-type-expansion} and the normal splitting verify directly the hypotheses of the version of the expanding-collar theorem for domains. We refer the reader to \cite[Theorem~3.1]{WZ24}.

\subsection{Infinite contact and thick quadrilaterals}
In this subsection, we prove the converse direction. We show that infinite contact of an affine plane rules out Gromov hyperbolicity. The argument uses centered balls in one affine plane, and it does not require an analytic disc in the boundary.

We first show that infinite contact produces a plane in the complex tangent space whose slice graph is flat to infinite order at the base point but is not identically zero near it.

\begin{lemma}\label{lem:infinite-plane}
Under the standing assumptions of this section, suppose that $M_k(\Omega)=\infty$. Then there are $p\in\partial\Omega$, a complex $k$-plane $V\subset T_p^\C\partial\Omega$ and a complex linear isometry $E:\C^k\to V$ such that, in \eqref{eq:smooth-graph}, the function
\[
 h(\zeta)=f_p\bigl(0,(E\zeta)'\bigr)
\]
is nonnegative, vanishes to infinite order at zero, and is not identically zero on any neighborhood of zero. For this pair, the radius
$\tau(r)=R_\Omega(p_r;V)$ is nondecreasing for small $r>0$, tends to zero as $r\downarrow0$, and, for every integer $N\ge1$, satisfies
\begin{equation}\label{eq:infinite-radius}
 \tau(r)\ge c_Nr^{1/N}
\end{equation}
at all sufficiently small depths.
\end{lemma}
\begin{proof}
We choose pairs $(p_j,E_j)$ with contact orders tending to infinity, and we pass to a limit in the compact parameter space of boundary points and frames. Each derivative of fixed order vanishes at all but finitely many of these pairs, and therefore it vanishes at the limiting pair. First-order vanishing and complex linearity place the limiting plane $V$ in the complex tangent space.

We now consider two cases. If the slice graph of the limiting pair is not identically zero near the origin, we keep this pair. Otherwise, the affine plane $P=p+V$ contains a relatively open ball of $\partial\Omega$. By the support condition, $P\cap\Omega=\varnothing$. The set $S=P\cap\partial\Omega$ is compact, and its relative interior $O$ is nonempty. We choose a point $q$ in the boundary of $O$ relative to $P$. Then $q\in S$. On the other hand, no relative neighborhood of $q$ is contained in $S$, since $O$ is the entire relative interior. All derivatives of the restriction to $P$ of a defining function vanish on $O$, and hence they vanish at $q$. Thus $P$ is infinitely tangent at $q$, and the slice graph there has the required properties.

It remains to prove the assertions about the radius. For the pair obtained above, $h\ge0$ by the support condition, and the centered slice at a given depth is the sublevel set $\{h<r\}$. Hence its inradius is nondecreasing in the depth. Since $h$ is not identically zero on any neighborhood of zero, we have $\tau(r)\to0$. Otherwise a fixed ball would have $h=0$ on it. Infinite flatness gives $h(\zeta)\le C_N\norm\zeta^N$ on a sufficiently small ball, and this proves \eqref{eq:infinite-radius}.
\end{proof}

The next lemma shows that the slow decay of this radius produces, at arbitrarily large depths, intervals of prescribed logarithmic length on which the radius is almost constant.

\begin{lemma}\label{lem:slow-interval}
Under the standing assumptions of this section, let $\tau$ be the radius in Lemma~\ref{lem:infinite-plane}, and fix a small $r_*>0$. For every $L>0$, $0<\eta<1$ and prescribed sufficiently large $T$, there is $a\ge T$ such that
\begin{equation}\label{eq:slow-interval}
 \tau(r_*e^{-(a+L)})\ge(1-\eta)\tau(r_*e^{-a}).
\end{equation}
\end{lemma}
\begin{proof}
The function $g(t)=\log(1/\tau(r_*e^{-t}))$ is nondecreasing and eventually nonnegative. For every $N$, \eqref{eq:infinite-radius} gives an estimate $g(t)\le t/N+C_N$ on a tail. Hence $g(t)/t\to0$. Suppose, to get a contradiction, that \eqref{eq:slow-interval} fails for every $a\ge T$. Then $g(a+L)-g(a)>-\log(1-\eta)>0$ at every such depth. Iterating this inequality along $a=T+jL$ would give a positive linear lower bound for the growth of the function, which contradicts the fact that its quotient by the depth tends to zero. This proves the lemma.
\end{proof}

We now construct quasigeodesic quadrilaterals which are not uniformly thin. This proves that infinite contact rules out Gromov hyperbolicity.

\begin{proposition}\label{prop:infinite-not-gh}
Under the standing assumptions of this section, if $M_k(\Omega)=\infty$, then $(\Omega,\dk)$ is not Gromov hyperbolic.
\end{proposition}
\begin{proof}
We use the pair given by Lemma~\ref{lem:infinite-plane}, and we translate coordinates so that $p=0$ and $p_r=-r\nu$. Fix an integer $j\ge2$. We apply Lemma~\ref{lem:slow-interval} with $L=2j$ and $\eta=e^{-j}$, choosing $T$ so large that all points considered below lie in the fixed graph neighborhood. We put
\[
 r_a=r_*e^{-a},\quad r_b=r_ae^{-2j},\quad r_c=r_ae^{-j},
 \quad R=\tau(r_a),\quad c=1-2\eta.
\]
By the choice of the interval and the monotonicity of the radius, we have $\tau(r)\ge(1-\eta)R$ for $r_b\le r\le r_a$. We choose $\xi\in V$ with $\norm\xi=R$ and $-r_a\nu+\xi\in\partial\Omega$. Then $h(\xi)=r_a$. We consider the normal sides
\[
 A(u)=-r_ae^{-u}\nu,\qquad B(u)=-r_ae^{-u}\nu+c\xi,
 \qquad 0\le u\le j,
\]
and the caps
\[
 \Gamma_0(\lambda)=-r_a\nu+\lambda\xi,\qquad
 \Gamma_1(\lambda)=-r_c\nu+\lambda\xi,
 \qquad 0\le\lambda\le c.
\]
These four curves form the quadrilateral. We first estimate the normal sides from below. The support at zero has equation $z_1=0$. It gives the distance lower bounds
\[
 \dk(A(u),A(u'))\ge\abs{u-u'},\qquad
 \dk(B(u),B(u'))\ge\abs{u-u'}.
\]
The speed of $A$ is one by \eqref{eq:normal-support}, so $A$ is a geodesic.

We next estimate the speed of the second normal side. Since $cR<(1-\eta)R\le\tau(r_b)$, we have $h(c\xi)<r_b$. Hence, for $r_c\le r\le r_a$,
\[
 r-h(c\xi)\ge r-r_b\ge(1-e^{-j})r\ge r/2.
\]
We now convert this vertical clearance into a Euclidean distance. A point with vertical clearance $b$ in a graph of Lipschitz constant $K_0$ has distance at least $b/\sqrt{1+K_0^2}$ from that graph. This follows by minimizing $a^2+\max\{b-K_0a,0\}^2$ over $a\ge0$. Once the construction is sufficiently small, the boundary points outside the graph patch are uniformly farther away. Thus
\[
 \delta_\Omega(-r\nu+c\xi)\ge c_0r
\]
with $c_0$ independent of $j$. Equation~\eqref{eq:elementary} bounds the speed of $B$ by $c_0^{-1}$. Therefore both normal sides are $(C_*,0)$-quasigeodesics for a fixed $C_*\ge1$.

We now turn to the caps. Parametrize the caps by $s=-\log(1-\lambda)$. The centered complex $k$-ball of radius $R$ at $-r_a\nu$ lies in $\Omega$, so
\[
 \qk(\Gamma_0(\lambda);\xi)\le(1-\lambda)^{-1}.
\]
For the lower bound, take a support through $-r_a\nu+\xi$. It satisfies
\[
 \phi_H(\Gamma_0(\lambda))=(\lambda-1)\ell_H(\xi),\qquad \ell_H(\xi)\ne0.
\]
Lemma~\ref{lem:global-support} gives the matching lower distance bound. Thus $\Gamma_0$ is a geodesic in the parameter $s$, of length $\ell_j=j-\log2$.

The second cap is treated in the same way. At depth $r_c$ there is a centered $k$-ball of radius at least $(1-\eta)R$. Hence, for $\lambda\le1-2\eta$,
\[
 \qk(\Gamma_1(\lambda);\xi)\le\frac1{1-\eta-\lambda}\le\frac2{1-\lambda}.
\]
For the lower bound, we locate the first boundary point on the ray. The ray $-r_c\nu+\lambda\xi$ first meets the boundary at some $\lambda_\bullet\in[1-\eta,1]$. The lower endpoint of this interval follows from the centered ball, and the upper endpoint from $h(\xi)=r_a>r_c$. A complex support at that boundary point gives, for $0\le\lambda_1<\lambda_2\le c$,
\[
 \dk(\Gamma_1(\lambda_1),\Gamma_1(\lambda_2))
 \ge\log\frac{\lambda_\bullet-\lambda_1}{\lambda_\bullet-\lambda_2}
 \ge\log\frac{1-\lambda_1}{1-\lambda_2}.
\]
Consequently, $\Gamma_1$ is a $(2,0)$-quasigeodesic in the same parameter.

Finally, we show that the midpoint of the first cap is far from the other three sides. Let $U_j$ be the midpoint of $\Gamma_0$, so that its parameter is
$\lambda_j=1-\sqrt{2\eta}$, and put $h_j=\ell_j/2$. The support $z_1=0$ gives
\[
 \dk(U_j,\Gamma_1)\ge j,\qquad
 \dk(U_j,A(u)),\ \dk(U_j,B(u))\ge u.
\]
The triangle inequality and the upper bounds on the normal sides also give
\[
 \dk(U_j,A(u)),\ \dk(U_j,B(u))\ge h_j-C_*u.
\]
Since $\max\{u,h_j-C_*u\}\ge h_j/(C_*+1)$, we obtain
\[
 \dist_{\dk}(U_j,A\cup B\cup\Gamma_1)
 \ge\min\left\{j,\frac{j-\log2}{2(C_*+1)}\right\}\longrightarrow\infty.
\]
All four sides have fixed quasigeodesic constants, and the distance from a point of one side to the union of the other three tends to infinity. This contradicts the uniform thinness of such quadrilaterals in a Gromov hyperbolic space. Hence the proof is complete.
\end{proof}

\begin{proof}[Proof of Theorem~\ref{thm:planes}]
We first prove the equivalence of the four conditions. Proposition~\ref{prop:two-depth} gives the normal expansion (ii) from the finite affine-plane contact (i). Integration of the expansion inequality gives the Hardy inequality (iii). By Lemma~\ref{lem:quasitriangle}, the map $(t,p,v)\mapsto H_t^{(k)}(p;v)$ is continuous. Hence Lemma~\ref{lem:normal-splitting}, Theorem~\ref{thm:hardy} and Lemma~\ref{lem:compact-interior} give Gromov hyperbolicity (iv) from the Hardy inequality. Proposition~\ref{prop:infinite-not-gh} proves the converse, namely that (iv) implies (i).

We now prove the optimality of the exponent. Choose a pair $(p,V)$ of contact order $M=M_k(\Omega)$. Such a pair exists because the finite contact orders are integers with maximum $M$. At this pair, the jet polynomial has no terms of degree below $M$, and therefore
\begin{equation}\label{eq:maximal-plane-radius}
 R_\Omega(p_r;V)\asymp r^{1/M}.
\end{equation}
For $0\ne v\in V$, this gives
\[
 \qk(p_r;v)\le Cr^{-1/M}\norm v.
\]
If $M>1$, the plane is complex tangent at $p$, so $v$ is a real boundary tangent vector. Lemma~\ref{lem:normal-splitting} transfers this upper bound to $H_t^{(k)}(p;v)$. On the other hand, an expansion with $\lambda>1/M$, applied with initial depth zero, would force $H_t^{(k)}(p;v)\ge ce^{\lambda t}$. This is a contradiction.

It remains to treat the case $M=1$. In this case $k=n$. Indeed, for $k<n$ there is a complex $k$-plane in the complex tangent hyperplane, and its contact order is at least two. For $k=n$, the gauge is $\norm v/\delta_\Omega$ and $H_t^{(n)}(p;v)\asymp e^t\norm v$. This proves optimality in the remaining case and completes the proof.
\end{proof}

\begin{remark}[Finite regularity]\label{rem:finite-regularity}
The sufficient direction needs much less boundary regularity. For a fixed integer $M\ge1$, it needs only a $C^{\max\{2,M\}}$ boundary, the support condition \eqref{eq:global-support}, and a nonzero affine $k$-plane jet of order at most $M$ at each boundary-plane pair. Compactness gives a uniform positive lower bound for the sum of the absolute jet coefficients. Uniform continuity of the order-$M$ derivatives gives a uniform $o(R^M)$ remainder, and this remainder is absorbed by $J_{p,E}(R)\ge cR^M$ as in Lemma~\ref{lem:finite-jets}. The normal splitting needs only $C^2$ regularity. Thus the same two-depth estimate gives Gromov hyperbolicity for $d_\Omega^{(k)}$.
\end{remark}

\section{Power geometry from Lipschitz Gromov hyperbolicity}\label{sec:lipschitz}
In this section, we prove Theorem~\ref{thm:lipschitz}. From now until the end of its proof, $\Omega\Subset\C^n$ is a $\C$-convex domain with Lipschitz boundary, $(\Omega,k_\Omega)$ is Gromov hyperbolic, and a finite Lipschitz atlas is fixed as in \eqref{eq:lipschitz-graph} below. We shall refer to these three requirements as the standing assumptions of this section.

The qualitative input that we take from Gromov hyperbolicity is the absence of nontrivial affine complex discs in $\partial\Omega$. This was proved by Bracci, Gaussier, Nikolov and Thomas. For $C^1$ boundaries it is due to Zimmer \cite[Theorem~1.4]{Zim17}. Our aim is to turn this qualitative statement into a quantitative one. We first show that the directional radii contract by a fixed factor between two depths whose ratio is one fixed logarithmic scale, and then we iterate this one-step contraction. The precise statement that we shall use is the following.

\begin{result}[Bracci--Gaussier--Nikolov--Thomas]\label{res:no-affine-discs}
A bounded $\C$-convex domain with Lipschitz boundary whose Kobayashi distance is Gromov hyperbolic has no nontrivial affine complex disc in its boundary. This is \cite[Proposition~4.2]{BGNT24}.
\end{result}
\subsection{Support-adapted charts}
In this subsection, we fix the Lipschitz charts in which all the estimates of this section are made, and we prove three lemmas about them. The first lemma shows that, on the two-plane spanned by the inward chart direction and a direction in a complex support, the domain is again the region above a graph that is Lipschitz in one real variable. The second lemma shows that the directional radii along such frames tend to zero uniformly and are monotone at small depths. The third lemma compares the height above the graph with the boundary distance.

Fix a finite Lipschitz atlas whose smaller patches cover $\partial\Omega$. In unitary affine coordinates, each larger chart has the form
\begin{equation}\label{eq:lipschitz-graph}
 \Omega\cap U=\{(z_1,z')\in U:\Ree z_1>g(\Imm z_1,z')\},
\end{equation}
where the graph functions have Lipschitz constant at most $\mathfrak L\ge1$. Let $p$ be a point of a smaller patch. We recenter the coordinates so that $p=0$, and we let $e=(1,0,\ldots,0)$ be the inward chart direction. An \emph{admissible frame} is a quadruple $(p,e,H_p,v)$, where $H_p=p+V_p$ is a complex support and $v\in V_p$ is a unit vector. Since $p+re\in\Omega$ for small $r>0$, whereas a complex support does not meet the domain, we have $e\notin V_p$. Consequently, the map
\[
 \Pi(\xi,\zeta)=p+\xi e+\zeta v
\]
is a complex affine embedding of $\C^2$.

Our first lemma says that, in the two complex variables of this embedding, the domain is the region above the graph of a function that is Lipschitz in the imaginary part of the first variable, on a bidisc whose size does not depend on the frame.

\begin{lemma}\label{lem:oblique-chart}
Under the standing assumptions of this section, there is a bidisc centered at $0\in\C^2$, independent of the admissible frame, on which
\begin{equation}\label{eq:oblique-graph}
 \Pi(x+iy,\zeta)\in\Omega\quad\Longleftrightarrow\quad x>F(y,\zeta),
\end{equation}
where $F$ is a real function satisfying
\begin{equation}\label{eq:oblique-lipschitz}
 F(0,0)=0,\qquad F(0,\zeta)\ge0,\qquad
 \abs{F(y,\zeta)-F(\widetilde y,\zeta)}\le \mathfrak L\abs{y-\widetilde y}.
\end{equation}
\end{lemma}
\begin{proof}
Write $v=(v_1,v')$. We substitute the coordinates of the embedded point into \eqref{eq:lipschitz-graph}. This gives
\[
 F(y,\zeta)=g\bigl(y+\Imm(\zeta v_1),\zeta v'\bigr)-\Ree(\zeta v_1).
\]
The Lipschitz estimate in $y$ follows from the Lipschitz estimate for the graph function. We next prove the sign condition. When $x=y=0$, the point lies in $H_p$, and a complex support cannot meet $\Omega$. Hence $F(0,\zeta)\ge0$. It remains to find the bidisc. Since $\norm e=\norm v=1$, since the atlas is finite, and since the smaller patches are compactly contained in the larger ones, a bidisc independent of the frame is mapped into the chart. This proves the lemma.
\end{proof}

Write $R_{p,e,v}(r)=r_\Omega(p+re;v)$. The next lemma says that these radii tend to zero uniformly in the frame, and that at small depths each of them is a monotone function of the depth. Both facts rest on the absence of affine discs in the boundary.
\begin{lemma}\label{lem:uniform-shrinking}
Under the standing assumptions of this section,
\begin{equation}\label{eq:uniform-shrinking}
 R_{p,e,v}(r)\longrightarrow0\qquad(r\downarrow0),
\end{equation}
uniformly over all admissible frames. Moreover, there is $r_0>0$, independent of the frame, such that each function $r\mapsto R_{p,e,v}(r)$ is nondecreasing on $(0,r_0]$.
\end{lemma}
\begin{proof}
We first prove the uniform convergence. By Result~\ref{res:no-affine-discs}, the boundary contains no nontrivial affine complex disc, so Lemma~\ref{lem:shrinking} applies and gives $\RR_\Omega(z)\to0$ as $\delta_\Omega(z)\to0$. Since $\delta_\Omega(p+re)\le r$, we have $R_{p,e,v}(r)\le\sup\{\RR_\Omega(z):\delta_\Omega(z)\le r\}$. The right side does not depend on the frame, and it tends to zero with the depth. This proves \eqref{eq:uniform-shrinking}, uniformly in the frame.

We now prove the monotonicity. Using \eqref{eq:uniform-shrinking}, we choose $r_0$ so small that every centered disc in the direction of the frame, up to its first contact with the boundary, stays inside the uniform bidisc of Lemma~\ref{lem:oblique-chart}. In the coordinates of that lemma, the centered disc at depth $r$ is the largest centered disc contained in the sublevel set $\{F(0,\zeta)<r\}$. These sublevel sets increase with the depth. Hence the radii of the centered discs are nondecreasing in the depth, as claimed.
\end{proof}

Points near the boundary have to be written in the form used by the frames. The next lemma does this in a chosen chart and compares the height above the graph with the Euclidean boundary distance.

\begin{lemma}\label{lem:height-distance}
Under the standing assumptions of this section, there is $\delta_1>0$ such that every $z\in\Omega$ with $\delta_\Omega(z)<\delta_1$ can be written in one of the chosen charts as $z=p+re$, where $p\in\partial\Omega$ lies in a smaller patch, $e$ is the inward chart direction, and
\begin{equation}\label{eq:height-distance}
 \delta_\Omega(z)\le r\le\sqrt{1+\mathfrak L^2}\,\delta_\Omega(z).
\end{equation}
\end{lemma}
\begin{proof}
In the chart coordinates, write $z=(x+iy,z')$, and take $p=(g(y,z')+iy,z')$ and $r=x-g(y,z')$. Thus the boundary point is reached from the given point by moving against the inward chart direction until the graph is met. The first inequality is immediate, because this boundary point is at Euclidean distance equal to the height from the given point. For the second inequality, we estimate the distance from the given point to an arbitrary point of the graph. By the Lipschitz bound on the graph function, a graph point whose horizontal displacement from $(y,z')$ has size $a$ is at Euclidean distance at least
\[
 \sqrt{a^2+\max\{0,r-\mathfrak La\}^2}
\]
from $z$. The minimum of this quantity over $a\ge0$ is $r/\sqrt{1+\mathfrak L^2}$. This gives the second inequality, provided every nearest boundary point of the given point lies in the larger chart. To ensure this, we choose the chart by a Lebesgue number of the cover of $\partial\Omega$ by the smaller patches, and we shrink the collar uniformly. Then $p$ lies in a smaller patch, and every nearest boundary point of $z$ lies in the corresponding larger chart. This proves the other inequality.
\end{proof}

\subsection{The second support and a fixed-scale rectangle}
In this subsection, we prove the one-step contraction. The argument places a quadrilateral with quasigeodesic sides inside the domain, at one fixed logarithmic scale, and it uses two complex supports, namely the support at the base point of the frame and a second support at the first contact point of the centered disc. Gromov hyperbolicity forces the quadrilateral to be thin. A lower bound for the distances between its sides contradicts this thinness unless the directional radius contracts.

We first record an elementary estimate for a complex support at a boundary point that is reached by a polydisc inside the domain. The lemma holds for arbitrary domains.

\begin{lemma}[Supporting-polydisc estimate]\label{lem:polydisc-support}
Let $\Omega\subset\C^n$ be a domain, and write $D(\xi_0,t)=\{\xi\in\C:\abs{\xi-\xi_0}<t\}$. Suppose that $p\in\C^n$, that $e,v\in\C^n$ are complex linearly independent, that $\Pi(\xi,\zeta)=p+\xi e+\zeta v$, and that
\[
 \Pi(D(r,\rho_1)\times D(0,\rho_2))\subset\Omega,\qquad
 q=\Pi(r,R)\in\partial\Omega,\qquad 0<\rho_2\le R.
\]
Then, for every complex support at $q$, an affine equation of the support can be normalized on this two-plane so that
\begin{equation}\label{eq:polydisc-support}
 \ell_q(\Pi(\xi,\zeta))=\zeta-R-A(\xi-r),\qquad
 \abs A\rho_1+\rho_2\le R.
\end{equation}
\end{lemma}
\begin{proof}
We restrict an affine equation of the support to the two-plane. We claim that the coefficient of $\zeta$ in the restricted affine functional cannot vanish. Indeed, if it vanished, then the functional would not depend on the second variable, so its vanishing at $(r,R)$ would force its vanishing at $(r,0)$. The latter is an interior point of the domain, which the support cannot contain. Dividing by this coefficient, we obtain the normalization in the first part of the display. It remains to bound the slope. The image of the product of discs under the normalized functional is the open disc with center $-R$ and radius $\rho_2+\abs A\rho_1$. Since the support avoids $\Omega$, the normalized functional does not vanish on the product of discs, i.e., this disc avoids zero. Its radius is therefore at most $R$. This proves the lemma.
\end{proof}

We now fix the constants of the contraction argument. Put
\begin{equation}\label{eq:c0}
 c_0=\frac1{4(1+\mathfrak L)},\qquad \Lambda=c_0^{-1}.
\end{equation}
Let $D_0$ be a thinness constant for $(\Lambda,1)$-quasigeodesic quadrilaterals in $(\Omega,k_\Omega)$. Such a constant exists by Gromov hyperbolicity and the stability of quasigeodesics. Choose the logarithmic scale
\begin{equation}\label{eq:rectangle-scale}
 L>\max\{\log20,\log(4/c_0),4D_0+4\},\qquad
 h=e^{-L},\qquad\theta=e^{-2L}=h^2.
\end{equation}
Then $h<1/20$, $c_0^{-1}h<1/4$, and $c_0^{-1}\theta<h/4$. We shall use these three inequalities, and the lower bound on the scale by the thinness constant, in the proof below.

We are now ready to prove the one-step contraction, which is the heart of this section.

\begin{proposition}[One-step contraction]\label{prop:one-step}
Assume the standing assumptions of this section, and let $\theta$ be as in \eqref{eq:rectangle-scale}. After decreasing $r_0$ once, by an amount independent of the frame, we have
\begin{equation}\label{eq:one-step}
 R_{p,e,v}(\theta r)\le(1-\theta)R_{p,e,v}(r)
\end{equation}
for every admissible frame and every $0<r\le r_0$.
\end{proposition}
\begin{proof}
Fix an admissible frame and omit its subscripts. Suppose, to get a contradiction, that $R(\theta r)\ge(1-\theta)R(r)$. Put
\[
 R=R(r),\qquad a=\theta r,\qquad b=2a,\qquad R_-=(1-\theta)R.
\]
After multiplying $v$ by a scalar of modulus one, we may assume that the first contact of the centered disc at depth $r$ with the boundary occurs at $q=\Pi(r,R)\in\partial\Omega$. We first show that a large polydisc in the two-plane lies in the domain. In the coordinates of the oblique chart, the assumed inequality says that $F(0,\zeta)<a$ for $\abs\zeta<R_-$. We use the Lipschitz estimate in the first variable to thicken this disc. If $b\le u\le r$, $\abs\lambda<c_0u$ and $\abs\zeta<R_-$, then Lemma~\ref{lem:oblique-chart} gives
\[
 F(\Imm\lambda,\zeta)<a+\mathfrak Lc_0u
 \le(\tfrac12+\mathfrak Lc_0)u<(1-c_0)u<\Ree(u+\lambda).
\]
All these points lie in the chart, by uniform shrinking. Consequently,
\begin{equation}\label{eq:rectangle-polydiscs}
 \Pi(D(u,c_0u)\times D(0,R_-))\subset\Omega,\qquad b\le u\le r.
\end{equation}
We next normalize the two supports. Apply Lemma~\ref{lem:polydisc-support} at $u=r$. The support through $q$ has an affine equation
\begin{equation}\label{eq:second-support}
 \ell_q(\Pi(\xi,\zeta))=\zeta-R-A(\xi-r),\qquad
 \abs A r\le c_0^{-1}\theta R.
\end{equation}
The original support at $p$ can be normalized globally so that
\begin{equation}\label{eq:first-support}
 \ell_p(\Pi(\xi,\zeta))=\xi.
\end{equation}
Consider the rectangle with vertices
\[
 \Pi(r,0),\quad \Pi(r,(1-h)R),\quad
 \Pi(b,(1-h)R),\quad \Pi(b,0).
\]
Since $(1-h)R<R_-$, the polydiscs above show that the rectangle lies in $\Omega$.

We now show that the four sides of the rectangle are quasigeodesics with uniform constants. Consider first the two vertical sides. On either of them, the first factor of \eqref{eq:rectangle-polydiscs} gives $\kappa_\Omega(\Pi(u,\zeta);e)\le(c_0u)^{-1}$. We parametrize the side by logarithmic depth, $u=re^{-s}$. The Kobayashi speed of this parametrization is then at most $\Lambda$. On the other hand, \eqref{eq:k-support} applied with $\ell_p$ bounds the Kobayashi distance between two points of the side below by $\abs{s-s'}/4$. Thus each vertical side is a $(\Lambda,0)$-quasigeodesic.

Consider next the horizontal side at height $u\in\{b,r\}$. We parametrize it by
\[
 \sigma_u(s)=\Pi\bigl(u,R(1-e^{-s})\bigr),\qquad 0\le s\le L.
\]
By the second factor of the polydiscs above, the remaining centered radius in direction $v$ along this side is at least $R(e^{-s}-\theta)$. Hence the Kobayashi speed satisfies
\begin{equation}\label{eq:horizontal-speed}
 \kappa_\Omega(\sigma_u(s);\sigma_u'(s))
 \le\frac{e^{-s}}{e^{-s}-\theta}=\frac1{1-\theta e^s}\le2.
\end{equation}
This is the upper bound. For the lower bound we use the second support. It satisfies
\[
 \ell_q(\sigma_u(s))=-Re^{-s}+A(r-u),\qquad
 \abs{A(r-u)}\le c_0^{-1}\theta R\le\tfrac14Re^{-s}.
\]
Thus
\[
 \tfrac34Re^{-s}\le\abs{\ell_q(\sigma_u(s))}\le\tfrac54Re^{-s}.
\]
For $s\le s'$, the global support estimate gives
\begin{equation}\label{eq:horizontal-distance}
 k_\Omega(\sigma_u(s),\sigma_u(s'))
 \ge\tfrac14(s'-s)-\tfrac14\log(5/3).
\end{equation}
Together with \eqref{eq:horizontal-speed}, this makes each horizontal side a $(\Lambda,1)$-quasigeodesic. We emphasize that the constants are independent of the frame and of $r,R$.

Finally, we show that the quadrilateral is not thin. Take the logarithmic midpoint $z_* =\Pi(\sqrt{br},0)$ of the left vertical side. We estimate its distance to the three other sides. Let $w$ be a point on either horizontal side. Between the midpoint and that side, the height changes by the factor equal to the square root of the ratio of the two depths, so the first support, i.e., \eqref{eq:first-support}, gives
\begin{equation}\label{eq:separation-horizontal}
 k_\Omega(z_*,w)\ge\tfrac18\log(r/b)
 =\tfrac14\bigl(L-\tfrac12\log2\bigr)>D_0.
\end{equation}
Suppose now that $w=\Pi(u,(1-h)R)$ lies on the opposite vertical side. Then \eqref{eq:second-support} gives
\[
 \abs{\ell_q(z_*)}\ge R(1-c_0^{-1}\theta),\qquad
 \abs{\ell_q(w)}\le R(h+c_0^{-1}\theta).
\]
Since $c_0^{-1}\theta<h/4$ and $h<1/20$,
\[
 \frac{1-c_0^{-1}\theta}{h+c_0^{-1}\theta}
 \ge\frac{1-h/4}{(5/4)h}>e^{L-1}.
\]
Therefore the global support estimate for the second support gives $k_\Omega(z_*,w)>(L-1)/4>D_0$. Thus the midpoint of the left vertical side is farther than $D_0$ from each of the other three sides of the quadrilateral. This contradicts the definition of $D_0$ as a thinness constant for quasigeodesic quadrilaterals. So our starting assumption must be wrong, and the proposition is proved.
\end{proof}

\subsection{Iteration and arbitrary directions}
In this subsection, we iterate the one-step contraction to obtain the relative power law along admissible frames, and we then pass from the directions in a complex support to arbitrary directions. The passage to arbitrary directions goes through the infinitesimal Carath\'eodory metric. At the end of the subsection we prove Theorem~\ref{thm:lipschitz}, and we derive from it the geometric half of Theorem~\ref{thm:main}.

Our next result is the relative two-depth estimate along admissible frames. It follows from the one-step contraction by iteration, and it gives the exponent of the power law.

\begin{proposition}\label{prop:power-contraction}
Assume the standing assumptions of this section, let $\theta$ be as in \eqref{eq:rectangle-scale}, let $r_0$ be as in Proposition~\ref{prop:one-step}, and put
\begin{equation}\label{eq:alpha}
 \alpha=\frac{\log(1-\theta)}{\log\theta}.
\end{equation}
Then $0<\alpha<1/2$ and, for every admissible frame,
\begin{equation}\label{eq:power-contraction}
 R_{p,e,v}(s)\le\frac1{1-\theta}(s/r)^\alpha R_{p,e,v}(r),
 \qquad 0<s\le r\le r_0.
\end{equation}
In particular, $R_{p,e,v}(r)\le C_\alpha r^\alpha$ for $0<r\le r_0$, where $C_\alpha=(1-\theta)^{-1}r_0^{-\alpha}\diam\Omega$.
\end{proposition}
\begin{proof}
We omit the subscripts of the frame. Given the two depths, choose the integer $j\ge0$ with $\theta^{j+1}r<s\le\theta^jr$. The monotonicity of the radii at small depths, together with repeated application of Proposition~\ref{prop:one-step}, gives
\[
 R(s)\le R(\theta^jr)\le(1-\theta)^jR(r).
\]
Since $\theta^\alpha=1-\theta$ and $(s/r)^\alpha>(1-\theta)^{j+1}$, this display proves \eqref{eq:power-contraction}. For the one-depth bound, we fix the larger depth at $r_0$ and use $R(r_0)\le\diam\Omega$. Finally, we prove the bound on the exponent. By \eqref{eq:rectangle-scale}, we have $1-\theta>\sqrt\theta=h$. Taking logarithms and dividing by the negative number $\log\theta$ gives $\alpha<1/2$.
\end{proof}

The exponent $\alpha$ depends only on the Lipschitz constant $\mathfrak L$ of the atlas and on the hyperbolicity constant of $(\Omega,k_\Omega)$, the latter through $D_0$. It is not explicit in terms of the boundary geometry.

The relative estimate above concerns directions in a complex support. The next proposition removes this restriction. It gives a lower bound for the infinitesimal Carath\'eodory metric in every direction, with the same exponent.

\begin{proposition}\label{prop:caratheodory-power}
Under the standing assumptions of this section, let $\alpha$ be as in \eqref{eq:alpha}. There are $c,\delta_0>0$ such that
\begin{equation}\label{eq:caratheodory-power}
 \gamma_\Omega(z;w)\ge c\delta_\Omega(z)^{-\alpha}\norm w,
 \qquad\delta_\Omega(z)<\delta_0,\quad w\in\C^n.
\end{equation}
\end{proposition}
\begin{proof}
Let the point be close to the boundary. We use Lemma~\ref{lem:height-distance} to write $z=p+re$, where
$\delta_\Omega(z)\le r\le\sqrt{1+\mathfrak L^2}\delta_\Omega(z)$. Choose a complex support $H_p=p+V_p$ with an affine equation $\ell_p$ whose linear part $L_p$ satisfies $L_p(e)=1$. We split an arbitrary vector into a normal component and a tangential component. For $w\in\C^n$, put
\[
 \lambda=L_p(w),\qquad w_T=w-\lambda e\in V_p.
\]
We estimate the metric on the two components separately. The power law in the directions of the support, i.e., the one-depth bound of the preceding proposition, and \eqref{eq:invariant-comparison} give
\begin{equation}\label{eq:tangential-gamma}
 \gamma_\Omega(z;w_T)\ge\frac{\norm{w_T}}{4C_\alpha r^\alpha},
\end{equation}
where the case of the zero vector is interpreted trivially. The supporting norm representation and the same invariant comparison give
\begin{equation}\label{eq:normal-gamma}
 \gamma_\Omega(z;w)\ge\frac{\abs\lambda}{4r}.
\end{equation}
We next bound the metric on the normal component from above. The chart disc $\{p+(r+\eta)e:\abs\eta<c_0r\}$ lies in $\Omega$, because
$F(\Imm\eta,0)\le \mathfrak Lc_0r<(1-c_0)r<\Ree(r+\eta)$. Therefore
\[
 \gamma_\Omega(z;\lambda e)\le\kappa_\Omega(z;\lambda e)
 \le\frac{\abs\lambda}{c_0r}\le4c_0^{-1}\gamma_\Omega(z;w).
\]
To combine the estimates, we use that the infinitesimal Carath\'eodory metric is a norm on a bounded domain. Indeed, it is a supremum of absolute values of differentials of disc-valued holomorphic maps, and scaled coordinate functions make it positive. Its triangle inequality yields
\[
 \gamma_\Omega(z;w_T)\le(1+4c_0^{-1})\gamma_\Omega(z;w).
\]
Equations~\eqref{eq:tangential-gamma} and \eqref{eq:normal-gamma} now give, for $r<1$,
\begin{align*}
 \norm w&\le\norm{w_T}+\abs\lambda\\
 &\le\bigl(4C_\alpha(1+4c_0^{-1})r^\alpha+4r\bigr)\gamma_\Omega(z;w)\\
 &\le C_1r^\alpha\gamma_\Omega(z;w).
\end{align*}
In the last step we used $r\le r^\alpha$, which holds since $\alpha<1$. The height-distance comparison \eqref{eq:height-distance} converts the height into the boundary distance and proves \eqref{eq:caratheodory-power}.
\end{proof}

\begin{proof}[Proof of Theorem~\ref{thm:lipschitz}]
The relative two-depth estimate in the theorem is Proposition~\ref{prop:power-contraction}. It remains to prove the power bound for the directional radius in every direction. Near the boundary, Proposition~\ref{prop:caratheodory-power} and \eqref{eq:invariant-comparison} give
\[
 c\delta_\Omega(z)^{-\alpha}\norm v
 \le\gamma_\Omega(z;v)\le\kappa_\Omega(z;v)
 \le\frac{\norm v}{r_\Omega(z;v)}.
\]
Canceling $\norm v$ proves the radius bound near the boundary. On the set $\{\delta_\Omega\ge\delta_0\}$, we use $r_\Omega\le\diam\Omega$ and enlarge the constant. Hence the power bound holds on the whole domain. In particular, \eqref{eq:mCconvex} holds with the real parameter $m=1/\alpha$. It also holds with any integer $m\ge1/\alpha$, after enlarging the constant. This completes the proof.
\end{proof}

We now turn to the relation between the power bound and the D'Angelo type. On a smoothly bounded $\C$-convex domain, the D'Angelo type at a boundary point equals the maximal order of contact of complex affine lines at that point \cite[Proposition~6]{NPZ11}. For convex domains this equality is due to McNeal \cite{McN92}, Boas--Straube \cite{BS92} and Yu \cite{Yu92}. The multitype version for lineally convex domains is due to Conrad \cite{Con02}. See also \cite[Proposition~8]{NPZ11}. In this article we need only a global bound on the type. The next lemma proves such a bound directly from the power law, singular curve germs included.

\begin{lemma}[Directional power bounds and full curve type]\label{lem:full-curve-type}
Let $\Omega\Subset\C^n$, $n\ge2$, be $\C$-convex with $C^\infty$ boundary. Suppose that
\[
 r_\Omega(z;v)\le C\delta_\Omega(z)^\alpha,
 \qquad z\in\Omega,\quad v\ne0,
\]
for some $C>0$ and $\alpha>0$. Then for each $\tau>0$ there is $C_\tau>0$ such that, for all $x,y\in\Omega$,
\begin{equation}\label{eq:fixed-ball-power}
 k_\Omega(x,y)\le \tau
 \quad\Longrightarrow\quad
 \norm{x-y}\le C_\tau\delta_\Omega(x)^\alpha.
\end{equation}
Consequently,
\[
 \Delta_1(\partial\Omega,p)\le\frac1\alpha,
 \qquad p\in\partial\Omega.
\]
\end{lemma}
\begin{proof}
We first convert the power bound for the radius into a lower bound for the Kobayashi metric. By \eqref{eq:invariant-comparison},
\[
 \kappa_\Omega(z;v)\ge c\,\delta_\Omega(z)^{-\alpha}\norm v.
\]
Fix $x,y$ with $k_\Omega(x,y)\le \tau$, and choose an absolutely continuous joining curve $\sigma$ of Kobayashi gauge length at most $\tau+1$. We shall show that this curve stays at a boundary distance comparable to that of its initial point. Let $p$ be a nearest boundary point to $x$. The complex tangent hyperplane $H$ at $p$ avoids $\Omega$. Choose an affine equation $\phi_H$ with linear part of Euclidean norm one. Since $x-p$ is normal to the boundary, we have $|\phi_H(x)|=\delta_\Omega(x)$. Lemma~\ref{lem:global-support}, applied to each initial segment of $\sigma$, gives
\[
 \delta_\Omega(\sigma(t))
 \le\dist(\sigma(t),H)=|\phi_H(\sigma(t))|
 \le e^{4(\tau+1)}\delta_\Omega(x).
\]
Integrating the metric lower bound along $\sigma$ and using this estimate for the boundary distance, we obtain
\[
 \norm{x-y}\le\int\norm{\sigma'(t)}\,dt
 \le c^{-1}e^{4\alpha(\tau+1)}(\tau+1)\delta_\Omega(x)^\alpha,
\]
and this proves \eqref{eq:fixed-ball-power}.

We now prove the type bound. From this point on, we argue exactly as in the part of the proof of Theorem~\ref{thm:curve-type} that follows \eqref{eq:k-ball-euclidean}, with $\alpha$ in place of $\eps$ and with \eqref{eq:fixed-ball-power} for $\tau=1$ in place of \eqref{eq:k-ball-euclidean}. The point of the unit disc used there may be taken to be $\tfrac12$, since $k_\D(0,\tfrac12)=\tfrac12\log3<1$. This argument shows that every holomorphic curve germ through a point $p\in\partial\Omega$ has contact ratio at most $1/\alpha$. Taking the supremum over all curve germs, we get $\Delta_1(\partial\Omega,p)\le1/\alpha$. This proves the lemma.
\end{proof}

We are now ready to prove the geometric half of Theorem~\ref{thm:main}, namely the equivalence of its conditions (i) and (ii).

\begin{proposition}[Gromov hyperbolicity and finite type]\label{prop:geometric-half}
Let $\Omega\Subset\C^n$, $n\ge2$, be $\C$-convex with $C^\infty$ boundary. Then $(\Omega,k_\Omega)$ is Gromov hyperbolic if and only if $\partial\Omega$ is of finite type.
\end{proposition}

\begin{proof}
Suppose first that $(\Omega,k_\Omega)$ is Gromov hyperbolic. The boundary is smooth, hence Lipschitz, so Theorem~\ref{thm:lipschitz} applies and gives $r_\Omega\le C\delta_\Omega^\alpha$. Lemma~\ref{lem:full-curve-type} then gives $\Delta_1(\partial\Omega,p)\le1/\alpha$ for every $p$. Hence the boundary is of finite type.

Conversely, suppose that $\partial\Omega$ is of finite type. We first show that the maximal order of contact of complex lines is finite. Complex affine lines are holomorphic curve germs, so every complex line has finite order of contact with $\partial\Omega$. This order is upper semicontinuous in the base point and the direction, because order at least $N$ means the vanishing of finitely many continuous derivatives. By compactness of the set of base points and directions, $M_1(\Omega)<\infty$. The support condition \eqref{eq:global-support} holds by $\C$-convexity. Therefore Theorem~\ref{thm:planes} with $k=1$ shows that $(\Omega,d^{(1)}_\Omega)$ is Gromov hyperbolic. It remains to pass from this distance to the Kobayashi distance. By \eqref{eq:invariant-comparison}, $k_\Omega$ and $d^{(1)}_\Omega$ are bi-Lipschitz equivalent. Both are proper geodesic metrics (Lemma~\ref{lem:proper}), and Gromov hyperbolicity of geodesic spaces is invariant under quasi-isometries \cite[Theorem~III.H.1.9]{BH99}. Hence $(\Omega,k_\Omega)$ is Gromov hyperbolic. This completes the proof.
\end{proof}

\begin{remark}[Zimmer's projective cone]\label{rem:projective-cone}
The Lipschitz hypothesis in Theorem~\ref{thm:lipschitz} cannot simply be removed. Let $n\ge2$ and let
\[
 \mathcal C=\{(z_0,z')\in\C\times\C^{n-1}:\Imm z_0>\norm{z'}\},\qquad
 T(z_0,z')=\left(\frac1{z_0+i},\frac{z'}{z_0+i}\right).
\]
By \cite[Theorem~1.8]{Zim17}, $(\mathcal C,k_{\mathcal C})$ is Gromov hyperbolic. Substituting $z_0=1/w_0-i$ and $z'=w'/w_0$ in the inequality defining the cone and multiplying by $\abs{w_0}^2$, we see that $T$ maps $\mathcal C$ onto
\[
 U_*=\{(w_0,w'):\abs{w_0}^2+\Imm w_0+\abs{w_0}\norm{w'}<0\},
\]
and the resulting inequality implies $\abs{w_0}+\norm{w'}<1$, so $U_*$ is bounded. The projective map identifies the complex-line sections of $U_*$ with convex line sections of $\mathcal C$, so $U_*$ is $\C$-convex, and $k_{U_*}$ is Gromov hyperbolic. For $x_t=(-it,0)$, $0<t<1$, and every unit direction $v=(0,v')$, restricting the defining inequality to the line through $x_t$ in the direction $v$ gives $r_{U_*}(x_t;v)=1-t$, whereas $\delta_{U_*}(x_t)\le t$, since $(0,0)\in\partial U_*$. Thus no positive directional power estimate holds on $U_*$, and by Theorem~\ref{thm:lipschitz} the domain $U_*$ does not have Lipschitz boundary. We also note that $\partial U_*$ contains the complex affine ball $\{0\}\times\B^{n-1}$, cf.\ \cite[Proposition~1.9]{Zim17}, since $(-i\varepsilon,w')$, $0<\varepsilon<1$, lies in $U_*$ exactly when $\norm{w'}<1-\varepsilon$, while no point with $w_0=0$ lies in $U_*$.
\end{remark}
\section{Euclidean and Gromov boundary comparison}\label{sec:comparison}
In this section, we compare the Euclidean closure of a Gromov hyperbolic domain of the kind considered in this paper with its Gromov compactification. We estimate the Gromov product of two interior points in terms of their Euclidean distance and their distances to the boundary, and we show that the identity map extends to a homeomorphism between the two compactifications. We use the directional radius defined in \eqref{eq:radius}. Throughout this section, the $m$-$\C$-convexity estimate is written in its global form
\begin{equation}\label{eq:mconvex}
 r_\Omega(z;v)\le C_r\delta_\Omega(z)^{1/m},
 \qquad z\in\Omega,\quad v\ne0,
\end{equation}
where $C_r>0$ and $m\ge1$. The global form is equivalent to the collar estimate \eqref{eq:mCconvex}, because the interior set left out by the collar is compact. The restriction $m\ge1$ is no loss of generality. Indeed, the directional radius satisfies $r_\Omega(z;v)\ge\delta_\Omega(z)$, and this inequality rules out $m<1$ as $z$ tends to the boundary.

We fix the normalization of the Kobayashi distance so that
\[
 k_{\D}(0,r)=\frac12\log\frac{1+r}{1-r}.
\]
For a base point $o\in\Omega$, the Gromov product with respect to this base point is
\[
 (x\mid y)_o=\frac12\bigl(k_\Omega(o,x)+k_\Omega(o,y)-k_\Omega(x,y)\bigr).
\]
We denote the Gromov compactification by $\overline{\Omega}^{\,G}$ and the Gromov boundary by
$\partial_G\Omega$. A visual metric with parameter $\varepsilon>0$
and base point $o$ is a metric $\rho_\varepsilon$ on $\partial_G\Omega$ such that
\begin{equation}\label{eq:visual}
 Q^{-1}e^{-\varepsilon(\xi\mid\eta)_o}
 \le \rho_\varepsilon(\xi,\eta)
 \le Q e^{-\varepsilon(\xi\mid\eta)_o}
 \qquad(\xi,\eta\in\partial_G\Omega)
\end{equation}
for some $Q\ge1$. The Gromov product of two boundary points is defined in
Subsection~\ref{sec:boundary}. It is shown in \cite[Proposition~III.H.3.21]{BH99} that a proper geodesic Gromov hyperbolic
space admits visual metrics for every sufficiently small $\varepsilon>0$.

We are now ready to state the principal result of this section. It gives a two-sided estimate of the Gromov product in terms of the Euclidean distance of two points and their distances to the boundary, with exponents that are explicit in the exponent of the global estimate above and in the constant of the logarithmic growth bound for the Kobayashi distance from the base point.

\begin{theorem}[Quantitative boundary comparison]\label{thm:boundary-comparison}
Let $\Omega\Subset\C^n$ be a $\C$-convex domain with Lipschitz boundary
such that $(\Omega,k_\Omega)$ is Gromov hyperbolic, and let $m\ge1$ and
$C_r>0$ be such that \eqref{eq:mconvex} holds. Fix $o\in\Omega$ and
$R>\max\{1,3\diam(\Omega)\}$. Then there exist $A\ge0$ and $B>0$ such
that
\begin{equation}\label{eq:logupper}
 k_\Omega(o,z)\le A+B\log\frac{R}{\delta_\Omega(z)}\qquad(z\in\Omega).
\end{equation}
Fix such $A,B$ and put
\[
 S(x,y)=\norm{x-y}+\delta_\Omega(x)+\delta_\Omega(y).
\]
Then there exists $C_0\ge0$ such that
\begin{equation}\label{eq:main}
 \frac14\log\frac{R}{S(x,y)}-C_0
 \le (x\mid y)_o
 \le Bm\log\frac{R}{S(x,y)}+C_0
 \qquad(x,y\in\Omega).
\end{equation}
Moreover, the identity map on $\Omega$ extends to a homeomorphism
\begin{equation}\label{eq:identification}
 \mathcal I:\overline{\Omega}\longrightarrow\overline{\Omega}^{\,G},
\end{equation}
and for every visual metric $\rho_\varepsilon$ satisfying
\eqref{eq:visual} there exist $c,C>0$ such that
\begin{equation}\label{eq:boundary-main}
 c\left(\frac{\norm{p-q}}{R}\right)^{\varepsilon Bm}
 \le \rho_\varepsilon\bigl(\mathcal I(p),\mathcal I(q)\bigr)
 \le C\left(\frac{\norm{p-q}}{R}\right)^{\varepsilon/4}
 \qquad(p,q\in\partial\Omega).
\end{equation}
The constant $C_0$ depends only on $\Omega$, its finite Lipschitz atlas,
$m,C_r,o,R,A,B$ and the hyperbolicity constant, and the constants $c,C$ depend in
addition on $\rho_\varepsilon$.
\end{theorem}

\begin{remark}
We point out that the hypothesis \eqref{eq:mconvex} only names the exponent. By Theorem~\ref{thm:lipschitz}, an exponent $m$ as in
Theorem~\ref{thm:boundary-comparison} always exists under the other hypotheses. The constants $A,B$ are
supplied by Lemma~\ref{lem:inward} below.
\end{remark}

The homeomorphism \eqref{eq:identification} is due to Bracci,
Gaussier, Nikolov and Thomas \cite[Proposition~4.4]{BGNT24}. The new
content of Theorem~\ref{thm:boundary-comparison} is the two-sided estimate
\eqref{eq:main}, whose exponents are explicit in $m$ and $B$. Its proof uses
complex supports in place of the real supporting hyperplanes that are used
for convex domains in \cite{LWZ22}. The proof occupies the remainder of this section. We first collect the Kobayashi estimates that we need (Lemma~\ref{lem:speed}), then we describe the Lipschitz boundary charts and the inward curves (Lemmas~\ref{lem:charts} and~\ref{lem:inward}), and finally we prove the two sides of the Gromov product estimate (Lemmas~\ref{lem:length} and~\ref{lem:lowerproduct}) and the boundary assertions (Subsection~\ref{sec:boundary}).

\subsection{Kobayashi estimates and Euclidean speed}

In this subsection, we record the estimates for the Kobayashi metric and distance that we shall use, and we derive from them a bound on the Euclidean speed of Kobayashi geodesics. We begin with the infinitesimal estimates. By \eqref{eq:invariant-comparison},
\begin{equation}\label{eq:directional}
 \kappa_\Omega(z;v)\ge\frac{\norm{v}}{4r_\Omega(z;v)}.
\end{equation}
Combining this inequality with \eqref{eq:mconvex}, we get
\begin{equation}\label{eq:metriclower}
 \kappa_\Omega(z;v)\ge\frac{1}{4C_r}\delta_\Omega(z)^{-1/m}\norm{v}.
\end{equation}
Since the Euclidean ball $B(z,\delta_\Omega(z))$ is contained in $\Omega$, we also have
\begin{equation}\label{eq:metricupper}
 \kappa_\Omega(z;v)\le\frac{\norm{v}}{\delta_\Omega(z)}.
\end{equation}

We next derive a lower bound for the Kobayashi distance in terms of the ratio of the boundary distances. Let $z,w\in\Omega$ with $\delta_\Omega(w)\le\delta_\Omega(z)$.
Choose a nearest boundary point $q$ to $w$. Since $\Omega$ is $\C$-convex, there is a complex support $H$
through $q$. We normalize $\phi_H$ so that its linear part has
Euclidean norm one. Then
$\abs{\phi_H(w)}=\dist(w,H)\le\norm{w-q}=\delta_\Omega(w)$. On the other hand, $H$
is disjoint from $\Omega$, so it does not meet the ball
$B(z,\delta_\Omega(z))$, and hence $\abs{\phi_H(z)}\ge\delta_\Omega(z)$.
Therefore \eqref{eq:k-support} gives
$k_\Omega(z,w)\ge\frac14\log(\delta_\Omega(z)/\delta_\Omega(w))$.
Interchanging the roles of $z$ and $w$, we obtain
\begin{equation}\label{eq:ratio}
 k_\Omega(z,w)\ge\frac14\left|\log\frac{\delta_\Omega(z)}{\delta_\Omega(w)}\right|,
 \qquad z,w\in\Omega.
\end{equation}
This is the argument of Nikolov \cite[Proposition~2]{Nik14}, who proves
$c_\Omega(z,w)\ge\frac14\log\bigl(\delta_\Omega(z)/(4\delta_\Omega(w))\bigr)$
for the Carath\'eodory distance $c_\Omega\le k_\Omega$.
In particular, taking the base point as one of the two points, we see that for the fixed $o$ and $R$ there is $b_1\ge0$ such that
\begin{equation}\label{eq:radiallower}
 k_\Omega(o,z)\ge\frac14\log\frac{R}{\delta_\Omega(z)}-b_1.
\end{equation}

We recall from Lemma~\ref{lem:proper} that every bounded $\C$-convex domain is a
proper geodesic metric space for $k_\Omega$. Thus any two points are joined by a distance geodesic, and
the standard Gromov compactification is available as soon as the space is Gromov hyperbolic.

The following direct comparison with a continuous weighted Euclidean norm gives the Euclidean speed estimate that we need below.

\begin{lemma}\label{lem:speed}
Let $\Omega\Subset\C^n$ be a $\C$-convex domain satisfying
\eqref{eq:mconvex}. Every $k_\Omega$-unit-speed distance geodesic
$\gamma:[a,b]\to\Omega$ is Euclidean Lipschitz and satisfies
\begin{equation}\label{eq:speed}
 \norm{\gamma'(t)}\le4C_r\delta_\Omega(\gamma(t))^{1/m}
 \quad\text{for almost every }t\in[a,b].
\end{equation}
\end{lemma}
\begin{proof}
We compare the Kobayashi distance with the length distance of a weighted Euclidean gauge. Set
\[
 w(z)=\frac{1}{4C_r}\delta_\Omega(z)^{-1/m},\qquad F(z;v)=w(z)\norm{v},
\]
and let $d_F$ be the length distance of this gauge. Integrating
\eqref{eq:metriclower} along curves gives $d_F\le k_\Omega$. Since $\Omega$ is bounded,
the weight $w$ has a positive lower bound $w_*>0$, and hence
\[
 w_*\norm{z-z'}\le d_F(z,z')\le k_\Omega(z,z').
\]
This proves that $\gamma$ is Euclidean Lipschitz.

Next we show that the continuity of $w$ implies the local expansion
\begin{equation}\label{eq:localF}
 d_F(z,z')=w(z)\norm{z-z'}+o(\norm{z-z'})\qquad(z'\to z).
\end{equation}
The straight segment gives the upper bound. For the lower bound, fix $a>0$ and
choose a small ball about $z$, compactly contained in $\Omega$, on
which $w\ge w(z)-a$. A curve that stays in this ball has $F$-length
at least $(w(z)-a)\norm{z-z'}$. A curve that leaves the ball has $F$-length bounded below by a fixed
positive constant, so such a curve cannot improve the
upper bound once $z'$ is sufficiently close to $z$. Letting $a\downarrow0$
proves \eqref{eq:localF}.

Finally, since $\gamma$ is Euclidean Lipschitz, its Euclidean derivative exists at almost every $t$. At such a parameter, $d_F(\gamma(t+h),\gamma(t))\le |h|$, because the geodesic has unit Kobayashi speed and the length distance is dominated by the Kobayashi distance. Hence \eqref{eq:localF} gives
$w(\gamma(t))\norm{\gamma'(t)}\le1$, which is \eqref{eq:speed}. This completes the proof.
\end{proof}

\subsection{Lipschitz boundary charts and inward curves}\label{sec:charts}

In this subsection, we describe the boundary geometry that the lower estimate for the Gromov product requires. We fix a finite collection of Lipschitz graph charts, and we show that the inward vertical curves in these charts are quasi-geodesics that stay close to the geodesics from the base point.
All graph coordinates in this section are real orthonormal coordinates
on $\C^n=\R^{2n}$. They are used only for Euclidean geometry.

Our first lemma fixes the charts and records the two properties of them that we shall use, namely a two-sided bound for the boundary distance of a point at a given height above the graph, and a representation in a single chart of any two points that are close to each other and to the boundary.

\begin{lemma}\label{lem:charts}
Let $\Omega\Subset\C^n$ be a domain with Lipschitz boundary, and let
$S(x,y)$ be as in Theorem~\ref{thm:boundary-comparison}. Then a finite collection of
nested graph charts can be chosen, together with constants
$h_0>0$, $c_h\in(0,1]$, $\mathfrak L\ge0$ and $C_1\ge1$ that do not depend on the chart, with the following properties. In each chart,
$\Omega$ lies above a graph
\[
 \{(a,\varphi(a)):a\in U\subset\R^{2n-1}\},
 \qquad |\varphi(a)-\varphi(b)|\le \mathfrak L\norm{a-b},
\]
with inward vertical unit vector $V$. For every graph point $p$ in the
smaller chart and every $0<h\le h_0$,
\begin{equation}\label{eq:height}
 c_hh\le\delta_\Omega(p+hV)\le h.
\end{equation}
The points at height $h_0$ over all the smaller charts lie in a fixed
compact subset of $\Omega$.

Moreover, whenever $S(x,y)$ is sufficiently small, one of the charts
admits representations
\begin{equation}\label{eq:representations}
 x=p_x+t_xV,\qquad y=p_y+t_yV,
\end{equation}
in which $p_x,p_y$ lie in its smaller graph region, $0<t_x,t_y<h_0$,
and
\begin{equation}\label{eq:chartbounds}
 t_x\le c_h^{-1}\delta_\Omega(x),\quad
 t_y\le c_h^{-1}\delta_\Omega(y),\quad
 \norm{p_x-p_y}\le C_1 S(x,y).
\end{equation}
\end{lemma}
\begin{proof}
We first choose the charts. Since $\partial\Omega$ is compact and Lipschitz, it is covered by finitely many Lipschitz graph
charts. We shrink these charts so that the smaller charts still cover the boundary,
while their closures keep a positive distance from the edges of the larger
charts. More precisely, we take the base of each larger graph chart to be a Euclidean
ball and use three nested smaller balls before selecting the finite
cover. In particular, the straight segment between two base points in
the middle ball stays inside the larger base, with a positive margin.
The graph Lipschitz constants are made uniform by taking their maximum, and the positive margins are made uniform by taking their minimum.

We now prove the two-sided bound for the boundary distance. For $p=(a,\varphi(a))$, the point $p+hV$ has vertical height $h$ above the graph.
The graph Lipschitz bound shows that its distance from the graph is at least
$h/\sqrt{1+\mathfrak L^2}$. Indeed, for any graph point $(b,\varphi(b))$,
\[
 h\le |\varphi(a)+h-\varphi(b)|+\mathfrak L\norm{a-b}
 \le\sqrt{1+\mathfrak L^2}\,
 \norm{(a,\varphi(a)+h)-(b,\varphi(b))}.
\]
Boundary points outside the larger chart stay a uniformly positive
distance away, by the margins chosen above. Taking $h_0$ smaller if necessary, we obtain the lower
bound in \eqref{eq:height}. The upper bound follows by comparison with
the boundary point $p$. The same positive lower bound at height $h_0$, together with the boundedness of the domain,
gives the compactness assertion.

It remains to prove the last assertion. Choose a closest boundary point to $x$.
If $S(x,y)$ is sufficiently small, then the two points are close to this boundary point, so both $x$ and $y$ lie in one of the
larger graph charts, with a fixed margin from its edges. This follows
from a Lebesgue number for a still smaller finite cover of the boundary.
The vertical projections of the two points onto the graph give
\eqref{eq:representations}. The projections remain in the smaller
region once the covers are fixed in this order and the
smallness threshold is decreased. Since a nearest boundary point of $x$ lies on the
graph, the Lipschitz bound gives $t_x\le(1+\mathfrak L)\delta_\Omega(x)$, and similarly for $y$. Thus $t_x,t_y<h_0$ once $S(x,y)$ is small. Equation~\eqref{eq:height},
applied with $p=p_x$, $h=t_x$ and with $p=p_y$, $h=t_y$, then gives the
first two bounds in \eqref{eq:chartbounds}. Finally,
\[
 \norm{p_x-p_y}\le\norm{x-y}+t_x+t_y\le C_1 S(x,y).
\]
This completes the proof.
\end{proof}

The next lemma collects three facts about the inward vertical curves of the charts. That these curves are quasi-geodesics is
\cite[Lemma~4.1]{BGNT24}. The bound
$k_\Omega(o,z)\le A+B\log(1/\delta_\Omega(z))$ on domains satisfying an interior
cone condition is \cite[Lemma~2.3]{BZ17}. We recall the short proofs,
because we need the constants in \eqref{eq:quasi} and the comparison of these curves
with geodesics from $o$.

\begin{lemma}\label{lem:inward}
Let $\Omega\Subset\C^n$ be a $\C$-convex domain with Lipschitz
boundary, and let $h_0$ and $V$ be as in Lemma~\ref{lem:charts}.
Then the curves
\[
 \sigma_p(t)=p+h_0e^{-t}V,\qquad t\ge0,
\]
where $p$ is a graph point in a smaller chart, are quasi-geodesics with
constants independent of $p$. For every $o\in\Omega$ and
$R>\max\{1,3\diam\Omega\}$, \eqref{eq:logupper} holds for some $A\ge0$
and $B>0$. If $(\Omega,k_\Omega)$ is Gromov hyperbolic, then there is
$K_0\ge0$ such that, for every point
$x=p+t_xV$ sufficiently close to the boundary, every point $p+hV$ with $t_x\le h\le h_0$ has
$k_\Omega$-distance at most $K_0$ from any geodesic segment $[o,x]$.
\end{lemma}
\begin{proof}
We first prove that the curves are quasi-geodesics. The Euclidean speed of $\sigma_p$ is $h_0e^{-t}$, which is the height of the point above the graph. Hence
\eqref{eq:metricupper} and \eqref{eq:height} give
\[
 k_\Omega(\sigma_p(s),\sigma_p(t))\le c_h^{-1}|s-t|.
\]
For the reverse inequality, let $s\le t$. Then \eqref{eq:ratio} and \eqref{eq:height} imply
\[
 k_\Omega(\sigma_p(s),\sigma_p(t))
 \ge\frac14(t-s)-\frac14\log(c_h^{-1}).
\]
Thus there are uniform $\lambda\ge1$ and $\beta\ge0$ such that
\begin{equation}\label{eq:quasi}
 \lambda^{-1}|s-t|-\beta
 \le k_\Omega(\sigma_p(s),\sigma_p(t))
 \le\lambda|s-t|+\beta.
\end{equation}

We next prove the logarithmic bound. The points at height $h_0$ over the smaller charts lie in a compact subset of the domain, so
$k_\Omega(o,p+h_0V)\le L_0$ for a uniform constant. For $x=p+t_xV$, the upper bound along the inward curve proved above gives
\[
 k_\Omega(o,x)\le L_0+c_h^{-1}\log\frac{h_0}{t_x}
 \le L_0+c_h^{-1}\log\frac{R}{\delta_\Omega(x)},
\]
where we also take $h_0<R$. The complement of the boundary chart
neighborhoods is compact. Increasing the additive constant therefore proves
\eqref{eq:logupper} on all of $\Omega$.

To prove the final assertion, we concatenate a unit-speed geodesic from
$o$ to $\sigma_p(0)$ with the restriction of $\sigma_p$ to
$[0,\log(h_0/t_x)]$. We claim that this concatenation is a quasi-geodesic with
uniform constants. To see this, note that its initial part has length at most $L_0$.
Consider two points on different parts, with remaining initial length $a\le L_0$
and subsequent parameter $t$. The triangle inequality and
\eqref{eq:quasi} give the lower bound
\[
 \lambda^{-1}t-\beta-a
 \ge \lambda^{-1}(a+t)-\beta-2L_0,
\]
and the corresponding upper bound is immediate. Two points on the same
part satisfy the required estimates already. This proves the claim. The Morse lemma
\cite[Chapter~III.H, Theorem~1.7]{BH99} now places the concatenation
at a uniform Hausdorff distance from $[o,x]$. Every point $p+hV$ in the
asserted range lies on the inward portion of the concatenation, and hence lies within a uniform distance of the geodesic segment. This completes the proof.
\end{proof}

\subsection{The two Gromov-product estimates}\label{sec:products}

In this subsection, we prove the two Gromov-product estimates that together give the two-sided inequality \eqref{eq:main} of Theorem~\ref{thm:boundary-comparison}. The upper estimate integrates the Euclidean speed bound along a geodesic and uses the logarithmic upper bound for the distance from the base point. The lower estimate uses the Lipschitz charts, the inward curves and Gromov hyperbolicity.

\subsubsection{The upper estimate and Euclidean geodesic length}
We first prove the upper estimate. The idea is to integrate the Euclidean speed bound of Lemma~\ref{lem:speed} against the growth of $k_\Omega(o,\cdot)$ along a geodesic. This argument follows the visibility estimates of Bharali--Zimmer \cite{BZ17}. We point out that the next lemma uses neither Gromov hyperbolicity nor Lipschitz regularity of the boundary.

\begin{lemma}\label{lem:length}
Let $\Omega\Subset\C^n$ be a $\C$-convex domain satisfying
\eqref{eq:mconvex}. Fix $o\in\Omega$ and $R>\max\{1,3\diam\Omega\}$,
suppose that $A\ge0$ and $B>0$ satisfy \eqref{eq:logupper}, and let
$S(x,y)$ be as in Theorem~\ref{thm:boundary-comparison}. Then there is $C_2>0$ such that
every Kobayashi distance geodesic $\gamma$ from $x$ to $y$ satisfies
\begin{equation}\label{eq:length}
 \ell_E(\gamma)\le C_2 e^{-(x\mid y)_o/(Bm)},
\end{equation}
where $\ell_E(\gamma)=\int\norm{\gamma'(t)}\,dt$ denotes the Euclidean
length of $\gamma$. Consequently, there is $C_3\ge0$ such that
\begin{equation}\label{eq:upperproduct}
 (x\mid y)_o\le Bm\log\frac{R}{S(x,y)}+C_3
 \qquad(x,y\in\Omega).
\end{equation}
\end{lemma}
\begin{proof}
We first prove the length estimate. Suppose that $x\ne y$, and parametrize $\gamma:[0,\ell]\to\Omega$ by Kobayashi
arclength, where $\ell=k_\Omega(x,y)$. Put
\[
 a=k_\Omega(o,x),\quad b=k_\Omega(o,y),\quad
 P=(x\mid y)_o,\quad t_0=\frac{a-b+\ell}{2}.
\]
The triangle inequality gives $0\le t_0\le\ell$. It also gives, for every
$t\in[0,\ell]$,
\begin{align}
 k_\Omega(o,\gamma(t))
 &\ge\max\{a-t,b-\ell+t\}\notag\\
 &=P+|t-t_0|.\label{eq:triangle-growth}
\end{align}
Indeed, the distance from the base point to a point of the geodesic is at least the distance from the base point to either endpoint minus the distance along the geodesic to that endpoint. To convert this growth into a Euclidean speed bound, we need the boundary distance along the geodesic, and the logarithmic upper bound controls it by the distance from the base point. Rearranging \eqref{eq:logupper} gives
\[
 \delta_\Omega(\gamma(t))\le R e^{A/B}
                e^{-k_\Omega(o,\gamma(t))/B}.
\]
Lemma~\ref{lem:speed} and \eqref{eq:triangle-growth} therefore imply
\[
 \norm{\gamma'(t)}
 \le 4C_r R^{1/m} e^{A/(Bm)}
       e^{-P/(Bm)} e^{-|t-t_0|/(Bm)}
\]
almost everywhere. The last exponential factor has integral at most
$2Bm$ on $[0,\ell]$. Integrating the speed bound therefore proves \eqref{eq:length}. For $x=y$, the
constant geodesic has zero Euclidean length, and the assertion is immediate.

We now deduce the product estimate from the length estimate. For $x\ne y$, the Euclidean distance between the endpoints of the geodesic is at most its Euclidean length, so
\begin{equation}\label{eq:pairbound}
 P\le Bm\log\frac{R}{\norm{x-y}}+C_4.
\end{equation}
The Gromov product is also at most each of the two distances from the base point, that is, $P\le\min\{k_\Omega(o,x),k_\Omega(o,y)\}$. Because $m\ge1$ and
$R>\delta_\Omega(z)$ at every point, the logarithmic term in \eqref{eq:logupper} is nonnegative, and this bound consequently gives
\[
 P\le Bm\log\frac{R}{\delta_\Omega(x)}+A,
 \qquad
 P\le Bm\log\frac{R}{\delta_\Omega(y)}+A.
\]
We now take the minimum of these two bounds and of the distance bound \eqref{eq:pairbound}. When
$x=y$, only the two boundary-distance bounds are needed. Writing
\[
 D(x,y)=\max\{\norm{x-y},\delta_\Omega(x),\delta_\Omega(y)\},
\]
we obtain $P\le Bm\log(R/D(x,y))+C_5$. Since
$S(x,y)\le3D(x,y)$, \eqref{eq:upperproduct} follows.
\end{proof}

\subsubsection{The lower estimate}
We now prove the lower estimate. Here Gromov hyperbolicity enters, through the last assertion of
Lemma~\ref{lem:inward}, which places the inward curves within bounded distance of the geodesics from the base point. The power estimate \eqref{eq:mconvex} is not used. The next lemma gives the lower bound in the two-sided estimate.

\begin{lemma}\label{lem:lowerproduct}
Let $\Omega\Subset\C^n$ be a $\C$-convex domain with Lipschitz boundary
such that $(\Omega,k_\Omega)$ is Gromov hyperbolic. Fix $o\in\Omega$
and $R>\max\{1,3\diam\Omega\}$, and let $S(x,y)$ be as in
Theorem~\ref{thm:boundary-comparison}. Then there is $C_6\ge0$ such that
\begin{equation}\label{eq:lowerproduct}
 (x\mid y)_o\ge\frac14\log\frac{R}{S(x,y)}-C_6
 \qquad(x,y\in\Omega).
\end{equation}
\end{lemma}
\begin{proof}
Put $s=S(x,y)$. We consider two cases. Suppose first that $s$ is sufficiently small, so that Lemma~\ref{lem:charts} represents both points in one chart. Thus
\[
 x=p_x+t_xV,\quad y=p_y+t_yV,\quad
 p_x=(a,\varphi(a)),\quad p_y=(b,\varphi(b)).
\]
Let $\mathfrak L$ be the common graph Lipschitz bound. The idea is to lift both points to a common height proportional to their separation, where a straight segment joins them at bounded Kobayashi distance, and then to compare the lifted points with the geodesics from the base point. Fix
$D\ge\max\{c_h^{-1},2\mathfrak LC_1\}$, depending only on the chart
constants, and set
\[
 h=Ds,\qquad u=p_x+hV,\qquad v=p_y+hV.
\]
By \eqref{eq:chartbounds}, $t_x,t_y\le c_h^{-1}s\le h$, and by the choice of the constant,
$h\ge2\mathfrak LC_1s$. After reducing the threshold for $s$, we also have
$h\le h_0$.

We claim that the straight segment $[u,v]$ has boundary distance
at least $c_2h$, for a uniform $c_2>0$. To see this, observe that its point with parameter
$\theta\in[0,1]$ has base
$a_\theta=(1-\theta)a+\theta b$, and that its height above the graph is
\[
 h+(1-\theta)\varphi(a)+\theta\varphi(b)
      -\varphi(a_\theta)
 \ge h-\mathfrak L\norm{a-b}\ge h/2.
\]
It also has height at most $h+\mathfrak L\norm{a-b}$. The graph-distance
argument in the proof of Lemma~\ref{lem:charts} therefore gives the claim, since the
fixed chart margins exclude any nearer boundary point outside the chart.
By the same margins and the smallness of $s$, the segment remains inside the larger chart. Hence the upper bound \eqref{eq:metricupper} for the Kobayashi metric, integrated along the segment, implies
\begin{equation}\label{eq:bridge}
 k_\Omega(u,v)\le\int_{[u,v]}\frac{\norm{dw}}{\delta_\Omega(w)}
 \le\frac{\norm{p_x-p_y}}{c_2h}
 \le\frac{C_1}{c_2D}=:J.
\end{equation}

We now compare the lifted points with geodesics from the base point. Choose geodesics $[o,x]$ and $[o,y]$. Since the lifted points lie on the inward curves at admissible heights, the last assertion of Lemma~\ref{lem:inward} gives
points $u'\in[o,x]$, $v'\in[o,y]$ with
\[
 k_\Omega(u,u')\le K_0,\qquad k_\Omega(v,v')\le K_0.
\]
By the triangle inequality, \eqref{eq:bridge} yields $k_\Omega(u',v')\le J+2K_0$.
Furthermore, $\delta_\Omega(u),\delta_\Omega(v)\le h$, so the radial lower bound
\eqref{eq:radiallower} gives
\begin{equation}\label{eq:lifts-radial}
 k_\Omega(o,u'),\ k_\Omega(o,v')
 \ge\frac14\log\frac{R}{h}-b_1-K_0.
\end{equation}

Since $u'$ and $v'$ lie on the selected geodesics, the distance from the base point to each endpoint is the sum of the distances along that geodesic, and the triangle
inequality gives
\begin{align*}
 k_\Omega(x,y)
 &\le k_\Omega(x,u')+k_\Omega(u',v')+k_\Omega(v',y)\\
 &=k_\Omega(o,x)-k_\Omega(o,u')+k_\Omega(u',v')
        +k_\Omega(o,y)-k_\Omega(o,v').
\end{align*}
Rearranging this inequality and applying \eqref{eq:lifts-radial}, we get
\[
 (x\mid y)_o\ge(u'\mid v')_o
 \ge\frac14\log\frac{R}{h}-C_7
 =\frac14\log\frac{R}{s}-C_7-\frac14\log D.
\]
This proves the estimate for small $s$. It remains to treat the second case. If $s$ is bounded below by
the chosen positive threshold, then $\log(R/s)$ is bounded above. In this case we increase
the constant and use $(x\mid y)_o\ge0$. This completes the proof.
\end{proof}

We record what has been proved so far. Lemma~\ref{lem:inward} gives the existence of constants satisfying \eqref{eq:logupper}, and Lemmas~\ref{lem:length} and~\ref{lem:lowerproduct} together give the two-sided estimate \eqref{eq:main} in
Theorem~\ref{thm:boundary-comparison}.

\subsection{Boundary identification and visual metrics}\label{sec:boundary}

In this subsection, we prove the remaining assertions of Theorem~\ref{thm:boundary-comparison}, namely the homeomorphism between the Euclidean closure and the Gromov compactification, and the bi-H\"older estimate for visual metrics. Both follow from the two-sided estimate of the previous subsection. We first recall the description of the Gromov boundary in terms of Gromov products.
A \emph{Gromov sequence} is a sequence $(x_j)$ in $\Omega$ such that
$(x_j\mid x_k)_o\to\infty$ as $j,k\to\infty$. Two such sequences
are equivalent if their mixed products tend to infinity. The
Gromov boundary consists of the equivalence classes. For boundary points, we use the
standard extension
\begin{equation}\label{eq:boundaryproduct}
 (\xi\mid\eta)_o
 =\inf\left\{\liminf_{j\to\infty}(x_j\mid y_j)_o:
                (x_j)\in\xi,\ (y_j)\in\eta\right\}
\end{equation}
of the Gromov product, where $(x_j)\in\xi$ means that $(x_j)$ represents
the Gromov class $\xi$. If one argument belongs to $\Omega$, it is kept
fixed in this definition. Equivalent conventions change the product
by a uniformly bounded amount. In the Gromov compactification,
convergence to $\xi\in\partial_G\Omega$ is characterized by divergence of the
product with $\xi$. We refer the reader to \cite[Section~III.H.3]{BH99}.

\begin{proof}[Proof of the boundary assertions of Theorem~\ref{thm:boundary-comparison}]
We first define the map on the boundary. Suppose that $x_j\to p\in\partial\Omega$ in the Euclidean
topology. Then
\[
 S(x_j,x_k)\longrightarrow0\qquad(j,k\to\infty).
\]
The lower bound in \eqref{eq:main} therefore shows that $(x_j)$ is a Gromov
sequence. If another sequence tends to the same $p$, the same lower
bound, applied to the mixed products, shows that the two sequences are equivalent. We may therefore
define $\mathcal I(p)$ to be this common Gromov class.

Conversely, let $(x_j)$ be a Gromov sequence. The upper bound in
\eqref{eq:main} gives
\begin{equation}\label{eq:cauchy}
 S(x_j,x_k)\le R\exp\left(\frac{C_0-(x_j\mid x_k)_o}{Bm}\right)
 \longrightarrow0.
\end{equation}
Thus $(x_j)$ is a Cauchy sequence for the Euclidean distance, and $\delta_\Omega(x_j)\to0$. Hence it converges
to a point of $\partial\Omega$. Applying \eqref{eq:cauchy} to
mixed products shows that equivalent Gromov sequences have the same
Euclidean limit. The two constructions are inverse to each other. Therefore $\mathcal I$ is a
bijection between the two boundaries, and it extends the identity on $\Omega$.

We now check that this bijection is a homeomorphism. This requires the convergence of interior points to boundary points as well as the convergence of boundary points to boundary points.
Extend $\delta_\Omega$ continuously to $\overline{\Omega}$ by setting it equal to zero on
$\partial\Omega$. For $p\in\partial\Omega$ and $z\in\overline{\Omega}$,
$z\ne p$, the estimate \eqref{eq:main} for interior points, together with
the definition \eqref{eq:boundaryproduct} of the boundary product, gives
\begin{equation}\label{eq:mixed}
 \frac14\log\frac{R}{\norm{p-z}+\delta_\Omega(z)}-C_0
 \le (\mathcal I(p)\mid\mathcal I(z))_o
 \le Bm\log\frac{R}{\norm{p-z}+\delta_\Omega(z)}+C_0.
\end{equation}
To see this, note that every sequence representing $\mathcal I(p)$ has Euclidean limit
$p$, and similarly for a boundary point $z$. An interior point $z$ is kept
fixed. Thus both bounding expressions have the asserted limits along the representing sequences.

It follows from \eqref{eq:mixed} that $z_j\to p$ in $\overline{\Omega}$ if and
only if $\mathcal I(z_j)\to\mathcal I(p)$ in the Gromov compactification. This settles continuity at boundary points, in both directions. At interior points, we note that the interior $\Omega$ is open in both compactifications, and that the
induced topology there is the Euclidean one. Since $\mathcal I$ is the identity
on $\Omega$ and maps boundary to boundary, it and its inverse are
continuous at interior points as well. Both compactifications are
metrizable in this proper geodesic setting
\cite[Chapter~III.H]{BH99}. Hence the sequential criterion for continuity proves that
$\mathcal I$ is a homeomorphism.

Finally, we prove the visual estimate. For distinct $p,q\in\partial\Omega$, \eqref{eq:mixed} reduces to
\begin{equation}\label{eq:boundarylog}
 \frac14\log\frac{R}{\norm{p-q}}-C_0
 \le(\mathcal I(p)\mid\mathcal I(q))_o
 \le Bm\log\frac{R}{\norm{p-q}}+C_0.
\end{equation}
Exponentiating this inequality and applying the defining inequality \eqref{eq:visual} of a visual metric proves
\eqref{eq:boundary-main} for distinct points. The case $p=q$ is immediate.
\end{proof}

In particular, the boundary restriction of $\mathcal I$ is H\"older of
exponent $\varepsilon/4$, and its inverse satisfies
\begin{equation}\label{eq:inverseholder}
 \norm{\mathcal I^{-1}(\xi)-\mathcal I^{-1}(\eta)}
 \le C\rho_\varepsilon(\xi,\eta)^{1/(\varepsilon Bm)}.
\end{equation}

\begin{remark}\label{rem:rays}
The same estimates give a quantitative rate for the landing of geodesic rays.
If $\gamma:[0,\infty)\to\Omega$ is a unit-speed geodesic with
$\gamma(0)=o$, then \eqref{eq:logupper} and \eqref{eq:speed} imply
\[
 \int_T^\infty\norm{\gamma'(t)}\,dt\le C e^{-T/(Bm)}.
\]
Consequently $\gamma$ has a Euclidean endpoint $p\in\partial\Omega$,
and $\norm{\gamma(T)-p}\le C e^{-T/(Bm)}$. That the ray lands also
follows from the preceding boundary identification. The displayed
estimate records the rate at which it lands.
\end{remark}

\subsection{Dini-smooth and convex domains}\label{sec:dini}

In this subsection, we show that Dini-smoothness of the boundary allows the coefficient one half in the logarithmic upper bound for the distance from the base point, and that convexity allows the coefficient one half in place of one quarter in the lower product estimate. Both improvements sharpen the visual exponents. We first recall the definition of Dini-smoothness.
A $C^1$ boundary is Dini-smooth if the modulus of continuity $\omega$
of its unit normal, in local charts, satisfies
$\int_0^1\omega(t)\,dt/t<\infty$.

The following corollary is the Dini-smooth case of the boundary comparison. The improvement comes from an estimate of Nikolov and Andreev for the Kobayashi distance on Dini-smooth domains.

\begin{corollary}[Dini-smooth boundary]\label{cor:dini}
Under the hypotheses of Theorem~\ref{thm:boundary-comparison}, suppose in addition
that $\partial\Omega$ is Dini-smooth. Then \eqref{eq:logupper} holds
with $B=1/2$ and some $A\ge0$. Consequently, for every visual metric
$\rho_\varepsilon$ satisfying \eqref{eq:visual} there are $c,C>0$
such that, for all $p,q\in\partial\Omega$,
\begin{equation}\label{eq:dini}
 c\left(\frac{\norm{p-q}}{R}\right)^{\varepsilon m/2}
 \le\rho_\varepsilon\bigl(\mathcal I(p),\mathcal I(q)\bigr)
 \le C\left(\frac{\norm{p-q}}{R}\right)^{\varepsilon/4},
\end{equation}
and
\begin{equation}\label{eq:dini-inverse}
 \norm{p-q}\le
 C\rho_\varepsilon\bigl(\mathcal I(p),\mathcal I(q)\bigr)^{2/(\varepsilon m)}.
\end{equation}
\end{corollary}
\begin{proof}
It is shown by Nikolov--Andreev \cite[Corollary~8]{NA17} that for every bounded
Dini-smooth domain there is $C_*>0$ such that
\begin{equation}\label{eq:NA}
 k_\Omega(x,y)\le\log\left(1+
 \frac{C_*\norm{x-y}}{\sqrt{\delta_\Omega(x)\delta_\Omega(y)}}\right).
\end{equation}
We apply this estimate with $x=o$ fixed and use the boundedness of $\Omega$ to absorb the numerator into the constant. If $T=\delta_\Omega(y)$, then the
right side is at most
\[
 \log(1+C_{**}/\sqrt T)
 =\frac12\log(1/T)+\log(\sqrt T+C_{**}).
\]
The second term on the right is bounded above. After adjusting the additive constant and
inserting $R$ into the logarithm, this gives \eqref{eq:logupper} with $B=1/2$.
We now apply Theorem~\ref{thm:boundary-comparison} with this value of the coefficient. This gives the two-sided visual estimate, and rearranging its lower bound gives the estimate for the inverse.
\end{proof}

\begin{remark}
We note that a $C^\infty$ boundary is Dini-smooth. Hence Corollary~\ref{cor:dini}
applies to every domain with $C^\infty$ boundary that satisfies the other
hypotheses of Theorem~\ref{thm:boundary-comparison}.
\end{remark}

In the convex case, real supporting hyperplanes give a lower bound for the Gromov product with coefficient one half in place of one quarter. This improves the right-hand visual exponent, as the next corollary shows.

\begin{corollary}[The convex case]\label{cor:convex}
Let $\Omega\Subset\C^n$ be a convex domain satisfying
\eqref{eq:mconvex} such that $(\Omega,k_\Omega)$ is Gromov
hyperbolic. Fix $o\in\Omega$ and $R>\max\{1,3\diam\Omega\}$, let $A,B$
satisfy \eqref{eq:logupper}, and let $S(x,y)$ be as in
Theorem~\ref{thm:boundary-comparison}. Then there is $C_0\ge0$ such that
\begin{equation}\label{eq:convexproduct}
 \frac12\log\frac{R}{S(x,y)}-C_0
 \le(x\mid y)_o
 \le Bm\log\frac{R}{S(x,y)}+C_0
 \qquad(x,y\in\Omega).
\end{equation}
Consequently, the upper bound in \eqref{eq:boundary-main} holds with
$\varepsilon/2$ in place of $\varepsilon/4$. If $\partial\Omega$ is also
Dini-smooth, then for every visual metric $\rho_\varepsilon$ satisfying
\eqref{eq:visual} there are $c,C>0$ such that, for all
$p,q\in\partial\Omega$,
\begin{equation}\label{eq:convexvisual}
 c\left(\frac{\norm{p-q}}{R}\right)^{\varepsilon m/2}
 \le\rho_\varepsilon\bigl(\mathcal I(p),\mathcal I(q)\bigr)
 \le C\left(\frac{\norm{p-q}}{R}\right)^{\varepsilon/2}.
\end{equation}
\end{corollary}
\begin{proof}
Bounded convex domains have Lipschitz boundary. Therefore the argument of the previous subsections applies, and we only need to see where the sharper coefficient comes from. We replace
\eqref{eq:ratio} by the sharper estimate \eqref{eq:convexratio}
below. This estimate is due to Mercer \cite[Proposition~2.4]{Mer93}. We also refer the reader to B\l ocki \cite{Blo05} and Nikolov \cite{Nik14}, where the
corresponding bound for the Carath\'eodory distance is recorded. We
include the short proof. Suppose first that
$\delta_\Omega(z)\ge\delta_\Omega(w)$, and choose $q\in\partial\Omega$ with
$\norm{w-q}=\delta_\Omega(w)$. We first note that the hyperplane through $q$ orthogonal to
$w-q$ supports $\Omega$. Indeed, a supporting hyperplane at $q$
must also support the interior ball $B(w,\delta_\Omega(w))$, and its normal
is consequently parallel to $w-q$.

Let $\nu=(w-q)/\delta_\Omega(w)$ and define the affine holomorphic functional
\[
 \ell(\zeta)=\sum_{j=1}^n(\zeta_j-q_j)\overline{\nu_j}.
\]
Then $\ell(\Omega)\subset\{\Ree\zeta>0\}$,
$\Ree\ell(w)=\delta_\Omega(w)$, and
$\Ree\ell(z)\ge\delta_\Omega(z)$. The last inequality follows because the
supporting hyperplane is disjoint from $\Omega$. The Kobayashi distance decreases
under $\ell$, and the Kobayashi distance of the half-plane is explicit. Together, these give
\[
 k_\Omega(z,w)\ge\frac12
       \left|\log\frac{\Ree\ell(z)}{\Ree\ell(w)}\right|
 \ge\frac12\log\frac{\delta_\Omega(z)}{\delta_\Omega(w)}.
\]
Interchanging the roles of $z,w$ proves
\begin{equation}\label{eq:convexratio}
 k_\Omega(z,w)\ge\frac12
       \left|\log\frac{\delta_\Omega(z)}{\delta_\Omega(w)}\right|.
\end{equation}
In particular, the radial lower bound \eqref{eq:radiallower} holds with coefficient $1/2$.
We use this improved radial lower bound in
\eqref{eq:lifts-radial}. All the other steps in the proof of
Lemma~\ref{lem:lowerproduct} are unchanged. This gives
\eqref{eq:convexproduct}. The boundary argument is the same as before, and
Corollary~\ref{cor:dini} supplies $B=1/2$ in the Dini-smooth case. We also note that, for convex domains with Dini-smooth boundary, the lower bound in \eqref{eq:convexproduct} follows directly from \eqref{eq:convexratio} and the estimate \eqref{eq:NA} of Nikolov--Andreev, since these give $k_\Omega(o,x)\ge\frac12\log(1/\delta_\Omega(x))-C$ and $k_\Omega(x,y)\le\log\bigl(CS(x,y)/\sqrt{\delta_\Omega(x)\delta_\Omega(y)}\bigr)$, and the boundary distances cancel in the Gromov product. This proves the corollary.
\end{proof}

\begin{remark}
We emphasize that for general $\C$-convex domains our proof gives the right-hand visual
exponent $\varepsilon/4$. The sharper coefficient in the convex case
comes from real supporting hyperplanes. The general $\C$-convex argument uses complex supports only, and it does not use this sharper coefficient.
\end{remark}

\subsection{Consequences for the domains in the main theorems}
In this subsection, we specialize the boundary comparison to the domains that appear in our main theorems, namely the Lipschitz domains with Gromov hyperbolic Kobayashi distance of Theorem~\ref{thm:lipschitz}, and the smooth domains of finite type of Theorem~\ref{thm:main}. We also deduce a boundary extension theorem for Kobayashi isometries.

Our first corollary needs no smoothness of the boundary. Theorem~\ref{thm:lipschitz} supplies the exponent in the power estimate for the directional radius.

\begin{corollary}\label{cor:lipschitz}
Let $\Omega\Subset\C^n$, $n\ge2$, be a $\C$-convex domain with
Lipschitz boundary such that $(\Omega,k_\Omega)$ is Gromov hyperbolic.
Let $\alpha$ be as in Theorem~\ref{thm:lipschitz}, fix $o\in\Omega$
and $R>\max\{1,3\diam\Omega\}$, and let $A\ge0$ and $B>0$ satisfy
\eqref{eq:logupper}. Then the identity extends to a homeomorphism
$\mathcal I:\overline\Omega\to\overline\Omega^{\,G}$, and for every
visual metric $\rho_\varepsilon$ satisfying \eqref{eq:visual} there
are $c,C>0$ such that
\begin{equation}\label{eq:lipschitz-visual}
 c\left(\frac{\norm{p-q}}R\right)^{\varepsilon B/\alpha}
 \le \rho_\varepsilon(\mathcal I(p),\mathcal I(q))
 \le C\left(\frac{\norm{p-q}}R\right)^{\varepsilon/4},
 \qquad p,q\in\partial\Omega.
\end{equation}
In particular, the boundary restriction of $\mathcal I$ is bi-H\"older
for every such visual metric.
\end{corollary}
\begin{proof}
Theorem~\ref{thm:lipschitz} supplies \eqref{eq:mconvex} with
$m=1/\alpha>2$. With this exponent, Theorem~\ref{thm:boundary-comparison} gives the homeomorphism and the visual estimate.
\end{proof}

In the smooth finite type case, the exponent in the power estimate is the maximal type, and the boundary is Dini-smooth. This gives the following corollary.

\begin{corollary}[Maximal type and the visual boundary]\label{cor:smooth}
Let $\Omega\Subset\C^n$, $n\ge2$, be $\C$-convex with $C^\infty$
boundary of finite type, and put
$M=\max_{p\in\partial\Omega}\Delta_1(\partial\Omega,p)$. Fix
$o\in\Omega$ and $R>\max\{1,3\diam\Omega\}$. Then the identity extends
to a homeomorphism $\mathcal I:\overline\Omega\to\overline\Omega^{\,G}$,
and for every visual metric $\rho_\varepsilon$ satisfying
\eqref{eq:visual} there are $c,C>0$ such that
\begin{equation}\label{eq:smooth-visual}
 c\left(\frac{\norm{p-q}}R\right)^{\varepsilon M/2}
 \le \rho_\varepsilon(\mathcal I(p),\mathcal I(q))
 \le C\left(\frac{\norm{p-q}}R\right)^{\varepsilon/4},
 \qquad p,q\in\partial\Omega.
\end{equation}
Consequently, the inverse boundary map is H\"older with exponent
$2/(\varepsilon M)$. If $\Omega$ is convex, the exponent
$\varepsilon/4$ on the right can be replaced by $\varepsilon/2$.
\end{corollary}
\begin{proof}
By Proposition~\ref{prop:geometric-half}, $(\Omega,k_\Omega)$ is Gromov
hyperbolic, and the proof of that proposition shows that $M_1(\Omega)<\infty$. We next identify the exponent.
Proposition~\ref{prop:two-depth} with $k=1$ and the larger depth fixed
gives \eqref{eq:mconvex} with $m=M_1(\Omega)$. Lemma~\ref{lem:full-curve-type}
then gives $M\le M_1(\Omega)$, and a complex line of maximal contact gives
$M\ge M_1(\Omega)$. Thus \eqref{eq:mconvex} holds with $m=M$. Finally, a smooth
boundary is Dini-smooth, so Corollaries~\ref{cor:dini}
and~\ref{cor:convex} give the assertions. This proves the corollary.
\end{proof}

We end this section with the boundary extension of Kobayashi isometries between two domains of the kind considered here. The H\"older exponents are explicit in the constants of the two domains.

\begin{corollary}[Boundary extension of isometries]\label{cor:isometry}
For $i=1,2$, let $\Omega_i$ satisfy the hypotheses of
Theorem~\ref{thm:boundary-comparison}, with exponent $m_i$ in \eqref{eq:mconvex}
and constant $B_i$ in \eqref{eq:logupper}. Then every surjective
Kobayashi isometry $F:\Omega_1\to\Omega_2$ extends to a homeomorphism
$\overline F:\overline\Omega_1\to\overline\Omega_2$, and there is $C>0$
such that
\begin{align}
 \norm{\overline F(p)-\overline F(q)}
 &\le C\norm{p-q}^{1/(4B_2m_2)},\qquad p,q\in\partial\Omega_1,
 \label{eq:map-holder}\\
 \norm{\overline F^{-1}(u)-\overline F^{-1}(v)}
 &\le C\norm{u-v}^{1/(4B_1m_1)},\qquad u,v\in\partial\Omega_2.
 \label{eq:inverse-map-holder}
\end{align}
If both boundaries are Dini-smooth, the displayed exponents can be taken
to be $1/(2m_2)$ and $1/(2m_1)$.
\end{corollary}
\begin{proof}
Choose $o_1\in\Omega_1$ and $o_2=F(o_1)$. An isometry preserves the Gromov products of interior points exactly, with respect to these base points. Hence it carries Gromov sequences and their
equivalence classes bijectively to those for $\Omega_2$. Since convergence in the Gromov compactification is characterized by the divergence of Gromov products, the map so obtained is a homeomorphism of the Gromov
compactifications. Composing it with the identifications in
Theorem~\ref{thm:boundary-comparison} for the two domains gives $\overline F$.

We now prove the H\"older estimates. For distinct boundary points $p,q$, the boundary products are also
preserved. This follows from their definition as an infimum of lower limits over representing sequences, together with the surjectivity of the isometry.
We apply the lower bound of \eqref{eq:boundarylog} in $\Omega_1$ and
the upper bound in $\Omega_2$ and obtain
\[
 \frac14\log\frac{R_1}{\norm{p-q}}-K_1
 \le B_2m_2\log\frac{R_2}{\norm{\overline F(p)-\overline F(q)}}+K_2,
\]
where $R_i$ and $K_i$ are the constants $R$ and $C_0$ of
\eqref{eq:boundarylog} for $\Omega_i$. Rearranging this inequality proves
\eqref{eq:map-holder}. Applying the same argument to $F^{-1}$ proves
\eqref{eq:inverse-map-holder}. For the last assertion, Corollary~\ref{cor:dini} permits
$B_1=B_2=1/2$. This completes the proof.
\end{proof}

\begin{remark}
A biholomorphism is a surjective Kobayashi isometry. Hence
Corollary~\ref{cor:isometry} applies to biholomorphisms between the
domains of Corollary~\ref{cor:lipschitz}. For isometries between
$m$-convex Dini-smooth convex domains, bi-H\"older extension was proved
by Liu, Wang and Zhou \cite[Theorem~1.2]{LWZ22} and by Liu, Pu and Wang
\cite[Theorem~1.6]{LPW23}. For the continuous extension of isometries by means of
visibility, we refer the reader to \cite{BGZ21,CMS24}. We emphasize that Corollary~\ref{cor:isometry}
gives H\"older estimates on the boundaries, and not on the whole closures.
\end{remark}

\part{Estimates from interior weights}\label{part:analysis}

\section{Preliminaries for the analytic part}\label{sec:prelim}

In this section, we fix the notation of the analytic part of the paper. We then recall the facts that we shall use about the complex of maximal distributional operators on square-integrable forms, about Hessian inequalities in the sense of distributions and about the weighted inequality of Morrey, Kohn and H\"ormander.

\subsection{Notation}\label{sec:notation}
Throughout Part~\ref{part:analysis}, $D$ denotes a bounded domain in $\C^n$. We assume regularity of the boundary only where it is stated. Besides the notation of Sections~\ref{sec:introduction} and~\ref{sec:gauges}, we shall use the following notation. Constants $c,C$ may change from line to line, and they depend only on the quantities indicated.
\begin{enumerate}
\item The set $\HH=\{\zeta\in\C:\operatorname{Re}\zeta>0\}$ is the right half-plane. We write $\Vol$ for Lebesgue measure.
\item The function $K_D(z,w)$ is the Bergman kernel of $D$, i.e., the reproducing kernel of the Bergman space $A^2(D)=\OO(D)\cap L^2(D)$. We write $K_D(z)=K_D(z,z)$ for its restriction to the diagonal, and $P_D$ denotes the Bergman projection. The Bergman metric is
\[
 \beta_D(z;v)^2=\sum_{j,k=1}^n
 \frac{\partial^2\log K_D(z)}{\partial z_j\partial\bar z_k}v_j\bar v_k.
\]
These conventions agree with those of \cite{Ber70,JP13}.
\item The operator $N_q$ is the $\ddbar$-Neumann operator on $L^2_{(0,q)}(D)$, and $S_q$ is the canonical solution operator, which is defined after the statement of Lemma~\ref{lem:uniform-L2}. For a compact operator $T$, we write $s_k(T)$ for its singular values.
\item The space $C^k(\overline D)$ consists of the functions whose derivatives through order $k$ extend continuously to $\overline D$.
\end{enumerate}

\subsection{The \texorpdfstring{$L^2$}{L2} complex}\label{sec:complex}

In this subsection, we set up the complex of maximal distributional operators on square-integrable forms, and we record the facts about eigenvalues of induced Hessians that let us pass from one form degree to a higher one.

Let $\mathcal H_q(D)=L^2_{(0,q)}(D)$, with the standard coefficient norm. Inner products are linear in the first variable. We write
\[
T_q=\ddbar_q:\dom T_q\subset\mathcal H_q(D)\longrightarrow\mathcal H_{q+1}(D),\qquad T_n=0.
\]
These operators are closed and densely defined. Since distributional derivatives commute, we have $T_qT_{q-1}=0$ on the natural domain of the composition. The adjoint $T^*_{q-1}$ agrees in the interior with the formal adjoint $\vartheta$. In the sense of distributions, we have
\begin{equation}\label{eq:prelim-1}
\ddbar\vartheta+\vartheta\ddbar=-\tfrac14\Delta
\end{equation}
coefficientwise.

Let $H$ be a Hermitian matrix. After diagonalization, one sees that the eigenvalues of $H^{[q]}$ are
\begin{equation}\label{eq:prelim-2}
\mu_{j_1}+\dots+\mu_{j_q},\qquad j_1<\dots<j_q.
\end{equation}
Therefore, if $H\ge cG$ with $G\ge 0$, then $H^{[q]}\ge cG^{[q]}$. We shall also use the identity
\begin{equation}\label{eq:prelim-3}
\sideset{}{'}\sum_{|K|=q-1}\sum_{j=1}^n|u_{jK}|^2=q|u|^2.
\end{equation}
A lower bound for an induced Hessian passes to higher degrees in the following way. If $H^{[q]}\ge w\,\mathrm{Id}$ for $w\ge 0$, then $H^{[q']}\ge(q'/q)w\,\mathrm{Id}$ for $q\le q'\le n$. To see this, note that the sum of the first $q$ eigenvalues is at least $w$. Hence $\mu_q\ge w/q$, and all subsequent eigenvalues are at least $w/q$ as well, which gives the bound.

\subsection{Distributional inequalities}\label{sec:distributional}
In this subsection, we recall what a Hessian inequality means for a bounded weight. For $\lambda \in L^\infty(D;\R)$, the entries of $H^{[q]}_\lambda$ are distributions. Let $A$ be a continuous Hermitian endomorphism field on $\Lambda^{0,q}\C^n$. The inequality $H^{[q]}_\lambda \ge A$ in distributions is understood as in Section~\ref{sec:introduction}. That is, the inequality is tested against constant coefficient vectors $\eta$ and against nonnegative functions $\chi\in C^\infty_c(D)$. This notion is invariant under constant unitary changes of coordinates.

\subsection{The weighted inequality on smooth domains}\label{sec:weighted-identity}
In this subsection, we state the weighted inequality on smoothly bounded pseudoconvex domains from which all our estimates start. The following result is due to Morrey, Kohn and H\"ormander.

\begin{result}[Morrey--Kohn--H\"ormander, \cite{Hor65,Koh63,Mor58}]\label{res:morrey-kohn-hormander}
Suppose that $U$ is smoothly bounded and pseudoconvex, that $\varphi$ is real and $C^2$ near $\overline U$, and that $a\in\mathcal V_q(U)$. Then
\begin{equation}\label{eq:prelim-4}
\int_U\mathcal L^\varphi_q(a)e^{-\varphi}\le\norm{T_qa}^2_\varphi+\norm{T^*_{q-1}a+(\partial\varphi)\lrcorner a}^2_\varphi.
\end{equation}
\end{result}

The inequality \eqref{eq:prelim-4} is the weighted identity of Morrey \cite{Mor58} and Kohn \cite{Koh63}, in the weighted form of H\"ormander \cite{Hor65}, after the nonnegative gradient term and the nonnegative Levi boundary term have been discarded. See also Folland--Kohn \cite{FK72}. H\"ormander's form of the identity \cite{Hor65} allows the $C^2$ weights that we use here. The passage from forms that are smooth up to the boundary to the full form domain is H\"ormander's density theorem \cite{Hor65}. For the weighted identity and for the density theorem in the graph norm we also refer the reader to McNeal--Varolin \cite[Theorems~1.1.3 and~1.1.4]{MV15}. We shall use the inequality only on smooth relatively compact subdomains of $D$.

\section{Uniform estimates and exhaustion correction}\label{sec:exhaustion}

In this section, we collect the facts about canonical solutions on an arbitrary bounded pseudoconvex domain that we shall use in the rest of this part. We first prove uniform estimates whose constants depend only on the dimension and on the radius of a ball containing the domain. We then show how to correct the restrictions of a form to the members of a smooth exhaustion so that the corrected forms satisfy the adjoint boundary condition.

\subsection{Uniform estimates and canonical solutions}\label{sec:uniform}
The following lemma is H\"ormander's $L^2$ existence theorem \cite{Hor65}. See also \cite{MV15,Str10}. We recall the argument because we need the constants to depend on $R$ only, without any assumption on $\partial D$.

\begin{lemma}[Uniform $L^2$ estimates]\label{lem:uniform-L2}
Let $D\subset B(0,R)\subset\C^n$ be a pseudoconvex open set. For $1\le j\le n$ and $g\in\ker T_j$, there exists $a\in\dom T_{j-1}$ with
\[
T_{j-1}a=g,\qquad\norm{a}_2\le C_R\norm{g}_2.
\]
In particular, $\operatorname{Ran}T_{j-1}=\ker T_j$ is closed. Moreover,
\begin{equation}\label{eq:exhaustion-1}
\norm{u}_2^2\le C_RQ_D(u),\qquad u\in\mathcal V_q(D),\quad 1\le q\le n.
\end{equation}
Here $C_R$ depends only on $n$ and $R$.
\end{lemma}

Before the proof, we fix some notation. We denote by $S_j$ the minimum-norm solution operator of $T_{j-1}$ on $\ker T_j$, extended by zero on $(\ker T_j)^\perp$.

\begin{proof}
We first suppose that $D=U$ is smooth and pseudoconvex. We choose $t>0$ with $tR^2\le 1/4$, and we apply Result~\ref{res:morrey-kohn-hormander} to the weight $\varphi=t|z|^2$. The gradient term in the adjoint norm has to be absorbed. For this weight we have
\[
|\partial\varphi(v)|^2\le tR^2\mathcal L^\varphi_1(v)\le\tfrac14\mathcal L^\varphi_1(v),
\]
so by the inequality $|x+y|^2\le 2|x|^2+2|y|^2$ half of the Hessian integral is absorbed into the left side. The weighted and the unweighted norms are uniformly equivalent, and $\mathcal L^\varphi_q(a)=tq|a|^2$. Hence we get \eqref{eq:exhaustion-1} on $U$.

We next deduce the existence of solutions from the coercive estimate by duality. Let $P$ be the orthogonal projection onto $\ker T_j$. Let $b\in\dom T^*_{j-1}$. Pairing with elements $T_{j-1}c\in\ker T_j$ of the range, we get $Pb\in\dom T^*_{j-1}$ and $T^*_{j-1}Pb=T^*_{j-1}b$. By coercivity applied to $Pb$, we obtain
\[
\norm{Pb}_2\le C_R\norm{T^*_{j-1}b}_2.
\]
Now let $g$ be closed. The conjugate-linear functional $T^*_{j-1}b\mapsto(g,b)=(g,Pb)$ is then well defined, and it is bounded by $C_R\norm{g}_2\norm{T^*_{j-1}b}_2$. Therefore, by the Hahn--Banach and Riesz theorems, there exists $a$ with $(a,T^*_{j-1}b)=(g,b)$ and with the desired norm bound. By the double-adjoint characterization, this means $T_{j-1}a=g$.

For general $D$, we take an increasing exhaustion of $D$ by smoothly bounded strictly pseudoconvex open sets $U_\nu\Subset D$. Such an exhaustion is obtained from a smooth strictly plurisubharmonic exhaustion function and its regular sublevel values, see \cite{Hor65}. We solve on $U_\nu$ for $g|_{U_\nu}$, and we extend the solutions by zero as $L^2$ forms on $D$. By the uniform bound of the smooth case, this gives a bounded sequence. A weakly convergent subsequence has a limit with the same bound, and compactly supported tests give $T_{j-1}a=g$ for this limit. This proves exactness and closed range, and it gives the existence of $S_j$.

It remains to obtain coercivity on $D$. We decompose $u=h+b$ orthogonally with $h\in\ker T_q$. Here $b$ is the minimum-norm solution of $T_qb=T_qu$, so that $\norm{b}_2\le C_R\norm{T_qu}_2$. In top degree we have $b=0$. Set $a=S_qh$. Since $T_{q-1}a=h$, the adjoint identity gives
\[
\norm{h}_2^2=(u,h)=(T^*_{q-1}u,a)\le C_R\norm{T^*_{q-1}u}_2\norm{h}_2.
\]
Orthogonality of the decomposition now gives \eqref{eq:exhaustion-1}.
\end{proof}

The next lemma records the identities satisfied by the canonical solution operators and by their adjoints. We shall use them to build the corrected forms in the next subsection.

\begin{lemma}[Canonical identities]\label{lem:canonical}
Let $D\Subset\C^n$ be pseudoconvex, let $1\le j\le n$, and let $P_{j-1}$ be the orthogonal projection onto $(\ker T_{j-1})^\perp$. Then
\[
S_jT_{j-1}b=P_{j-1}b\quad(b\in\dom T_{j-1}),
\]
and, for every $f\in\mathcal H_{j-1}(D)$,
\begin{equation}\label{eq:exhaustion-2}
S_j^*f\in\ker T_j\cap\dom T^*_{j-1},\qquad T^*_{j-1}S_j^*f=P_{j-1}f.
\end{equation}
If $j\le n-1$ and $g\in\ker T_{j+1}$, then $S_{j+1}g$ has $T_j$-derivative $g$ and $T^*_{j-1}$-derivative zero.
\end{lemma}

\begin{proof}
The first identity is the minimum-norm characterization of the canonical solution. Pairing the first identity with $f$, we get the adjoint identity. Since $S_j$ vanishes on $(\ker T_j)^\perp$, its adjoint takes values in $\ker T_j$. Finally, $S_{j+1}g$ is orthogonal to $\ker T_j$, and hence to $\operatorname{Ran}T_{j-1}$. This orthogonality is exactly the condition that its adjoint derivative is zero. Its $T_j$-derivative is $g$ by construction. This proves the lemma.
\end{proof}

\subsection{Correcting forms on an exhaustion}\label{sec:correction}
In this subsection, we modify the restrictions of a form to the members of a smooth exhaustion so that the modified forms satisfy the adjoint boundary condition on each member. This correction is due to Straube \cite[proof of Theorem~1]{Str97}. We give a version on the full domain, in which the zero extensions of the corrected forms converge weakly to the given form. The following lemma is the precise statement.

\begin{lemma}[Full-domain exhaustion correction]\label{lem:exhaustion}
Let $D$ be a bounded pseudoconvex domain, and let $D_\nu\Subset D_{\nu+1}\Subset D$ be an increasing exhaustion of $D$ by smoothly bounded pseudoconvex open sets. For every $u\in\mathcal V_q(D)$ there exist $u_\nu\in\mathcal V_q(D_\nu)$ such that
\begin{equation}\label{eq:exhaustion-3}
Q_{D_\nu}(u_\nu)\le Q_D(u),\qquad E_\nu u_\nu\rightharpoonup u\text{ in }\mathcal H_q(D),
\end{equation}
where $E_\nu$ is zero extension to $D$. The zero extensions of their adjoint derivatives converge strongly to $T^*_{q-1}u$.
\end{lemma}

\begin{proof}
Put $f=T^*_{q-1}u$ and $g=T_qu$. We define the corrected forms by
\begin{equation}\label{eq:exhaustion-4}
u_\nu=S^*_{q,\nu}(f|_{D_\nu})+S_{q+1,\nu}(g|_{D_\nu}),
\end{equation}
where the second term is omitted when $q=n$. The restriction of $g$ is closed. Let $P_\nu$ be the orthogonal projection onto $(\ker T_{q-1,\nu})^\perp$. Then Lemma~\ref{lem:canonical}, applied to both terms, gives
\begin{equation}\label{eq:exhaustion-5}
T_{q,\nu}u_\nu=g|_{D_\nu},\qquad T^*_{q-1,\nu}u_\nu=P_\nu(f|_{D_\nu}),
\end{equation}
and these two identities prove the energy bound in \eqref{eq:exhaustion-3}. By Lemma~\ref{lem:uniform-L2}, the forms $E_\nu u_\nu$ are then bounded uniformly in $L^2$.

We next show that the projections in the second identity disappear in the limit. Let $h_\nu=(I-P_\nu)(f|_{D_\nu})$. By testing on compact subsets, any weak subsequential limit $h$ of $E_\nu h_\nu$ is closed. Since $f=T^*_{q-1}u$ is orthogonal to $\ker T_{q-1}$, we have
\[
\norm{h_\nu}_2^2=(f,E_\nu h_\nu)\longrightarrow(f,h)=0
\]
along that subsequence. The same argument applies to every subsequence, so $E_\nu h_\nu\to 0$ strongly. Consequently,
\begin{equation}\label{eq:exhaustion-6}
E_\nu P_\nu(f|_{D_\nu})\longrightarrow f\quad\text{strongly in }L^2(D).
\end{equation}
By the second identity above, this is the strong convergence of the zero extensions of the adjoint derivatives asserted in the lemma.

The following observation lets the adjoint condition pass through zero extension. Let $a\in\dom T^*_{q-1,\nu}$ and $b\in\dom T_{q-1}$. The restriction of a form in the maximal domain to a subdomain lies in the maximal domain of the subdomain, and therefore
\[
(E_\nu a,T_{q-1}b)_D=(a,T_{q-1,\nu}(b|_{D_\nu}))_{D_\nu}=(E_\nu T^*_{q-1,\nu}a,b)_D.
\]
Hence $E_\nu a\in\dom T^*_{q-1}$, with the adjoint derivative given by the right side of the identity above.

We can now identify the weak limits. Let $u^*$ be a weak subsequential limit of $E_\nu u_\nu$. By \eqref{eq:exhaustion-6} and the observation just made, we get $T^*_{q-1}u^*=f$. Compactly supported tests in \eqref{eq:exhaustion-5} give $T_qu^*=g$. By coercivity applied to $u^*-u$, we get $u^*=u$. Since the weak subsequential limit is unique, the whole sequence converges. This completes the proof.
\end{proof}

\begin{remark}\label{rem:exhaustion}
In general, $E_\nu u_\nu$ need not belong to $\dom T_q$ on $D$. In what follows, we shall use only three properties of the corrected forms. These are that zero extension preserves the adjoint domain, that $T_q$ can be tested against compactly supported forms, and that the zero extensions converge weakly in $L^2$.
\end{remark}

\section{The matrix-weighted estimate}\label{sec:matrix}

In this section, we prove a weighted estimate on the full form domain of an arbitrary bounded pseudoconvex domain. This estimate is the analytic input of the proof of Theorem~\ref{thm:zero-extension} in Section~\ref{sec:proof-zero-extension}. The weight is bounded, and its induced Hessian is only assumed to dominate a continuous nonnegative endomorphism field in the sense of distributions. The following theorem is the precise statement. The constant in the estimate is absolute, and the weight enters only through its oscillation.

\begin{theorem}[Matrix-weighted estimate]\label{thm:matrix}
Let $D \Subset \C^n$ be pseudoconvex and $1 \le q \le n$. Suppose that $\lambda \in L^\infty(D;\R)$ and that $A$ is a continuous nonnegative Hermitian endomorphism field on $\Lambda^{0,q}\C^n$ satisfying $H^{[q]}_\lambda \ge A$ in distributions. Then
\begin{equation}\label{eq:matrix-1}
\int_D \langle A(z)u(z), u(z)\rangle \, dV(z) \le C \max\{1, \operatorname{osc}_D \lambda\}\, Q_D(u), \qquad u \in \mathcal V_q(D),
\end{equation}
where $\operatorname{osc}_D\lambda=\operatorname{ess\,sup}_D\lambda-\operatorname{ess\,inf}_D\lambda$ and $C$ is an absolute constant.
\end{theorem}

\begin{proof}
We first normalize the weight. We divide $\lambda$ and $A$ by $\Theta = \max\{1, \operatorname{osc}_D \lambda\}$, and we subtract the essential infimum of the normalized function. Hence we may assume that $0 \le \lambda \le 1$ almost everywhere.

We next prove the estimate on smooth subdomains for smooth weights. Let $U \Subset D$ be smooth pseudoconvex, and suppose that $\lambda$ is smooth near $\overline U$. Set $t = (4e)^{-1}$ and $\varphi = te^\lambda$. The exponential is chosen so that the gradient term can be absorbed by the Hessian term. By the chain rule,
\begin{equation}\label{eq:matrix-2}
\mathcal L^{\varphi}_q(a) = te^\lambda \left( \mathcal L^{\lambda}_q(a) + \sideset{}{'}\sum_{|K|=q-1} \biggl| \sum_j \lambda_j a_{jK} \biggr|^2 \right).
\end{equation}
Both terms in parentheses are nonnegative, because $H^{[q]}_\lambda \ge A \ge 0$. Thus
\[
|(\partial\varphi)\lrcorner a|^2 = t^2 e^{2\lambda} \sideset{}{'}\sum_{K} \biggl| \sum_j \lambda_j a_{jK} \biggr|^2 \le te^\lambda \mathcal L^{\varphi}_q(a) \le \tfrac14 \mathcal L^{\varphi}_q(a).
\]
Put $I = \int_U \mathcal L^{\varphi}_q(a) e^{-\varphi}$. Then the inequality \eqref{eq:prelim-4} gives
\[
I \le 2\bigl( \norm{T_q a}^2_\varphi + \bigl\| T^*_{q-1} a \bigr\|^2_\varphi \bigr) + \tfrac12 I.
\]
Since $0 < \varphi \le 1/4$ and $H^{[q]}_\varphi \ge tA$, we obtain
\begin{equation}\label{eq:matrix-3}
\int_U \langle Aa, a\rangle \le 4t^{-1} e^{1/4} Q_U(a), \qquad a \in \mathcal V_q(U).
\end{equation}

Now, let $\lambda$ be nonsmooth. We remove the smoothness assumption by regularization. Fix $U \Subset D$, and let $\chi_\eps$ be a nonnegative standard convolution kernel, with $\eps$ smaller than the distance from $\overline U$ to $\partial D$. Then $\lambda_\eps = \lambda * \chi_\eps$ is smooth near $\overline U$ and takes values in $[0,1]$. Moreover, it satisfies
\[
H^{[q]}_{\lambda_\eps} \ge A_\eps := A * \chi_\eps \ge 0,
\]
which follows by convolving the distributional inequality for each constant vector. We apply \eqref{eq:matrix-3} with $A_\eps$ in place of the field, and we let $\eps\to0$. Since $A_\eps \to A$ uniformly on $\overline U$, we get the same estimate with $A$. The constant is independent of $U$.

To pass to $D$, we use the exhaustion and the corrected forms $u_\nu$ of Lemma~\ref{lem:exhaustion}. The field need not be bounded, so we truncate it first. For $L > 0$, let $A_L$ be the spectral truncation of $A$, which replaces each eigenvalue $a$ by $\min\{a, L\}$. This is a bounded continuous nonnegative field with $A_L \le A$. Since multiplication by $A_L^{1/2}$ is bounded on $L^2$, the weak lower semicontinuity of the norm, together with the estimate already proved on the members of the exhaustion, gives
\[
\int_D \langle A_L u, u\rangle \le \liminf_{\nu} \int_{D_\nu} \langle A_L u_\nu, u_\nu\rangle \le C Q_D(u).
\]
We let $L \to \infty$. By monotone convergence, we get the normalized assertion. Restoring $\Theta$, we get \eqref{eq:matrix-1}. This completes the proof.
\end{proof}

On smoothly bounded domains, \eqref{eq:matrix-3} is Catlin's weighted estimate \cite[Theorem~2.1]{Cat87}. The argument above, in which the exponential weight absorbs the gradient term, is the standard proof of that estimate. Theorem~\ref{thm:matrix} extends the estimate in two directions. It holds on the full form domain of an arbitrary bounded pseudoconvex domain, and it allows bounded weights whose Hessian bound holds only in the sense of distributions. In particular, if $H^{[q]}_\lambda\ge w\operatorname{Id}$ in distributions with continuous $w\ge0$, then $\int_D w|u|^2\le C\max\{1,\operatorname{osc}_D\lambda\}\,Q_D(u)$ for $u\in\mathcal V_q(D)$.

The hypothesis of Theorem~\ref{thm:matrix} is a global lower bound in a fixed degree. In our applications, the lower bound is available only on a collar of the boundary. The following corollary shows that a lower bound on a collar gives an estimate on a smaller collar, and that the estimate passes to all higher form degrees.

\begin{corollary}[Collars and higher degrees]\label{cor:collars}
Let $D\Subset\C^n$ be pseudoconvex, $1\le q\le n$ and $0<t_1<t_0$. Suppose that $\lambda\in L^\infty(D;\R)$ satisfies $H^{[q]}_\lambda\ge0$ in distributions on $D$, and $H^{[q]}_\lambda\ge w\operatorname{Id}$ in distributions on $\{\delta_D<t_0\}$, where $w\ge0$ is continuous there. Then for every $q\le q'\le n$ there exists $C$ such that
\[
\int_{\{\delta_D<t_1\}} w|u|^2 \le C Q_D(u),\qquad u\in\mathcal V_{q'}(D).
\]
\end{corollary}

\begin{proof}
We first produce a global lower bound from the bound on the collar. Let $\chi$ be a continuous cutoff with values in $[0,1]$ that is equal to one on $\{\delta_D<t_1\}$ and that vanishes near $\{\delta_D\ge t_0\}$, and extend $\chi w$ by zero to $D$. Since the induced Hessian is nonnegative on all of $D$, the field $\chi w\operatorname{Id}$ is a valid global lower bound. Hence Theorem~\ref{thm:matrix} applies to this lower bound. We next pass to higher degrees. For $q'\ge q$, we convolve locally, and we use \eqref{eq:prelim-2} and the inequality that follows \eqref{eq:prelim-3} for the regularized weights. Passing back to distributions, we get $H^{[q']}_\lambda \ge (q'/q) w \operatorname{Id}$ on $\{\delta_D<t_0\}$ and nonnegativity everywhere. Therefore, Theorem~\ref{thm:matrix} in degree $q'$ gives the assertion.
\end{proof}

\section{Harmonic zero extension}\label{sec:harmonic}

In this section, we prove the two real-variable estimates on which the proof of Theorem~\ref{thm:zero-extension} in Section~\ref{sec:proof-zero-extension} rests. Throughout this section, $U\subset\R^d$ is a bounded open set. The estimates below use no pseudoconvexity. For $b=0$, the equivalence \eqref{eq:harmonic-1} is the classical characterization of $H^s$ by the Gagliardo seminorm, and it is due to Gagliardo \cite{Gag57}. We refer the reader to \cite[Proposition~3.4]{DNPV12} for this characterization. We include the proof because the logarithmic factor must be carried through it.

We begin with the kernel characterization of the logarithmically weighted spaces. The kernel below is the Gagliardo kernel multiplied by a negative power of the logarithm of the distance, and this power produces the logarithmic factor in the multiplier that defines the norm.

\begin{lemma}[A logarithmic Gagliardo kernel]\label{lem:gagliardo-log}
Let $d \ge 1$, $0 < s < 1$, and $b \ge 0$. Define
\[
\kappa_{s,b}(y) = |y|^{-d-2s} \ell(|y|)^{-2b} \quad (0 < |y| < 1), \qquad \kappa_{s,b}(y) = 0 \quad (|y| \ge 1).
\]
For every $f \in L^2(\R^d)$,
\begin{equation}\label{eq:harmonic-1}
\norm{f}^2_{H^{s,-b}} \asymp \norm{f}^2_2 + \int_{\R^d} \int_{\R^d} |f(x+y) - f(x)|^2 \kappa_{s,b}(y)\, dy\, dx.
\end{equation}
Both sides may be infinite. The constants depend only on $d, s, b$.
\end{lemma}

\begin{proof}
We use the unitary Fourier transform with phase $e^{-ix\cdot\xi}$. We first write the double integral as a Fourier multiplier. By the theorems of Tonelli and Plancherel, the double integral equals $\int A(\xi) |\widehat f(\xi)|^2\, d\xi$, where
\[
A(\xi) = \int_{|y|<1} |e^{iy\cdot\xi} - 1|^2 \kappa_{s,b}(y)\, dy.
\]
We shall compare this multiplier with the multiplier that defines the norm. To this end, we need two radial estimates. For $0 < a \le 1/2$, they read
\begin{align}
\int_0^a r^{1-2s} \ell(r)^{-2b}\, dr &\le C a^{2-2s} \ell(a)^{-2b}, \label{eq:harmonic-2}\\
\int_a^1 r^{-1-2s} \ell(r)^{-2b}\, dr &\le C a^{-2s} \ell(a)^{-2b}. \label{eq:harmonic-3}
\end{align}
The first estimate follows by dividing the interval of integration into dyadic intervals and summing geometric powers. For the second estimate, we put $T = \log(1/a)$ and $r = e^{-t}$. After factoring $e^{2sT}(1+T)^{-2b}$, the remaining integral is bounded by
\[
\int_0^T e^{-2su} (1+u)^{2b}\, du \le \int_0^\infty e^{-2su} (1+u)^{2b}\, du,
\]
because $(1+T)/(1+T-u) \le 1+u$ for $0 \le u \le T$. This proves the second estimate.

We now estimate the multiplier at large frequencies. For $R = |\xi| \ge 4$, we use the inequality $|e^{it} - 1| \le \min\{2, |t|\}$ and split the integral at $|y| = R^{-1}$. By \eqref{eq:harmonic-2} and \eqref{eq:harmonic-3}, we get $A(\xi) \le C R^{2s} [\log(e+R)]^{-2b}$. For the reverse inequality, we restrict the integral to the annulus $R^{-1} < |y| < 2R^{-1}$ and put $\eta = Ry$. By rotation invariance, we may assume that the direction of $\xi$ is a fixed unit vector $e$. The integral
\[
\int_{1<|\eta|<2} |e^{i\eta\cdot e} - 1|^2 |\eta|^{-d-2s}\, d\eta
\]
is a positive constant depending only on $d$ and $s$, and $\ell(|\eta|/R) \asymp \log(e+R)$ on this annulus. This gives the reverse inequality at large frequencies. For bounded $R$, the added $L^2$ term controls the multiplier. Hence, $1 + A(\xi)$ is comparable to the multiplier in \eqref{eq:Hsb}. This proves the lemma.
\end{proof}

We now present the second estimate. In the argument of Straube \cite{Str97}, the harmonic part is estimated by the theorem of Jerison and Kenig \cite[Theorem~4.2]{JK95} on harmonic functions in Lipschitz domains, combined with the subaveraging argument of Detraz \cite{Det81}. The next lemma replaces these two tools. Because it estimates the zero extension directly, it needs no regularity of $\partial U$.

\begin{lemma}[Harmonic coefficients on arbitrary open sets]\label{lem:harmonic}
Let $U \subset \R^d$ be bounded and open, and set $\delta_*(x) = \min\{\dist(x, \partial U), 1/2\}$. If $h \in L^2(U)$ is distributionally harmonic, $0 < s < 1$, and $b \ge 0$, then
\begin{equation}\label{eq:harmonic-4}
\norm{E_0 h}^2_{H^{s,-b}(\R^d)} \le C \int_U \delta_*(x)^{-2s} \ell(\delta_*(x))^{-2b} |h(x)|^2\, dx
\end{equation}
provided the right side is finite. The constant $C$ depends only on $d$, $s$ and $b$.
\end{lemma}

\begin{proof}
By interior elliptic regularity, $h$ is smooth. We shall estimate the double integral in the kernel characterization in three pieces, according to the position of the pair of points. We first introduce the weight. Set $W(t) = t^{-2s}\ell(t)^{-2b}$ for $0 < t \le 1/2$. This weight is positive, it is bounded below by a positive constant, and it satisfies $W(t') \asymp W(t)$ when $t'/t$ stays in a fixed compact subset of $(0,\infty)$. By \eqref{eq:harmonic-3}, we get
\begin{equation}\label{eq:harmonic-5}
\int_{a<|y|<1} \kappa_{s,b}(y)\, dy \le C W(a), \qquad 0 < a \le 1/2.
\end{equation}
We first consider the pairs with one point inside the open set and the other point outside it. For such a pair $x \in U$, $y \notin U$, we have $|x-y| \ge \delta_U(x) \ge \delta_*(x)$. Hence, for a fixed interior point, the integration in the exterior variable runs over distances at least the truncated boundary distance of that point. By \eqref{eq:harmonic-5}, it follows that the cross contribution to \eqref{eq:harmonic-1} is at most $C\int_U W(\delta_*)|h|^2$.

We next consider the far pairs, i.e., the pairs in $U \times U$ with $|x-y| \ge \delta_*(x)/4$. For these pairs, we use the inequality $|h(x) - h(y)|^2 \le 2|h(x)|^2 + 2|h(y)|^2$. The first term is controlled by \eqref{eq:harmonic-5}. For the second term, we need the distance between the two points to dominate the truncated boundary distance of the second point as well. Since $\delta_*$ is $1$-Lipschitz, we have $\delta_*(y) \le \delta_*(x) + |x-y| \le 5|x-y|$. Interchanging $x$ and $y$, we get the same bound for the second term.

It remains to consider the near pairs, i.e., the pairs with $|x-y| < \delta_*(x)/4$. For these pairs, we shall use the interior harmonic gradient estimate
\begin{equation}\label{eq:harmonic-6}
\sup_{B(x,\delta_*(x)/4)} |\nabla h|^2 \le C \delta_*(x)^{-d-2} \int_{B(x,\delta_*(x)/2)} |h(z)|^2\, dz.
\end{equation}
We recall why this estimate holds. By Caccioppoli's inequality on concentric interior balls, the integral of $|\nabla h|^2$ is bounded by the inverse square radius times the integral of $|h|^2$. By the mean-value inequality for the subharmonic function $|\nabla h|^2$, we then get \eqref{eq:harmonic-6}. Both inequalities follow from compactly supported integration by parts and the harmonic mean-value formula.

For a near pair, the segment joining $x$ to $y$ lies in this interior ball, so the difference of the values at its endpoints is bounded by its length times the supremum of the gradient over the ball. Multiplying \eqref{eq:harmonic-6} by $|x-y|^2$, integrating in $y$ against the kernel and using \eqref{eq:harmonic-2}, we see that the near contribution is bounded by
\[
C \int_U \delta_*(x)^{-d} W(\delta_*(x)) \int_{B(x,\delta_*(x)/2)} |h(z)|^2\, dz\, dx.
\]
On the region of integration, $\delta_*(x)/2 \le \delta_*(z) \le 3\delta_*(x)/2$ and $|x-z| \le \delta_*(z)$. The first relation lets us replace the weight at the outer point by the weight at the inner point, and the second relation lets us interchange the order of integration. Therefore, by Fubini's theorem and the volume of this last ball, the expression is bounded by $C\int_U W(\delta_*(z))|h(z)|^2\, dz$. Finally, the $L^2$ term in the kernel characterization is controlled by the same weighted integral, because the weight is bounded below. Hence we get \eqref{eq:harmonic-4}. This proves the lemma.
\end{proof}

\section{Proof of Theorem~\ref{thm:zero-extension}}\label{sec:proof-zero-extension}
In this section, we prove Theorem~\ref{thm:zero-extension}. The proof follows Straube's argument \cite{Str97}. We first obtain a weighted $L^2$ estimate from Theorem~\ref{thm:matrix}. We then split $u$ into a part with $H^1_0$ coefficients and a part with harmonic coefficients, and we estimate the zero extensions of the two parts separately.

\begin{proof}[Proof of Theorem~\ref{thm:zero-extension}]
We first reduce to the degree in which the hypothesis on the induced Hessian is stated. By Corollary~\ref{cor:collars}, it suffices to work in degree $q$. Next we obtain the weighted estimate. On a collar, the hypothesis on the induced Hessian and the matrix estimate bound the weighted integral by the energy. On the complement of the collar, the weight is bounded and coercivity bounds the integral by the energy. Combining the two bounds, we get
\begin{equation}\label{eq:proof-zero-extension-1}
\int_D W(\delta_*)|u|^2 \le C Q_D(u), \qquad u \in \mathcal V_q(D),
\end{equation}
where $W(t) = t^{-2s}\ell(t)^{-2b}$ and $\delta_* = \min\{\delta_D, 1/2\}$.

Now, write $u = \sum\nolimits'_{|J|=q} u_J\, d\bar z_J$. We shall show that the Laplacian of each coefficient is controlled by the energy in the dual norm. By the distributional identity \eqref{eq:prelim-1}, we get
\begin{equation}\label{eq:proof-zero-extension-2}
\norm{\Delta u_J}_{(H^1_0(D))^*} \le C Q_D(u)^{1/2}.
\end{equation}
To see this, observe that each coefficient of $\Delta u$ is a sum of first derivatives of coefficients of $T_q u$ and $T^*_{q-1} u$. Integration against a compactly supported test function therefore gives the estimate for such test functions, and the estimate extends to $H^1_0(D)$ by density.

We now solve the Dirichlet problem with these Laplacians as data. By the Lax--Milgram theorem for the Dirichlet energy, there exists $v_J \in H^1_0(D)$ such that $\Delta v_J = \Delta u_J$. The coercivity of the Dirichlet energy needed here is Poincar\'e's inequality. Observe that Poincar\'e's inequality holds on $D$ by zero extension into a fixed ball containing $D$, without any assumption on $\partial D$. With $v = \sum\nolimits'_J v_J\, d\bar z_J$, we get
\begin{equation}\label{eq:proof-zero-extension-3}
\norm{E_0 v}^2_{H^1(\R^{2n})} + Q_D(v) \le C Q_D(u).
\end{equation}
We next check that the form with these coefficients lies in the form domain. To see that $v\in\mathcal V_q(D)$, we approximate each $v_J$ in $H^1_0(D)$ by compactly supported smooth functions. Their first formal derivatives converge in $L^2$. Hence, by closedness of $T_q$ and $T^*_{q-1}$, $v$ lies in $\mathcal V_q(D)$.

Set $h = u - v$. Then the coefficients of $h$ are harmonic, $h \in \mathcal V_q(D)$ and $Q_D(h) \le C Q_D(u)$. The right side of \eqref{eq:harmonic-4} is the weighted integral that we have already bounded by the energy. Therefore, by \eqref{eq:proof-zero-extension-1} and Lemma~\ref{lem:harmonic}, we get $\norm{E_0 h}^2_{H^{s,-b}}\le C Q_D(u)$. For the other part, the first-order estimate \eqref{eq:proof-zero-extension-3} is stronger than what we need, since $s < 1$. Hence, $E_0 v$ is bounded in the same norm. Adding the two estimates, we get \eqref{eq:zero-extension}. We point out that no Hardy inequality and no extension operator for $D$ is used. This completes the proof.
\end{proof}

\section{Top degree and the half-order obstruction}\label{sec:complements}

In this section, we record two complements to Theorem~\ref{thm:zero-extension}. The first complement concerns top degree. It is standard that on $(0,n)$-forms the $\ddbar$-Neumann problem is a Dirichlet problem, see \cite{FK72,Str10}. We give the short proof because we need the estimate on arbitrary bounded open sets, without pseudoconvexity or boundary regularity. The second complement is the half-order obstruction mentioned in Subsection~\ref{sec:intro-estimates}.

We begin with the first-order estimate in top degree.

\begin{proposition}[Top degree]\label{prop:top-degree}
Let $D \subset \C^n$ be a bounded open set. There is $C_D$, depending only on $n$ and $\operatorname{diam}D$, such that
\begin{equation}\label{eq:complements-1}
\norm{E_0 u}^2_{H^1(\R^{2n})} \le C_D \bigl\| \ddbar^* u \bigr\|^2_2, \qquad u \in \dom \ddbar^*_{n-1} \subset L^2_{(0,n)}(D).
\end{equation}
\end{proposition}

\begin{proof}
Write $u = f\, d\bar z_1 \wedge \cdots \wedge d\bar z_n$. We first show that the formal adjoint commutes with zero extension in top degree. Testing the adjoint identity against restrictions of compactly supported smooth ambient $(0,n-1)$-forms, we get $\vartheta(E_0 u) = E_0(\ddbar^* u)$ in ambient distributions, cf. \cite{FK72}. Its distinct coefficients are $\pm\partial(E_0 f)/\partial z_j$. By Fourier multipliers,
\[
\norm{\nabla E_0 f}^2_2 = 4 \sum_j \norm{\partial(E_0 f)/\partial z_j}^2_2 = 4 \bigl\| \ddbar^* u \bigr\|^2_2.
\]
Since the function is already in $L^2$, the Fourier argument first establishes membership in $H^1$. It remains to control the norm of the function itself. Its support lies in the interior of a ball of radius $\operatorname{diam}D$ containing $D$, where Poincar\'e's inequality controls its $L^2$ norm. This proves the proposition.
\end{proof}

We now turn to the half-order obstruction. The next proposition shows that a half-order zero-extension estimate with logarithmic exponent at most one half fails in every degree below the top as soon as the boundary contains a smooth patch, and that on pseudoconvex domains no bounded weight satisfies the corresponding lower bound on the induced Hessian.

\begin{proposition}[Sharp half-order logarithmic obstruction]\label{prop:half-order}
Let $D \Subset \C^n$ be a domain and $1 \le q < n$. Suppose that $W$ is an open neighborhood of a point $p\in\partial D$ such that $W\cap\partial D$ is a $C^\infty$ hypersurface and $W\cap D$ lies on one side of it. Then the following hold.
\begin{enumerate}
\item[(i)] There is $u \in \mathcal V_q(D)$, smooth on $\overline D\cap W$ and supported in $W$, such that
\[
E_0 u \notin H^{1/2,-b}(\R^{2n}) \qquad (0 \le b \le 1/2).
\]
\item[(ii)] If $D$ is pseudoconvex and $0\le b\le 1/2$, there is no $\lambda\in L^\infty(D;\R)$ with $H^{[q]}_\lambda\ge0$ in distributions on $D$ and $H^{[q]}_\lambda \ge c\,\delta_D^{-1}\ell(\delta_D)^{-2b}\operatorname{Id}$ in distributions on $W\cap\{\delta_D<t_0\}$ for some $c,t_0>0$.
\end{enumerate}
\end{proposition}

\begin{proof}
We first construct the form in (i). We choose a smaller patch with a smooth defining function and a smooth local frame of complex tangential $(0,1)$-forms. We let $u$ be the wedge of $q$ independent frame elements, extended smoothly into a collar and multiplied by a smooth cutoff supported in $W$ and equal to one near $p$. The resulting form $u$ has nonzero boundary value, it satisfies the normal-contraction condition for $\dom \ddbar^*$, and it vanishes near all other boundary points. Since the form is smooth up to the boundary, $u\in\dom\ddbar_q$. We next verify the adjoint domain condition. If $\varphi\in\dom\ddbar_{q-1}$ and $\chi\in C^\infty_c(W)$ equals one near $\operatorname{supp}u$, then $(u,\ddbar\varphi)=(u,\ddbar(\chi\varphi))$, and $\chi\varphi$ is supported where $\partial D$ is smooth. Hence Friedrichs' lemma reduces the adjoint identity to smooth $\varphi$, and local integration by parts shows that $u\in\dom\ddbar^*_{q-1}$. Thus $u$ belongs to $\mathcal V_q(D)$.

We now show that the zero extension of this form is not in the half-order space. Because the form has nonzero boundary value, one coefficient $f$ is bounded away from zero in a smaller collar. We bound the cross contribution from below, i.e., the part of the double integral in which one point lies in the domain and the other point lies outside it. Because the boundary is smooth near the support, the exterior cone at a nearby boundary point contains, at distance comparable to $\delta_D(x)$ from $x \in D$, a set of volume at least $c\delta_D(x)^{2n}$. Therefore, the cross contribution in Lemma~\ref{lem:gagliardo-log}, for $s = 1/2$, is bounded below by
\[
c \int |f(x)|^2 \delta_D(x)^{-1} \ell(\delta_D(x))^{-2b}\, dx.
\]
In smooth collar coordinates the normal integral of this expression contains
\[
\int_0^{\eps} \frac{dt}{t[\log(e/t)]^{2b}} = \infty \qquad (2b \le 1).
\]
Hence, $E_0u\notin H^{1/2,-b}(\R^{2n})$. This should be compared with the failure of zero extension at $s=1/2$ on bounded $C^\infty$ domains, see \cite[Remark~2.5]{HM17}. This proves (i).

We now prove (ii). Suppose, to get a contradiction, that a weight $\lambda$ as in (ii) existed. The matrix estimate localizes to the collar in the same way as in the proof of Corollary~\ref{cor:collars}, namely by applying Theorem~\ref{thm:matrix} with a cutoff lower bound. Applying this localized estimate to the same form $u$, we see that the divergent weighted integral above would be bounded by the finite energy $Q_D(u)$, which is a contradiction.
\end{proof}

\begin{remark}
We point out what the proposition says for the ball. Taking $D=\B^n$, we see that a zero-extension estimate $\norm{E_0u}^2_{H^{1/2,-b}(\R^{2n})}\le CQ_D(u)$ with $b\le1/2$ fails on the ball in every degree $q<n$. This should be compared with Theorem~\ref{thm:directional}, where $b>1/2$.
\end{remark}

\section{Spectral bounds from zero extension}\label{sec:spectral}

In this section, we derive spectral consequences of a zero-extension estimate with a logarithmic loss, such as the one in Theorem~\ref{thm:zero-extension}. Such an estimate forces the $\ddbar$-Neumann operator $N_q$ to be compact, and it gives a decay rate for the eigenvalues and membership in a Schatten class. The argument is a Fourier truncation on a torus that contains the domain, and it uses the multiplier of Lemma~\ref{lem:gagliardo-log}. We begin with the eigenvalue bounds.

\begin{theorem}[Eigenvalue and Schatten bounds]\label{thm:spectral}
Let $D\Subset\C^n$ be pseudoconvex, $1\le q\le n$, $0<s<1$, and $b\ge0$. Suppose that
\[
\|E_0u\|^2_{H^{s,-b}(\R^{2n})}\le CQ_D(u),\qquad u\in\mathcal{V}_q(D).
\]
Then $N_q$ is compact, and its positive eigenvalues, in decreasing order, satisfy the upper bounds
\begin{equation}\label{eq:spectral-1}
\mu_j(N_q)\le C'j^{-s/n}[\log(e+j)]^{2b},
\end{equation}
where $C'$ depends only on $C$, $n$, $s$, $b$ and $\diam D$. Equivalently, the eigenvalues of the $\ddbar$-Neumann Laplacian, in increasing order, satisfy $\nu_j(\Box_q)\ge cj^{s/n}[\log(e+j)]^{-2b}$. Moreover, $N_q\in\mathfrak S_p$ for every $p>n/s$.
\end{theorem}

\begin{proof}[Proof of Theorem~\ref{thm:spectral}]
We give $V=\mathcal{V}_q(D)$ the Hilbert norm $Q_D^{1/2}$. By Lemma~\ref{lem:uniform-L2}, this norm is complete and equivalent to the form-graph norm. Let $\iota:V\longrightarrow\mathcal{H}_q(D)$ be the inclusion. We shall approximate the inclusion by operators of finite rank, and to this end we pass to a torus. We first translate $D$ so that it contains the origin. This changes none of the norms involved. We then choose a cube $Q=(-L,L)^{2n}$, with $L$ depending only on $\diam D$, whose boundary has distance more than three from $D$, and we periodize $E_0u$ on the torus obtained by identifying opposite faces of the cube.

The periodic $L^2$ norm equals $\|u\|_2$. For translations of length less than one, the squared translation difference of the periodized form over $Q$ equals the corresponding integral on $\R^{2n}$, because the support of the extended form and its unit neighborhood are separated from the faces of the cube. We now expand in Fourier series and repeat the proof of Lemma~\ref{lem:gagliardo-log}, with the same multiplier $A(\xi)$ evaluated at the lattice points. This gives
\begin{equation}\label{eq:spectral-2}
\sum_{\nu\in\mathbb{Z}^{2n}}\frac{(1+|\nu|^2)^s}{[\log(e+|\nu|)]^{2b}}|\widehat{(E_0u)_{\mathrm{per}}}(\nu)|^2\le CQ_D(u),
\end{equation}
where the form coefficients are summed and the frequency scaling $\pi/L$ is absorbed into the constants.

Let $P_R$ be the projection that retains the modes with $|\nu|\le R$. Its rank on forms is at most $C_{n,q}R^{2n}$. The multiplier is increasing for large frequencies, and therefore
\[
\|E_0u-P_RE_0u\|^2_{L^2(Q)}\le CR^{-2s}[\log(e+R)]^{2b}Q_D(u).
\]
Restricting $P_RE_0u$ to $D$, we obtain approximations of $\iota$ of finite rank whose errors tend to zero. We choose $R$ comparable to $k^{1/(2n)}$, so that the rank is less than $k$. Then
\begin{equation}\label{eq:spectral-3}
s_k(\iota)\le Ck^{-s/(2n)}[\log(e+k)]^b.
\end{equation}
Here we use the minimum-error characterization of the singular values of an operator between Hilbert spaces.

We next note that $N_q=\iota\iota^*$. Indeed, for $f\in\mathcal{H}_q(D)$, the vector $\iota^*f\in V$ is characterized by the identity $Q_D(\iota^*f,v)=(f,\iota v)$ for all $v\in V$, and its image in $\mathcal{H}_q(D)$ is exactly the weak solution that defines $N_qf$. Hence the nonzero eigenvalues of $N_q$ are the squares $s_k(\iota)^2$. Squaring \eqref{eq:spectral-3} gives \eqref{eq:spectral-1}, and the statement for $\Box_q$ follows by taking reciprocals. Finally,
\[
\sum_k\mu_k(N_q)^p\le C\sum_kk^{-ps/n}[\log(e+k)]^{2bp}<\infty\qquad(p>n/s).
\]
This completes the proof.
\end{proof}

The singular values of the canonical solution operator are controlled by the eigenvalues of the $\ddbar$-Neumann operator. The following corollary records this consequence, and it also records the top-degree case, in which only boundedness and pseudoconvexity are assumed.

\begin{corollary}[Canonical solutions and top degree]\label{cor:canonical-schatten}
\begin{enumerate}
\item[(i)] Under the hypotheses of Theorem~\ref{thm:spectral}, the canonical solution operator $S_q$ satisfies
\[
s_k(S_q)\le Ck^{-s/(2n)}[\log(e+k)]^b.
\]
Hence $S_q\in\mathfrak S_p$ for every $p>2n/s$.
\item[(ii)] If $D\subset\C^n$ is bounded and pseudoconvex, then $\mu_k(N_n)\le Ck^{-1/n}$ and $N_n\in\mathfrak S_p$ for every $p>n$.
\end{enumerate}
\end{corollary}

\begin{proof}
The closed-range complex gives $S_q=T_{q-1}^*N_q$, where $S_q$ is extended by zero on $(\ker T_q)^\perp$. By the spectral energy identity $Q_D(N_q^{1/2}f)=\|f\|_2^2$, the operator $T_{q-1}^*N_q^{1/2}$ is a contraction. Thus
\[
S_q=(T_{q-1}^*N_q^{1/2})N_q^{1/2},\qquad s_k(S_q)\le\mu_k(N_q)^{1/2},
\]
which proves (i). For (ii), we note that Proposition~\ref{prop:top-degree} gives a first-order estimate in top degree on every bounded pseudoconvex domain. The Fourier-truncation argument above, with this $H^1$ estimate in place of Lemma~\ref{lem:gagliardo-log}, gives (ii).
\end{proof}

\section{Bounded transforms of interior potentials}\label{sec:transforms}

In this section, we construct bounded plurisubharmonic weights by transforming a plurisubharmonic potential that is bounded below and whose gradient is controlled by its own Levi form. For a real $C^2$ function $\Psi$ we write $\mathcal L^\Psi(z;v)=\sum_{j,k}\Psi_{j\bar k}(z)v_j\bar v_k$ for its Levi form. The hypothesis \eqref{eq:transforms-1} below is the self-bounded-gradient condition of McNeal \cite{McN02a}. For $g(t)=e^{-t/(2A)}$, the weight $\lambda$ constructed in the proof of Theorem~\ref{thm:attenuation} is a multiple of $-e^{-L/(2A)}$. This is the exponential transform of McNeal \cite{McN02a}, see also \cite{Str10}. An integrable power profile $g$ loses only a logarithm (Corollary~\ref{cor:log-growth}). Estimates with logarithmic and more general gains from weights of this kind were developed by Khanh and Zampieri \cite{KZ10,KZ12}.

The following theorem converts a lower bound for the induced Hessian of the potential into a weighted estimate in a collar. The loss is the attenuation profile evaluated at the growth bound of the potential.

\begin{theorem}[Integrable attenuation]\label{thm:attenuation}
Let $D \Subset \C^n$ be pseudoconvex and let $\Psi \in C^2(D)$ be plurisubharmonic and bounded below. Suppose that
\begin{equation}\label{eq:transforms-1}
|\partial\Psi(z)(v)|^2 \le A \mathcal L^{\Psi}(z; v)
\end{equation}
for some $A > 0$. Fix $1 \le q \le n$, constants $B\in\R$ and $L_0\in\R$ with $L:=B+\Psi\ge L_0$ on $D$, and continuous functions $h\ge1$ and $\rho>0$ on $(0,t_0)$. Suppose that on $\{\delta_D<t_0\}$
\begin{equation}\label{eq:transforms-2}
H^{[q]}_\Psi \ge c\rho(\delta_D)^{-2}\operatorname{Id}, \qquad L \le C h(\delta_D).
\end{equation}
Let $g \in C^1([L_0, \infty))$ be positive, nonincreasing and integrable, and suppose that it satisfies
\begin{equation}\label{eq:transforms-3}
-g'(t) \le \frac{g(t)}{2A}.
\end{equation}
Then for every $0<t_1<t_0$ and $q\le q'\le n$ there is $C'$ such that
\begin{equation}\label{eq:transforms-4}
\int_{\{\delta_D<t_1\}} \rho(\delta_D)^{-2} g(Ch(\delta_D)) |u|^2 \le C' Q_D(u), \qquad u \in \mathcal V_{q'}(D).
\end{equation}
\end{theorem}

\begin{proof}
The function $\lambda(z) = -\int_{L(z)}^\infty g(t)\, dt$ is bounded and $C^2$, and the chain rule gives
\[
H_\lambda = g(L) H_\Psi + g'(L)\, \partial\Psi \otimes \overline{\partial\Psi}.
\]
By \eqref{eq:transforms-1}, we have $\partial\Psi \otimes \overline{\partial\Psi} \le A H_\Psi$. Since $g' \le 0$, multiplication by $g'$ reverses this inequality. Together with the bound on the derivative of the profile, this gives
\[
H_\lambda \ge \bigl( g(L) + A g'(L) \bigr) H_\Psi \ge \tfrac12 g(L) H_\Psi \ge 0.
\]
We now pass to the induced form in degree $q$. Using the monotonicity of $g$, we get
\[
H^{[q]}_\lambda \ge \tfrac{c}{2} \rho(\delta_D)^{-2} g(L)\operatorname{Id} \ge \tfrac{c}{2} \rho(\delta_D)^{-2} g(Ch(\delta_D))\operatorname{Id}.
\]
Corollary~\ref{cor:collars} yields \eqref{eq:transforms-4}. This proves the theorem.
\end{proof}

We now specialize to a power profile. When the potential has at most logarithmic growth at the boundary, the loss is a power of the logarithm of the boundary distance. When the lower bound for the Hessian is a power of the boundary distance, the resulting weight satisfies the hypotheses of Theorem~\ref{thm:zero-extension}.

\begin{corollary}[Logarithmic growth and power geometry]\label{cor:log-growth}
Let $D\Subset\C^n$ be pseudoconvex, and let $\Psi\in C^2(D)$ be plurisubharmonic, bounded below, and satisfy \eqref{eq:transforms-1}. Fix $1\le q\le n$ and a positive continuous function $\rho$ on $(0,t_0)$. Suppose that $\Psi\le C\ell(\delta_D)$ and $H^{[q]}_\Psi\ge c\rho(\delta_D)^{-2}\operatorname{Id}$ on $\{\delta_D<t_0\}$. Then for every $a>0$, $0<t_1<t_0$ and $q\le q'\le n$ there is $C_a$ such that
\begin{equation}\label{eq:transforms-5}
\int_{\{\delta_D < t_1\}} \frac{|u|^2}{\rho(\delta_D)^2 \ell(\delta_D)^{1+a}} \le C_a Q_D(u), \qquad u \in \mathcal V_{q'}(D).
\end{equation}
If moreover $\rho(t) \le C_0 t^\alpha$ with $0 < \alpha < 1$, then for every $b>1/2$ and $q\le q'\le n$,
\begin{equation}\label{eq:transforms-6}
\norm{E_0 u}^2_{H^{\alpha,-b}(\R^{2n})} \le C_b Q_D(u), \qquad u\in\mathcal V_{q'}(D).
\end{equation}
\end{corollary}

\begin{proof}
We choose $g(t) = t^{-1-a}$, after shifting $L$ so that $L \ge 2A(1+a)$. On this range the power profile satisfies the derivative condition. The resulting bounded transform is a positive multiple of $-L^{-a}$, and its Hessian is bounded below by $c_a \rho(\delta_D)^{-2} \ell(\delta_D)^{-1-a}$. Theorem~\ref{thm:attenuation} yields \eqref{eq:transforms-5}. For a power profile, the same bounded weight satisfies the hypotheses of Theorem~\ref{thm:zero-extension} with $s=\alpha$ and $2b=1+a$, and this gives \eqref{eq:transforms-6}.
\end{proof}

\begin{remark}[Other integrable losses]\label{rem:other-losses}
Other integrable profiles can be used as well. For large $t$, the function $g(t) = [t(\log t)^{1+a}]^{-1}$ is positive, decreasing and integrable, and after a shift it satisfies \eqref{eq:transforms-3}. Thus Theorem~\ref{thm:attenuation} also gives the stronger distance weight
\[
\frac{1}{\rho(\delta_D)^2 \ell(\delta_D) [\log(e + \ell(\delta_D))]^{1+a}}.
\]
The resulting estimate is a weighted $L^2$ estimate.
\end{remark}

\section{Bergman and directional geometry}\label{sec:bergman}

In this section, we apply the transforms of the previous section to the logarithm of the Bergman kernel. This potential is bounded below and has at most logarithmic growth at the boundary, and its Levi form is the Bergman metric. We first give a criterion that needs no convexity. We then show how the minimal-basis estimates of Nikolov, Pflug and Zwonek verify the hypotheses of this criterion from a bound on the directional radii.

\subsection{An explicit criterion without convexity}\label{sec:bergman-criterion}
In this subsection, we take the logarithm of the Bergman kernel as the potential of the previous section. Recall that $K_D(z,w)$ is the Bergman kernel of $D$, that $K_D(z)=K_D(z,z)$, and that
\[
\beta_D(z;v)^2=\mathcal{L}^{\log K_D}(z;v).
\]
For every bounded domain,
\begin{equation}\label{eq:bergman-1}
\Vol(D)^{-1}\le K_D(z)\le c_n\delta_D(z)^{-2n}.
\end{equation}
The lower bound follows by testing the constant function, and the upper bound follows from the mean-value inequality on the ball $B(z,\delta_D(z))$. We refer the reader, for instance, to \cite{Ber70} or \cite{JP13}. Hence $\log K_D$ is bounded below and has at most logarithmic growth at $\partial D$. The kernel $K_D$ is smooth by the reproducing-kernel representation, and $\beta_D$ is positive definite on a bounded domain, as one sees by testing the affine polynomials.

The following corollary is the criterion. Its hypothesis on the gradient of the logarithm of the kernel is the self-bounded-gradient condition for this particular potential, and its hypothesis on the eigenvalues of the Bergman metric depends on the form degree.

\begin{corollary}[Degree-dependent Bergman criterion]\label{cor:bergman-criterion}
Let $D\Subset\C^n$ be pseudoconvex and assume that, for some $A>0$,
\begin{equation}\label{eq:bergman-2}
|\partial\log K_D(z)(v)|^2\le A\beta_D(z;v)^2,\qquad z\in D,\ v\in\C^n.
\end{equation}
Let $\theta_1(z)\le\cdots\le\theta_n(z)$ be the Euclidean eigenvalues of the Bergman metric matrix of $D$ at $z$, let $1\le q\le n$, and let $\rho$ be positive and continuous on $(0,t_0)$. If, on $\{\delta_D<t_0\}$,
\begin{equation}\label{eq:bergman-3}
\theta_1(z)+\cdots+\theta_q(z)\ge c\rho(\delta_D(z))^{-2},
\end{equation}
then \eqref{eq:transforms-5} holds in every degree $q'\ge q$. If moreover $\rho(t)\le C_0t^\alpha$ with $0<\alpha<1$, then \eqref{eq:transforms-6} holds in every degree $q'\ge q$.
\end{corollary}

\begin{proof}
By \eqref{eq:bergman-1}, the potential $\Psi=\log K_D$ is bounded below, and it has at most logarithmic growth at the boundary. Its Hessian is the Bergman metric, so the potential is plurisubharmonic. The assumed gradient bound is the self-bounded-gradient condition for this potential, and the exterior-power eigenvalue formula \eqref{eq:prelim-2} turns the hypothesis on the eigenvalues into the lower bound for the induced Hessian. Therefore all hypotheses of Corollary~\ref{cor:log-growth} are verified. This proves the corollary.
\end{proof}

\begin{remark}\label{rem:sbg}
Hypothesis \eqref{eq:bergman-2} does not hold on every bounded pseudoconvex domain. Zimmer \cite[Proposition~1.11]{Zim21} gives a bounded domain in $\C^2$, biholomorphic to the bidisc, on which $\log K_D$ does not have bounded gradient in the Bergman metric. On bounded $\C$-convex domains the hypothesis holds (Lemma~\ref{lem:sbg}).
\end{remark}

\subsection{The \texorpdfstring{$\C$}{C}-convex inputs}\label{sec:cconvex-inputs}
In this subsection, we collect the inputs from the theory of minimal bases that verify the hypotheses of the criterion. Let $z\in D$. We choose a minimal orthonormal basis $e_1,\dots,e_n$ at $z$ as follows. We choose a nearest boundary point, and we repeat the choice in the successive complex orthogonal sections. We use the minimal basis in the form of Nikolov--Pflug--Zwonek \cite{NPZ11}. We refer the reader to Nikolov--Pflug--Thomas \cite{NPT13} for the distinction between minimal and maximal bases and for the correction of the orthogonality issue for maximal bases. We write $\tau_1\le\cdots\le\tau_n$ for the corresponding distances. Then $\tau_1=\delta_D(z)$ and
\begin{equation}\label{eq:bergman-4}
\tau_j=r_D(z;e_j).
\end{equation}
Indeed, the coordinate disc of radius $\tau_j$ lies in $D$, and the selected endpoint belongs to $\partial D$. We shall use the following estimates of Nikolov, Pflug and Zwonek \cite{NPZ11}.

\begin{result}[Nikolov--Pflug--Zwonek, \cite{NPZ11}]\label{res:npz}
For a bounded $\C$-convex domain there are constants $a_n,c_n,C_n>0$ depending only on $n$ such that, at every point $z$, for every minimal basis $e_1,\dots,e_n$ at $z$ and every $v=\sum_jv_je_j$,
\begin{align}
\Bigl\{z+\sum_j w_je_j:\sum_j\frac{|w_j|}{\tau_j}<1\Bigr\}&\subset D,\label{eq:bergman-5}\\
K_D(z)&\ge a_n\prod_j\tau_j^{-2},\label{eq:bergman-6}\\
c_n\sum_j\frac{|v_j|}{\tau_j}\le\beta_D(z;v)&\le C_n\sum_j\frac{|v_j|}{\tau_j},\label{eq:bergman-7}\\
\frac{1}{4r_D(z;v)}\le\beta_D(z;v)&\le\frac{C_n}{r_D(z;v)}\quad(|v|=1).\label{eq:bergman-8}
\end{align}
\end{result}

No boundary regularity is assumed in Result~\ref{res:npz}. The hypothesis of \cite{NPZ11} that the domain contains no complex affine line is satisfied, because the domain is bounded. The minimal-basis statements are \cite[Theorem~13, Proposition~14 and Lemma~15]{NPZ11}. For convex domains, the kernel estimate \eqref{eq:bergman-6} goes back to McNeal \cite{McN94} and Nikolov--Pflug \cite{NP03}. The directional comparisons \eqref{eq:bergman-8} follow from \cite[Proposition~1 and Theorem~12]{NPZ11} together with Hahn's inequality $\beta_D\ge\gamma_D$ between the Bergman and Carath\'eodory metrics \cite{Hahn78}.

We first record that the domain is pseudoconvex. By \eqref{eq:bergman-6}, together with $\tau_1=\delta_D$ and $\tau_j\le\operatorname{diam}(D)$, we have $K_D\ge c_D\delta_D^{-2}$. Hence $\log K_D$ is a smooth strictly plurisubharmonic exhaustion of $D$, and $D$ is pseudoconvex. In particular, Theorem~\ref{thm:zero-extension} applies to every bounded $\C$-convex domain.

For bounded $\C$-convex domains, the gradient estimate \eqref{eq:bergman-2} is due to Zimmer \cite[Proposition~4.12]{Zim21}. We give a short proof with minimal bases. The proof shows that the constant depends only on $n$.

\begin{lemma}[Self-bounded gradient]\label{lem:sbg}
Every bounded $\C$-convex domain satisfies \eqref{eq:bergman-2} with a constant $A$ depending only on $n$.
\end{lemma}

\begin{proof}
Fix $z\in D$ and a minimal basis at $z$, and set $s_j=\tau_j/(4n)$. By \eqref{eq:bergman-5}, the polydisc centered at $z$ with radii $2s_j$ lies in $D$. If $w$ lies in the smaller polydisc with radii $s_j$, then the polydisc centered at $w$ with these same radii lies in $D$. By the mean-value inequality and \eqref{eq:bergman-6}, we get
\[
K_D(w)\le\frac{1}{\pi^n\prod_js_j^2}\le C_nK_D(z).
\]
By the reproducing-kernel inequality, $|K_D(w,z)|^2\le K_D(w)K_D(z)$. We apply Cauchy's estimate to the holomorphic function $w\mapsto K_D(w,z)$ on the smaller polydisc, and we get
\[
|\partial_{e_j}K_D(z,z)|\le C_nK_D(z)/\tau_j.
\]
Since the second kernel variable is antiholomorphic, the first-variable derivative is the $(1,0)$ derivative of the diagonal. Consequently,
\[
|\partial\log K_D(z)(v)|\le C_n\sum_j|v_j|/\tau_j\le C_n'\beta_D(z;v).
\]
Squaring proves \eqref{eq:bergman-2}.
\end{proof}

We now state the degree-dependent criterion. Its hypothesis is a bound for a single minimal-basis radius, whose index depends on the form degree, and not a bound in all directions.

\begin{theorem}[Partial directional collapse]\label{thm:partial-collapse}
Let $D\Subset\C^n$ be $\C$-convex and $1\le q\le n$, and let $\rho$ be positive and continuous on $(0,t_0)$. Suppose that for every $z$ with $\delta_D(z)<t_0$ some minimal basis at $z$ satisfies
\begin{equation}\label{eq:bergman-9}
\tau_{n-q+1}(z)\le\rho(\delta_D(z)).
\end{equation}
Then \eqref{eq:transforms-5} holds in every degree $q'\ge q$. If moreover $\rho(t)\le C_0t^\alpha$ with $0<\alpha<1$, then \eqref{eq:transforms-6} holds for every $b>1/2$ in every degree $q'\ge q$.
\end{theorem}

\begin{proof}[Proof of Theorem~\ref{thm:partial-collapse}]
We square the lower estimate in \eqref{eq:bergman-7} and get
\[
H_{\log K_D}\ge c_n\operatorname{diag}(\tau_1^{-2},\dots,\tau_n^{-2})
\]
in the chosen basis. Therefore the least eigenvalue of the induced $q$-form is at least
\[
c_n\sum_{j=n-q+1}^{n}\tau_j^{-2}\ge c_n\tau_{n-q+1}^{-2}\ge c_n\rho(\delta_D)^{-2}.
\]
Lemma~\ref{lem:sbg} supplies the gradient hypothesis, and Corollary~\ref{cor:bergman-criterion} gives the assertions. We note that the conclusion concerns the Bergman matrix, which does not depend on the basis, so pointwise choices of bases suffice. This completes the proof.
\end{proof}

The following proposition shows that the exponent in the directional power bound is at most one half.

\begin{proposition}[The smallest directional power]\label{prop:m-at-least-2}
If a bounded domain $D\Subset\C^n$, $n\ge2$, satisfies \eqref{eq:mCconvex} for some $m>0$, then $m\ge2$.
\end{proposition}

\begin{proof}
Fix $a\in D$ and set $R=\delta_D(a)$. Let $p\in\partial D$ be a boundary point with $|p-a|=R$, put $\nu=(p-a)/R$ and $z_t=p-t\nu$ for $0<t<R$, and let $v$ be a unit complex vector orthogonal to $\nu$. Such a vector exists since $n\ge2$. Since $B(a,R)\subset D$, we have
\[
r_D(z_t;v)\ge\sqrt{2Rt-t^2},\qquad\delta_D(z_t)\le t.
\]
By \eqref{eq:mCconvex}, it follows that $\sqrt{Rt}\le C_0t^{1/m}$ for small $t$. Dividing by $t^{1/2}$ and letting $t\downarrow0$, we get a contradiction if $m<2$, because the right side then tends to zero while the left side does not.
\end{proof}

\section{Proof of Theorem~\ref{thm:directional} and comparison with Zimmer's estimates}\label{sec:zimmer-comparison}

In this section, we assemble the proof of Theorem~\ref{thm:directional} from the results of the previous three sections. We then introduce the directional index, we compare our estimates with those of Zimmer on convex domains, and we give a local form of the estimate.

\subsection{Proof of Theorem~\ref{thm:directional}}\label{sec:proof-directional}

The proof combines Proposition~\ref{prop:m-at-least-2}, Theorem~\ref{thm:partial-collapse} in degree one, and Theorem~\ref{thm:spectral}.

\begin{proof}[Proof of Theorem~\ref{thm:directional}]
By Proposition~\ref{prop:m-at-least-2}, we have $m\ge2$. By \eqref{eq:bergman-4}, the directional bound \eqref{eq:mCconvex} gives $\tau_n\le C_0\delta_D^{1/m}$ in every minimal basis. We apply Theorem~\ref{thm:partial-collapse} with $q=1$ and $\alpha=1/m$, and we get \eqref{eq:directional-estimate} in every degree. Theorem~\ref{thm:spectral} with $s=1/m$ then gives \eqref{eq:eigenvalues} and $N_q\in\mathfrak S_p$ for $p>nm$.
\end{proof}

The directional power bound suggests the following exponent.

\begin{definition}\label{def:dir}
Let $D\Subset\C^n$, $n\ge2$, be $\C$-convex. Its \emph{directional index} is
\[
\alpha_{\mathrm{dir}}(D)=\sup\{\alpha>0:\text{there are }C,t_0>0\text{ with }r_D(z;v)\le C\delta_D(z)^\alpha\text{ for }\delta_D(z)<t_0,\ v\ne0\},
\]
with the convention $\sup\varnothing=0$. We call such an $\alpha$ \emph{admissible}. Since $D$ is bounded, the same number is obtained if the bound is required on all of $D$.
\end{definition}

The following corollary expresses the estimates in terms of the directional index. The index is at most one half, every order below the index gives a zero-extension estimate, and the index itself gives one with logarithmic loss when the supremum is attained.

\begin{corollary}[Directional index and attainment]\label{cor:index}
Let $D\Subset\C^n$, $n\ge2$, be $\C$-convex. Then $\alpha_{\mathrm{dir}}(D)\le1/2$. For every $0<s<\alpha_{\mathrm{dir}}(D)$ and $1\le q\le n$,
\[
\norm{E_0u}^2_{H^s(\R^{2n})}\le C_sQ_D(u),\qquad u\in\mathcal V_q(D).
\]
If the supremum is attained at $\alpha>0$, then for every $b>1/2$ and $1\le q\le n$,
\[
\norm{E_0u}^2_{H^{\alpha,-b}(\R^{2n})}\le C_bQ_D(u),\qquad u\in\mathcal V_q(D).
\]
If $D$ is $m$-$\C$-convex for some $m$, then $\alpha_{\mathrm{dir}}(D)=1/m^{\C}_\star(D)$, where $m^{\C}_\star(D)=\inf\{m>0:D\text{ is }m\text{-}\C\text{-convex}\}$.
\end{corollary}

\begin{proof}
By Proposition~\ref{prop:m-at-least-2}, we have $\alpha_{\mathrm{dir}}(D)\le1/2$. If $s<\alpha_{\mathrm{dir}}(D)$, we choose an admissible $\alpha>s$, and we apply Theorem~\ref{thm:directional} with $m=1/\alpha$. If the supremum is attained, the logarithmic assertion follows in the same way. The last identity holds because $\alpha$ is admissible exactly when $D$ is $(1/\alpha)$-$\C$-convex. This proves the corollary.
\end{proof}

\subsection{Comparison with Zimmer's estimates}
In this subsection, we compare the gains given by Theorem~\ref{thm:directional} with the gains obtained by Zimmer on convex domains. We first recall his two parameters. On a bounded convex domain $D$, we define
\begin{equation}\label{eq:mstar}
m_\star(D)=\inf\{m\ge 1: D\text{ is }m\text{-convex}\}.
\end{equation}
This is the quantity denoted $m(D)$ by Zimmer \cite{Zim22}. Zimmer also defined $\alpha_\star(D)$ as the infimum of the $\alpha\ge 1$ for which, for each fixed $z_0\in D$, there is $B$ such that
\begin{equation}\label{eq:alphastar}
k_D(z,z_0)\le B+\frac{\alpha}{2}\log\frac{1}{\delta_D(z)},\qquad z\in D.
\end{equation}
For bounded convex domains whose K\"ahler--Einstein distance is Gromov hyperbolic, Zimmer \cite[Theorem~16.1]{Zim22} proved a subelliptic estimate of every gain
\begin{equation}\label{eq:zimmer-range}
0<\varepsilon<\frac{1}{\alpha_\star(D)m_\star(D)},
\end{equation}
and he showed that $\alpha_\star(D)=1$ when $\partial D$ is $C^1$ \cite[Proposition~16.2]{Zim22}. The input from Gromov hyperbolicity is that such a domain is $m$-convex for some $m$ \cite[Corollary~7.2]{Zim22}. More precisely, for every $R>0$ there are $C>0$ and $m\ge1$ such that \eqref{eq:mCconvex} holds on $D\cap B(0,R)$.

The following corollary applies Theorem~\ref{thm:directional} to the convex domains treated by Zimmer. Gromov hyperbolicity of any one of the three invariant distances in its statement suffices.

\begin{corollary}[Gromov hyperbolic convex domains]\label{cor:gh-convex}
Let $D\Subset\C^n$, $n\ge 2$, be convex. Suppose that its Kobayashi, Bergman, or complete K\"ahler--Einstein distance is Gromov hyperbolic. Then $m_\star(D)<\infty$, and the conclusions of Theorem~\ref{thm:directional} hold for every $m\ge1$ for which $D$ is $m$-convex. For every $0<s<1/m_\star(D)$ and $1\le q\le n$,
\begin{equation}\label{eq:gh-convex}
\norm{E_0u}^2_{H^s(\R^{2n})}\le C_sQ_D(u),\qquad u\in\mathcal V_q(D),
\end{equation}
and $N_q\in\mathfrak S_p$ for every $p>nm_\star(D)$. If the infimum in \eqref{eq:mstar} is attained, then \eqref{eq:directional-estimate} and \eqref{eq:eigenvalues} hold with $m=m_\star(D)$.
\end{corollary}

\begin{proof}[Proof of Corollary~\ref{cor:gh-convex}]
On a bounded convex domain, the Kobayashi, the Bergman and the complete K\"ahler--Einstein distances are bi-Lipschitz equivalent by a theorem of Frankel \cite{Fra91}, see also \cite[Remark~1.2]{Zim22}. Hence Gromov hyperbolicity of any one of them gives Gromov hyperbolicity of the Kobayashi distance. By \cite[Corollary~7.2]{Zim22}, for every $R>0$ there are $C>0$ and $m\ge1$ such that \eqref{eq:mCconvex} holds on $D\cap B(0,R)$. Since $D$ is bounded, $D$ itself is such a piece, and so $D$ is $m$-convex for some $m$. Therefore Theorem~\ref{thm:directional} applies. If $s<1/m_\star(D)$, the definition of the infimum supplies an $m<1/s$ for which $D$ is $m$-convex, and the estimate of gain $s$ follows. For $p>nm_\star(D)$, we choose such an $m<p/n$. If the infimum is attained, Theorem~\ref{thm:directional} applies directly with $m=m_\star(D)$. This completes the proof.
\end{proof}

Thus Corollary~\ref{cor:gh-convex} gives every gain $s<1/m_\star(D)$, without the factor $\alpha_\star(D)$ in \eqref{eq:zimmer-range}. The two ranges agree when $\alpha_\star(D)=1$, in particular for $C^1$ boundaries. The improvement comes from the weight and not from Theorem~\ref{thm:zero-extension}. Indeed, bounded convex domains are Lipschitz, and on them the same gains follow from the weight of Section~\ref{sec:transforms} and Straube's theorem \cite[Theorem~1]{Str97}, as we explain in Remark~\ref{rem:straube-route}. What Theorem~\ref{thm:zero-extension} adds is the logarithmic endpoint and the absence of boundary assumptions. We also note that $m$-convexity does not imply Gromov hyperbolicity \cite[Example~7.3]{Zim22}. The following lens is an instance.

\begin{remark}[A lens]\label{rem:lens}
Let $n\ge2$, let $e_1,\ldots,e_n$ be the standard coordinate vectors, and let $D_{\mathrm{lens}}=\bigcap_{j=1}^nB(e_j,1)$. This is the special case of Zimmer's example \cite[Example~7.3]{Zim22} in which the domains $\Omega_j$ there are the balls $B(e_j,1)$. The domain $D_{\mathrm{lens}}$ contains the points $(t,\ldots,t)$ with $0<t<2/n$, and it is bounded and convex, hence Lipschitz. For a ball $B(c,1)$, a point $z$ in it and a unit vector $v$, a direct computation gives $r_{B(c,1)}(z;v)\le\sqrt{1-|z-c|^2}\le\sqrt{2\delta_{B(c,1)}(z)}$. Since $r_{D_{\mathrm{lens}}}(z;v)=\min_jr_{B(e_j,1)}(z;v)$ and $\delta_{D_{\mathrm{lens}}}(z)=\min_j\delta_{B(e_j,1)}(z)$, the lens is $2$-convex, and Theorem~\ref{thm:directional} applies to it with $m=2$. A subelliptic estimate on $D_{\mathrm{lens}}$ also follows from \cite[Theorem~1.3]{Zim22}, which treats finite intersections of Gromov hyperbolic convex domains. On the other hand, the dilations $tD_{\mathrm{lens}}$, $t>0$, increase to $\HH^n$, which is biholomorphic to the polydisc, and the standard rescaling argument \cite{GS18,NTT16,Zim16} shows that $k_{D_{\mathrm{lens}}}$ is not Gromov hyperbolic. Thus the class of domains to which Theorem~\ref{thm:directional} applies strictly contains the class of Gromov hyperbolic convex domains of Corollary~\ref{cor:gh-convex}.
\end{remark}

\begin{remark}[An alternative proof in the Lipschitz case]\label{rem:straube-route}
We describe the alternative route. On a bounded Lipschitz $m$-$\C$-convex domain, Lemma~\ref{lem:sbg}, the metric lower bound \eqref{eq:bergman-8} and the chain-rule construction in Theorem~\ref{thm:attenuation} yield a bounded plurisubharmonic $\lambda$ with
\[
H_\lambda\ge c\delta_D^{-2/m}\ell(\delta_D)^{-1-a}\operatorname{Id}\ge c_s\delta_D^{-2s}\operatorname{Id}\qquad(0<s<1/m).
\]
The second inequality holds because $t^{-2(1/m-s)}\ell(t)^{-1-a}\to\infty$ as $t\downarrow0$. Since $s<1/2$, Straube's local Lipschitz theorem \cite[Theorem~1]{Str97}, a finite covering of the boundary and interior elliptic estimates give the global restriction estimate of order $s$. Boundedness of zero extension on Lipschitz domains for $0<s<1/2$ (Subsection~\ref{sec:intro-estimates}) then gives the zero-extension estimate. This route gives neither the logarithmic endpoint nor estimates on domains without boundary regularity.
\end{remark}

\subsection{A local form}
In this subsection, we show that the weight of Theorem~\ref{thm:attenuation} also gives a local statement. The directional power bound is assumed only near one boundary point. The conclusion is a weighted estimate near that point, together with a local subelliptic gain when the boundary is Lipschitz there.

\begin{corollary}[Local formulation]\label{cor:local}
Let $D\Subset\C^n$, $n\ge2$, be $\C$-convex, let $p\in\partial D$, and let $U$ be a neighborhood of $p$. Suppose that there are $m_p\ge2$ and $C_p,t_0>0$ such that $r_D(z;v)\le C_p\delta_D(z)^{1/m_p}$ for all $z\in U\cap D$ with $\delta_D(z)<t_0$ and all $v\ne0$.
\begin{enumerate}
\item[(i)] For every $a>0$ there are a neighborhood $U'\subset U$ of $p$ and $C>0$ such that
\[
\int_{U'\cap D}\delta_D^{-2/m_p}\ell(\delta_D)^{-1-a}|u|^2\le CQ_D(u),\qquad u\in\mathcal V_{q'}(D),\quad 1\le q'\le n.
\]
\item[(ii)] If $\partial D$ is Lipschitz near $p$, then every $0<s<1/m_p$ is a local subelliptic gain near $p$ in the sense of \cite[Theorem~1]{Str97}, that is, there are a neighborhood $V$ of $p$ and $C_s>0$ such that $\norm{u}^2_{H^s(D)}\le C_sQ_D(u)$ for every $u\in\mathcal V_1(D)$ supported in $V\cap\overline D$.
\end{enumerate}
\end{corollary}

\begin{proof}
Shrinking $U$, we may assume that $\delta_D<t_0$ on $U\cap D$. Let $\lambda$ be the bounded transform of the proof of Theorem~\ref{thm:attenuation} with $\Psi=\log K_D$ and $g(t)=t^{-1-a}$, as in Corollary~\ref{cor:log-growth}. Then $H_\lambda\ge\frac12g(L)H_{\log K_D}\ge0$ on $D$. We now bound the Hessian from below on the given neighborhood. By \eqref{eq:bergman-8}, we have $\beta_D(z;v)^2\ge(16\,r_D(z;v)^2)^{-1}\ge c\,\delta_D(z)^{-2/m_p}$ for $z\in U\cap D$ and $|v|=1$, and by \eqref{eq:bergman-1} we have $g(L)\ge c\,\ell(\delta_D)^{-1-a}$. Hence $H_\lambda\ge c\,\delta_D^{-2/m_p}\ell(\delta_D)^{-1-a}\operatorname{Id}$ on $U\cap D$. Let $\chi$ be a continuous cutoff with values in $[0,1]$, supported in $U$ and equal to one on $U'$. Since $H_\lambda\ge0$ everywhere, the field $A=\chi c\,\delta_D^{-2/m_p}\ell(\delta_D)^{-1-a}\operatorname{Id}$ is a valid global lower bound, and Theorem~\ref{thm:matrix} gives (i) in degree one. The higher degrees follow by the eigenvalue inequality of Section~\ref{sec:complex}. For (ii), we have $\delta_D^{-2/m_p}\ell(\delta_D)^{-1-a}\ge c_s\delta_D^{-2s}$ for $s<1/m_p$, as in Remark~\ref{rem:straube-route}, and Straube's theorem \cite[Theorem~1]{Str97} applies near $p$. This completes the proof.
\end{proof}

\part{The three conditions on smooth \texorpdfstring{$\C$}{C}-convex domains}\label{part:three}

\section{Subellipticity, directional radii and the proof of Theorem~\ref{thm:main}}\label{sec:necessity}
In this section, we shall prove that a subelliptic estimate forces a power bound for the directional radii, and we shall complete the proof of Theorem~\ref{thm:main}. The argument applies the canonical solution operator to normalized Bergman kernels. Before we begin, we record the two tools on which the argument rests, namely the regularity theorem of Kohn and Nirenberg and the fractional Hardy inequality.

The following result is due to Kohn and Nirenberg.

\begin{result}[Kohn--Nirenberg regularity]\label{res:smooth-regularity}
Let $D\Subset\C^n$ be pseudoconvex with $C^\infty$ boundary. Suppose that a subelliptic estimate of gain $\varepsilon>0$ holds in degree $q$. Then, for every $a\ge0$,
\[
 N_q:H^a_{(0,q)}(D)\longrightarrow H^{a+2\varepsilon}_{(0,q)}(D),
\]
and both $\ddbar^*N_q$ and $\ddbar N_q$ map $H^a_{(0,q)}(D)$ boundedly into the corresponding spaces of order $a+\varepsilon$. In particular, on closed $(0,1)$-forms,
\begin{equation}\label{eq:canonical-gain}
 S_1=\ddbar^*N_1:L^2_{(0,1)}(D)\cap\ker\ddbar\longrightarrow H^\varepsilon(D)
\end{equation}
is bounded whenever the estimate holds in degree one.
\end{result}
The result above is the pseudolocal regularity theorem of Kohn and Nirenberg \cite{KN65}, which is proved by elliptic regularization. The global form stated above follows from the pseudolocal one by a partition of unity, see \cite[Theorem~8]{BS99}. We also note that on $C^\infty$ domains the estimate on $\mathcal V_q(D)$ is equivalent, by density, to the same estimate on forms in $\dom\ddbar^*$ that are smooth up to the boundary.

The second tool is a fractional Hardy inequality on bounded Lipschitz domains. We state it in the form in which we shall use it.

\begin{lemma}[Fractional Hardy estimate]\label{lem:fractional-hardy}
Let $D\subset\R^d$ be a bounded Lipschitz domain and $0<s<1/2$. There is a constant $C_s$, depending on $s$ and $D$, such that
\begin{equation}\label{eq:fractional-hardy}
 \int_D\delta_D(x)^{-2s}|h(x)|^2\,dx\le C_s\|h\|_{H^s(D)}^2,\qquad h\in H^s(D).
\end{equation}
\end{lemma}
This Hardy inequality is classical, and we refer the reader to Grisvard \cite{Gri85}. We recall the short proof.
\begin{proof}
We first note that, for this range of $s$, zero extension is bounded from the restriction space $H^s(D)$ to $H^s(\R^d)$. To see this, observe that $C_c^\infty(D)$ is dense in $H^s(D)$ for $s\le1/2$ \cite[Chapter~3]{McL00}. On the closure of $C_c^\infty(D)$, zero extension is bounded by \cite[Theorem~3.33]{McL00} and \cite[Lemma~2.1]{HM17}. Next, we bound the weight by an integral over the exterior of the domain. The exterior-cone property of a bounded Lipschitz domain gives
\[
 \int_{\R^d\setminus D}|x-y|^{-d-2s}\,dy\ge c_s\delta_D(x)^{-2s}
\]
for $x$ in a boundary collar. Therefore, in that collar, the left side of \eqref{eq:fractional-hardy} is bounded by the part of the Gagliardo seminorm of $E_0h$ in which $x\in D$ and $y\notin D$. In the remaining compact interior the weight is bounded. The Gagliardo characterization of $H^s$ \cite{DNPV12,Gag57} and the zero-extension bound now give the inequality of the lemma. This proves the lemma.
\end{proof}

We are now ready to present the principal result of this section. It shows that a subelliptic estimate in degree one forces a power bound for the directional radii, with any exponent below the gain and below one half, and hence a bound for the D'Angelo one-type.

\begin{theorem}[Directional necessity from canonical solutions]\label{thm:analytic-necessity}
Let $D\Subset\C^n$, $n\ge2$, be $\C$-convex with $C^\infty$ boundary. Suppose that a subelliptic estimate of gain $\varepsilon>0$ holds on $(0,1)$-forms. For every $0<s<\min\{\varepsilon,1/2\}$ there is $C_s$ such that
\begin{equation}\label{eq:necessary-radius}
 r_D(z;v)\le C_s\delta_D(z)^s,\qquad z\in D,\quad v\ne0.
\end{equation}
Consequently, $\Delta_1(\partial D,p)\le1/s$ for every $p\in\partial D$. In particular, $\partial D$ has finite D'Angelo one-type.
\end{theorem}
\begin{proof}
Fix $s$ in the stated range. We first set up a polydisc adapted to a minimal basis at the given point. At $z\in D$ choose a minimal basis $e_1,\ldots,e_n$, with radii $\tau_1\le\cdots\le\tau_n$, as in Result~\ref{res:npz}. We choose $e_1$ to point from $z$ to a nearest boundary point $p$, so that $p=z+\delta_D(z)e_1$ and $\tau_1=\delta_D(z)$. Fix $c\in(0,1/(2n))$. The affine map
\[
 \Phi_z(\zeta)=z+c\sum_{j=1}^n\tau_j\zeta_je_j
\]
sends $\D^n$ into $D$. We write $P_z=\Phi_z(\D^n)$ for the image polydisc, and we put
\[
 J_z=|\det\Phi_z'|^2=c^{2n}\prod_{j=1}^n\tau_j^2.
\]
The kernel estimate in Result~\ref{res:npz} gives
\begin{equation}\label{eq:kernel-polydisc-mass}
 J_zK_D(z,z)\ge c_1>0.
\end{equation}
All constants here are independent of $z$ and of the choice of minimal basis.

Next, we show that the polydisc lies in a collar whose width is comparable to the boundary distance at its center. The complex tangent hyperplane $H=p+e_1^{\perp_\C}$ is disjoint from $D$. Hence, for $w=\Phi_z(\zeta)\in P_z$, we have
\begin{equation}\label{eq:polydisc-collar}
 \delta_D(w)\le\dist(w,H)
 =\delta_D(z)|1-c\zeta_1|\le(1+c)\delta_D(z).
\end{equation}
We now apply the canonical solution operator to a datum built from the normalized Bergman kernel. Let
\[
 f_z(w)=\frac{K_D(w,z)}{\sqrt{K_D(z,z)}}.
\]
Then $f_z$ is holomorphic, $\|f_z\|_2=1$, and $f_z(z)=\sqrt{K_D(z,z)}$. In the unitary coordinates given by the basis, let $\eta_n$ be the constant $(0,1)$-covector $d\bar w_n$. The form $f_z\eta_n$ is closed and has norm one. Set
\[
 h_z=S_1(f_z\eta_n).
\]
Since the chosen exponent is smaller than the gain, \eqref{eq:canonical-gain} gives $\|h_z\|_{H^s(D)}\le C_s$ uniformly in $z$. Moreover, the equation satisfied by the canonical solution, pulled back by the affine map, reads
\[
 \frac{\partial(h_z\circ\Phi_z)}{\partial\bar\zeta_n}
 =c\tau_n(f_z\circ\Phi_z)
\]
in distributions on $\D^n$.

We now bound the mass of the solution on the polydisc from below. To this end, choose a fixed nonnegative $\chi\in C_c^\infty(\D^n)$ which is a product of radial functions and has integral one. Since the pullback of the kernel is holomorphic on the polydisc, the mean-value property, applied in each variable, gives
\[
 \int_{\D^n}(f_z\circ\Phi_z)\chi=f_z(z).
\]
To express the left side through the solution, we use the equation for the pullback of the solution. Distributional integration by parts and the Cauchy--Schwarz inequality yield
\[
 c\tau_n\sqrt{K_D(z,z)}
 \le \|h_z\circ\Phi_z\|_{L^2(\D^n)}
       \|\partial\chi/\partial\bar\zeta_n\|_{L^2(\D^n)}.
\]
After squaring and multiplying by $J_z$, \eqref{eq:kernel-polydisc-mass} gives
\begin{equation}\label{eq:solution-mass-lower}
 \int_{P_z}|h_z|^2\ge c_2\tau_n^2.
\end{equation}
On the other hand, by \eqref{eq:polydisc-collar}, the boundary distance on the polydisc is at most a constant times the boundary distance at its center. Together with Lemma~\ref{lem:fractional-hardy} and the uniform bound on the Sobolev norm of the solution, this gives
\[
 \int_{P_z}|h_z|^2
 \le C\delta_D(z)^{2s}\int_D\delta_D(w)^{-2s}|h_z(w)|^2\,dV(w)
 \le C_s\delta_D(z)^{2s}.
\]
Combining the two inequalities above, we get $\tau_n\le C_s\delta_D(z)^s$. It remains to pass from the largest radius of the minimal basis to the directional radius. By the metric comparisons in Result~\ref{res:npz},
\[
 \frac{1}{r_D(z;v)}\asymp\frac{\beta_D(z;v)}{|v|}
 \ge\frac{c_3}{\tau_n},\qquad v\ne0.
\]
This proves \eqref{eq:necessary-radius}. Finally, Lemma~\ref{lem:full-curve-type}, applied with $\alpha=s$, gives $\Delta_1(\partial D,p)\le1/s$, and this bound includes singular curve germs.
\end{proof}

\begin{remark}
There is another route to Theorem~\ref{thm:analytic-necessity}. A subelliptic estimate of gain $\varepsilon$ is known to give, for every $\eta>0$, the lower bound $\beta_D(z;v)\gtrsim\delta_D(z)^{-\varepsilon+\eta}|v|$ for the Bergman metric. This is a theorem of McNeal \cite{McN92a}, and a statement of this kind is also given by Khanh--Zampieri \cite[Theorem~1.4]{KZ12}. Together with \eqref{eq:bergman-8}, such a lower bound gives another proof of Theorem~\ref{thm:analytic-necessity}. We do not use this route. The proof given above uses only the regularity theorem of Kohn and Nirenberg and the fractional Hardy inequality. We also note that the bound on the type can be made explicit. Letting $s\uparrow\min\{\varepsilon,1/2\}$ in Theorem~\ref{thm:analytic-necessity}, we obtain the bound $\Delta_1(\partial D,p)\le\max\{2,1/\varepsilon\}$ on $\C$-convex domains.
\end{remark}

We next record the classical upper bound for the gain of a subelliptic estimate in degree one. We shall need it when we identify the critical exponents.

\begin{lemma}[The upper half-order bound]\label{lem:upper-half-order}
Let $D\Subset\C^n$, $n\ge2$, be a domain with $C^\infty$ boundary. A subelliptic estimate on $(0,1)$-forms cannot have gain greater than $1/2$.
\end{lemma}
\begin{proof}
We shall test the estimate on oscillating forms concentrated at a boundary point at which the domain lies inside a tangent ball. To find such a point, choose a closed Euclidean ball containing $\overline D$ and tangent to $\partial D$ at a point where the distance from its center is maximal on $\overline D$. Translate and rotate this point to the origin, with outward normal in the positive real $z_n$ direction. In a fixed neighborhood write
\[
 D=\{x_n+F(y)<0\},\qquad
 y=(z',y_n)\in\R^{2n-1},\qquad r=x_n+F(y),
\]
where $F(0)=0$, $dF(0)=0$. Since the boundary lies inside the tangent ball and $F$ is smooth, after shrinking the neighborhood there are positive constants $c,C$ such that
\begin{equation}\label{eq:outer-ball-graph}
 c|y|^2\le F(y)\le C|y|^2,\qquad |dF(y)|\le C|y|.
\end{equation}
We now define the test forms. Set $\rho=x_n+F(y)$ and $\beta(y)=\rho_{z_1}/\rho_{z_n}$. The denominator is nonzero in the neighborhood, so $\beta$ is smooth, and $|\beta(y)|\le C|y|$. Choose real cutoffs $\chi(y)$ and $\chi_0(r)$ supported in the coordinate neighborhood, equal to one near $y=0$ and $r=0$, respectively. We choose the normal cutoff so that it vanishes for $r\le-a$ for some fixed $a>0$. For $t>0$ small, put
\[
 u_t=\chi(y)\chi_0(r)e^{z_n/t}
       \bigl(d\bar z_1-\beta(y)d\bar z_n\bigr).
\]
We extend this form by zero away from the coordinate patch. The result is a smooth form on $\overline D$. The coefficients are chosen so that $\sum\rho_{z_j}(u_t)_j=0$ on the boundary, and therefore the form belongs to the domain of $\ddbar^*$.

We first compute the mass and the energy of the test form. The change of variables $(y,r)\mapsto(z',r-F(y)+iy_n)$ has real Jacobian one. The Gaussian bounds \eqref{eq:outer-ball-graph} give
\begin{equation}\label{eq:oscillating-mass}
 \|u_t\|_2^2\asymp t^{n+1/2},\qquad
 \int |y|^2\chi(y)^2e^{-2F(y)/t}\,dy
 \le Ct\int\chi(y)^2e^{-2F(y)/t}\,dy.
\end{equation}
Derivatives of the cutoffs contribute bounded multiples of this mass, since the terms supported away from $y=r=0$ are exponentially smaller. The exponential is holomorphic, so it contributes no derivative to $\ddbar u_t$. In $\ddbar^*u_t$, the only derivative of the exponential is multiplied by $\beta/t$. The second estimate in \eqref{eq:oscillating-mass}, together with the boundedness of the derivatives of $\beta$, therefore gives
\begin{equation}\label{eq:oscillating-energy}
 Q_D(u_t)+\|u_t\|_2^2\le Ct^{-1}\|u_t\|_2^2.
\end{equation}

We next bound the Sobolev norm of the test form from below. The lower bound comes from the tangential frequency of the first coefficient, and we obtain it for the restriction norm, without zero extension across the boundary. In the flattened variables that coefficient is
\[
 v_t(y,r)=\chi(y)\chi_0(r)
       e^{(r-F(y)+iy_n)/t},\qquad r<0.
\]
Writing $D_{y_n}=-i\partial_{y_n}$, \eqref{eq:oscillating-mass} implies
\[
 \|(D_{y_n}-t^{-1})v_t\|_2^2\le Ct^{-1}\|v_t\|_2^2.
\]
Taking the partial Fourier transform in $y$, we see that the fraction of the $L^2$ mass of this coefficient outside the band
$|\xi_{y_n}-t^{-1}|\le(2t)^{-1}$ is at most $Ct$. On this frequency band $|\xi_{y_n}|\ge(2t)^{-1}$. Consequently, for every fixed $0<\sigma<1$,
\[
 \int_{-a}^0\|v_t(\cdot,r)\|_{H^\sigma(\R^{2n-1})}^2\,dr
 \ge c_\sigma t^{-2\sigma}\|v_t\|_2^2.
\]
The tangential norm on the left is bounded by a constant times the restriction norm of $u_t$ in $H^\sigma(D)$. To see this, localize any $H^\sigma$ extension of $u_t$, use the fixed smooth flattening map, and apply the Fourier inequality $(1+|\xi_y|^2)^\sigma\le(1+|\xi_y|^2+|\xi_r|^2)^\sigma$. Then take the infimum over all extensions. These coordinate and extension operations are bounded in the indicated Sobolev spaces \cite[Chapter~3]{McL00}. Since $\|v_t\|_2\asymp\|u_t\|_2$, we obtain
\[
 \|u_t\|_{H^\sigma(D)}^2\ge c_\sigma t^{-2\sigma}\|u_t\|_2^2.
\]
Suppose, to get a contradiction, that a gain greater than $1/2$ holds. Then the estimate holds for some $\sigma\in(1/2,1)$. The last inequality and \eqref{eq:oscillating-energy} would then imply $t^{-2\sigma}\le Ct^{-1}$ for arbitrarily small $t$, which is impossible. This contradiction proves the lemma.
\end{proof}

\begin{remark}
The bound in Lemma~\ref{lem:upper-half-order} is classical. On pseudoconvex domains it also follows from Catlin's necessary condition \cite{Cat83}, since $\Delta_1(\partial D,p)\ge2$. The gain $1/2$ is attained on strongly pseudoconvex domains \cite{Koh63}. The test forms used above are those of Kohn's classical argument, for which we refer the reader to \cite{FK72}. We include the short proof so that the converse direction of Theorem~\ref{thm:main} and Corollary~\ref{cor:critical} do not depend on \cite{Cat83}.
\end{remark}

We now give a first application of the directional power geometry. It concerns Lipschitz domains with Gromov hyperbolic Kobayashi distance, where the exponent comes from Theorem~\ref{thm:lipschitz}.

\begin{corollary}[Lipschitz Gromov hyperbolic domains]\label{cor:lipschitz-subelliptic}
Under the hypotheses of Theorem~\ref{thm:lipschitz}, suppose that $n\ge2$, and let $\alpha\in(0,1/2)$ be the exponent given there. Then for every $1\le q\le n$ and $b>1/2$ there is $C_b$ such that
\[
 \|E_0u\|_{H^{\alpha,-b}(\R^{2n})}^2\le C_bQ_\Omega(u),\qquad u\in\mathcal V_q(\Omega).
\]
Moreover, $N_q\in\mathfrak S_p$ for every $p>n/\alpha$.
\end{corollary}
\begin{proof}
By Theorem~\ref{thm:lipschitz}, the directional radii satisfy the power bound with the real parameter $m=1/\alpha$. Hence the hypothesis of Theorem~\ref{thm:directional} is satisfied, and that theorem gives both the estimate and the Schatten class membership.
\end{proof}

The next corollary gives the quantitative gains on smoothly bounded domains of finite type. It also shows that the maximal D'Angelo one-type agrees with the maximal order of contact of complex affine lines.

\begin{corollary}[Quantitative gains at finite type]\label{cor:finite-gains}
Let $\Omega\Subset\C^n$, $n\ge2$, be a $\C$-convex domain with $C^\infty$ boundary and finite D'Angelo one-type. Then
\[
 M=\max_{p\in\partial\Omega}\Delta_1(\partial\Omega,p)
   =M_1(\Omega)<\infty.
\]
For every $1\le q\le n$, $b>1/2$, and $0<s<1/M$,
\begin{equation}\label{eq:subelliptic}
 \|E_0u\|_{H^{1/M,-b}(\R^{2n})}^2\le C_bQ_\Omega(u),\qquad
 \|u\|_{H^s(\Omega)}^2\le C_sQ_\Omega(u),\quad u\in\mathcal V_q(\Omega).
\end{equation}
Moreover, $N_q\in\mathfrak S_p$ for every $p>nM$.
\end{corollary}
\begin{proof}
We first show that the maximal order of contact of complex affine lines is finite. Finite D'Angelo one-type implies finite order of contact of every complex affine line with the boundary. Since the boundary and the set of unit line directions are compact, the maximal order of contact $M_1(\Omega)$ is finite. We next derive the power bound for the directional radii. Proposition~\ref{prop:two-depth}, with $k=1$ and the larger depth fixed, yields
\[
 r_\Omega(z;v)\le C\delta_\Omega(z)^{1/M_1(\Omega)}
\]
in the collar. Since the directional radii are bounded, the same inequality holds on the compact interior after enlarging the constant. We now compare the two maxima. Lemma~\ref{lem:full-curve-type} gives $\Delta_1(\partial\Omega,p)\le M_1(\Omega)$ for every $p$. On the other hand, an affine line attaining the maximal contact gives the reverse inequality for the global maximum. Thus $M=M_1(\Omega)$. We emphasize that only this equality of the global maxima is used here. For the pointwise equality, we refer the reader to the references before Lemma~\ref{lem:full-curve-type}. Finally, Theorem~\ref{thm:directional} gives all the stated estimates.
\end{proof}

We are now ready to complete the proof of Theorem~\ref{thm:main}.

\begin{proof}[Proof of Theorem~\ref{thm:main}]
The equivalence of (i) and (ii) is the content of Proposition~\ref{prop:geometric-half}. Suppose that (ii) holds. Then Corollary~\ref{cor:finite-gains} gives a subelliptic estimate of every gain $0<s<1/M$ in every degree $1\le q\le n$. In particular, (iii) holds, and so does the last assertion of the theorem. Suppose now that (iii) holds. Then Theorem~\ref{thm:analytic-necessity} gives (ii).
\end{proof}

\begin{remark}[Condition $R$ and smooth boundary extension]\label{rem:condition-R}
Corollary~\ref{cor:finite-gains} has two classical consequences. Let $D\Subset\C^n$, $n\ge2$, be a $\C$-convex domain with $C^\infty$ boundary and finite D'Angelo one-type, and let $P_D$ be the Bergman projection. By Result~\ref{res:smooth-regularity} and Kohn's formula $P_Df=f-\ddbar^*N_1\ddbar f$ \cite{Koh63}, the domain $D$ satisfies Condition~$R$, that is, $P_D(C^\infty(\overline D))\subset C^\infty(\overline D)$, see \cite{BL80}. By the theorem of Bell and Ligocka \cite{BL80}, every biholomorphism between two such domains extends to a $C^\infty$ diffeomorphism of the closures. On every smoothly bounded pseudoconvex domain of finite type both conclusions follow in the same way from Catlin's theorem \cite{Cat87}, see \cite{BS99}. Here Corollary~\ref{cor:finite-gains} replaces Catlin's theorem.
\end{remark}

\subsection{Critical exponents}
In this subsection, we show that three exponents attached to a smoothly bounded domain of finite type coincide. The first is the directional index $\alpha_{\mathrm{dir}}$ of Definition~\ref{def:dir}. The other two are the normal exponent and the subelliptic exponent, which we now define. Put
\begin{align*}
 \lambda_{\rm nor}(\Omega)&=\sup\{\lambda>0:\text{for some }C>0,\ H_s^{(1)}(p;v)\le Ce^{-\lambda(t-s)}H_t^{(1)}(p;v)\\[-2pt]
 &\hspace{42mm}\text{for all }p,v\text{ and }0\le s\le t\},\\
 \varepsilon_{\rm sub}(\Omega)&=\sup\{\varepsilon>0:\text{a restriction-Sobolev subelliptic estimate}\\[-2pt]
 &\hspace{42mm}\text{of gain }\varepsilon\text{ holds on }(0,1)\text{-forms}\}.
\end{align*}
The normal exponent does not depend on the initial collar depth, provided that this depth is sufficiently small. Changing the depth changes only the constants.

We now present the equality of the three exponents.

\begin{corollary}[Equality of the three critical exponents]\label{cor:critical}
Under the hypotheses of Corollary~\ref{cor:finite-gains},
\begin{equation}\label{eq:critical-exponents}
 \alpha_{\mathrm{dir}}(\Omega)=\lambda_{\rm nor}(\Omega)
 =\varepsilon_{\rm sub}(\Omega)=\frac1M.
\end{equation}
Moreover, the suprema defining $\alpha_{\mathrm{dir}}(\Omega)$ and $\lambda_{\rm nor}(\Omega)$ are attained.
\end{corollary}
\begin{proof}
We first note that $M\ge2$, since complex tangent lines have contact order at least two. We treat the normal exponent first. Theorem~\ref{thm:planes} gives the normal exponent, and it shows that this exponent is optimal. We next treat the directional exponent. The two-depth estimate of the same theorem gives the power bound with the directional exponent $1/M$, so this exponent is admissible. At a line of maximal contact, \eqref{eq:maximal-plane-radius} gives a radius comparable to $r^{1/M}$, and this excludes every larger directional exponent.

It remains to treat the subelliptic exponent. Corollary~\ref{cor:finite-gains} gives every gain below $1/M$. Suppose, to get a contradiction, that a gain $\varepsilon>1/M$ holds in degree one. We consider two cases. If $M>2$, choose $s$ with $1/M<s<\min\{\varepsilon,1/2\}$. Then Theorem~\ref{thm:analytic-necessity} gives a power bound for the directional radii with the exponent just chosen, and this contradicts the radius estimate at a line of maximal contact. If $M=2$, then the gain exceeds one half, and Lemma~\ref{lem:upper-half-order} gives a contradiction. Thus no gain greater than $1/M$ is possible, and the proof is complete.
\end{proof}

\begin{remark}
We emphasize that the equality $\varepsilon_{\rm sub}(\Omega)=1/M$ concerns the supremum of the gains. It does not say that the supremum itself is a gain. The cases in which a subelliptic estimate of gain exactly $1/M$ is known are recalled in Subsection~\ref{sec:critical-intro}.
\end{remark}

\subsection{Other invariant distances}
In this subsection, we note that the Kobayashi distance in Theorem~\ref{thm:main} may be replaced by two other invariant distances. Write $d_{\gamma,\Omega}$ for the integrated infinitesimal Carath\'eodory distance and $d_{B,\Omega}$ for the Bergman distance. On bounded $\C$-convex domains, \cite[Proposition~1 and Theorem~12]{NPZ11}, together with Hahn's inequality $\beta_\Omega\ge\gamma_\Omega$ \cite{Hahn78}, give
\[
 \gamma_\Omega\asymp\kappa_\Omega\asymp\beta_\Omega.
\]
Hence the integrated distances of these three metrics are bi-Lipschitz equivalent. Both $d_{\gamma,\Omega}$ and $d_{B,\Omega}$ are length metrics. They are complete, because they are bi-Lipschitz to $k_\Omega$, and they are locally compact. Hence they are proper and geodesic by the Hopf--Rinow--Cohn-Vossen theorem \cite[Proposition~I.3.7]{BH99}. Consequently, Gromov hyperbolicity of $k_\Omega$ in Theorem~\ref{thm:main} can be replaced by Gromov hyperbolicity of $d_{\gamma,\Omega}$ or of $d_{B,\Omega}$.

\end{document}